\documentclass[12pt,amsymb,fullpage]{amsart}
\usepackage{amssymb,amscd,pstricks}

\newtheorem{theorem}{Theorem}[section]
\newtheorem{defn}[theorem]{Definition}

\newtheorem{lemma}[theorem]{Lemma}

\newtheorem{eple}[theorem]{Example}
\newtheorem{rmk}[theorem]{Remarks}
\newtheorem{dsc}[theorem]{Discussion}
\newtheorem{nota}[theorem]{Notation}

\newsavebox{\indbin}
\savebox{\indbin}{\begin{picture}(0,0)
\newlength{\gnu}
\settowidth{\gnu}{$\smile$} \setlength{\unitlength}{.5\gnu}
\put(-1,-.65){$\smile$} \put(-.25,.1){$|$}
\end{picture}}

\newcommand{\be}{\begin{enumerate}}
\newcommand{\bd}{\begin{defn}}
\newcommand{\bt}{\begin{theorem}}
\newcommand{\bl}{\begin{lemma}}
\newcommand{\ee}{\end{enumerate}}
\newcommand{\ed}{\end{defn}}
\newcommand{\et}{\end{theorem}}
\newcommand{\el}{\end{lemma}}

\begin{document}
\title{Some Arguments for the Wave Equation in Quantum Theory 2}
\author{Tristram de Piro}
\address{Flat 3, Redesdale House, 85 The Park, Cheltenham, GL50 2RP}
\begin{abstract}
We prove that if the frame $S$ is decaying surface non-radiating, in the sense of Definition \ref{strongly}, then if $(\rho,\overline{J})$ is analytic, either $\rho=0$ and $\overline{J}=\overline{0}$, or $S$ is non-radiating, in the sense of \cite{dep1}. In particularly, by the result there, the charge and current satisfy certain wave equations in all the frames $S_{\overline{v}}$ connected to $S$ by a real velocity vector $\overline{v}$, with $|\overline{v}|<c$.

\end{abstract}
\maketitle
\begin{section}{The Surface Radiating Condition}

\begin{defn}
\label{strongly}
Using notation as in the paper \cite{dep1}, we say that $S$ is decaying surface non-radiating, for a given smooth real pair $(\rho,\overline{J})$ satisfying the continuity equation if;\\

$(i)$. $S$ is surface non-radiating with respect to $(\rho,\overline{J})$\\

that is there exist solutions $(\overline{E}_{\overline{v}},\overline{B}_{\overline{v}})$ to Maxwell's equations for the transformed current and charge $(\rho_{\overline{v}},\overline{J}_{\overline{v}})$ in the frames $S_{\overline{v}}$, connected to $S$ by a real velocity vector $\overline{v}$, with $|\overline{v}|<c$, such that $div_{S_{\overline{v}}}(\overline{E}_{\overline{v}}\times \overline{B}_{\overline{v}})=0$.\\

$(ii)$. The solutions $(\overline{E}_{\overline{v}},\overline{B}_{\overline{v}})$ decay at infinity in the frames $S_{\overline{v}}$, that is for coordinates $(x,y,z,t)$ in $S_{\overline{v}}$, for given $t\in\mathcal{R}_{>0}$, we have that $lim_{|\overline{x}|\rightarrow\infty}\overline{E}(\overline{x},t)
=lim_{|\overline{x}|\rightarrow\infty}\overline{B}(\overline{x},t)=0$\\

\end{defn}

\begin{defn}
\label{reflections}
We let $C^{1}(\mathcal{R}^{3},\mathcal{R}_{>0})$ denote the continuously differentiable functions $f$ in the variables $(x,y,z,t)$. We let $x_{0}=t$, $x_{1}=x$, $x_{2}=y$ and $x_{3}=z$. Given $g\in O(3)$, which defines a coordinate transformation $(x',y',z')=g(x,y,z)$, and $f\in C^{1}(\mathcal{R}^{3},\mathcal{R}_{>0})$, we define $f^{g}$ by;\\

$f^{g}(x_{0}',x_{1}',x_{2}',x_{3}')=f(x_{0},x_{1},x_{2},x_{3})$\\

where $x_{0}'=x_{0}$ and $g(x_{1},x_{2},x_{3})=(x_{1}',x_{2}',x_{3}')$. Define;\\

${\partial f^{g}\over \partial x_{0}'}|_{\overline{x}'}={\partial f\over \partial x_{0}}|_{\overline{x}}$\\

${\partial f^{g}\over \partial x_{i}'}|_{\overline{x}'}=(Df)|_{\overline{x}}\centerdot (g^{-1})_{*}\overline{e}_{i}$, for $1\leq i\leq 3$.\\

We observe, using the matrix representations $g_{ij}$ and $(g^{-1})_{ij}$, with $1\leq i,j\leq 3$,  for $g$ and $g^{-1}$ respectively, and the fact that $g^{*}=g^{-1}$, that $(g^{-1})_{ji}=g_{ij}$. It follows;\\

${\partial f^{g}\over \partial x_{i}'}|_{\overline{x}'}=\sum_{j=1}^{3}(g^{-1})_{ji}{\partial f\over \partial x_{j}}|_{\overline{x}}=\sum_{j=1}^{3}g_{ij}{\partial f\over \partial x_{j}}|_{\overline{x}}$ $(*)$\\

Given a vector field $\overline{F}$, with components $(f_{1},f_{2},f_{3})$, we define $\overline{F}^{g}$ with components $(f_{1}',f_{2}',f_{3}')$ by;\\

$f_{i}'=\sum_{j=1}^{3}g_{ij}f_{j}^{g}$, for $1\leq i\leq 3$\\

We adopt the convention that if $\{g_{1},g_{2}\}\subset O(3)$, then;\\

$f^{g_{1}g_{2}}=(f^{g_{2}})^{g_{1}}$, $\overline{F}^{g_{1}g_{2}}=(\overline{F}^{g_{2}})^{g_{1}}$\\

\end{defn}

\begin{lemma}
\label{divcurl}
Given $g\in O(3)$, $f\in C^{1}(\mathcal{R}^{3},\mathcal{R}_{>0})$ and vector fields $\{\overline{F},\overline{H}\}$, we have that;\\

$(i).$ ${\partial \overline{F}^{g}\over \partial t'}=({\partial\overline{F}\over \partial t})^{g}$\\

$(ii).$ $\bigtriangledown'(f^{g})=(\bigtriangledown(f))^{g}$\\

$(iii).$ $\bigtriangledown'\centerdot\overline{F}^{g}=(\bigtriangledown\centerdot\overline{F})^{g}$\\

$(iv).$ $\bigtriangledown'\times \overline{F}^{g}=sign(g)(\bigtriangledown\times \overline{F})^{g}$\\

$(v).$ $(\overline{F}^{g}\times\overline{H}^{g})=sign(g)(\overline{F}\times\overline{H})^{g}$\\

where $sign(g)=det(g)$ and can take values $1$ or $-1$, as $g\in O(3)$.

\end{lemma}

\begin{proof}
For the first part, using components $(f_{1},f_{2},f_{3})$ for $\overline{F}$, we have, for $1\leq i\leq 3$, that;\\

$({\partial \overline{F}^{g}\over \partial t'})_{i}|_{\overline{x}'}={\partial\over \partial t'}(\sum_{j=1}^{3}g_{ij}f_{j}^{g})|_{\overline{x}'}$\\

$=\sum_{j=1}^{3}g_{ij}{\partial f_{j}^{g}\over \partial t'}|_{\overline{x}'}$\\

$=\sum_{j=1}^{3}g_{ij}{\partial f_{j}\over \partial t}|_{\overline{x}}$\\

$=\sum_{j=1}^{3}g_{ij}({\partial f_{j}\over \partial t})^{g}|_{\overline{x}'}$\\

$=({\partial \overline{F}\over \partial t})^{g}_{i}|_{\overline{x}'}$\\

For the second part, we have, using the observation $(*)$ in Definition \ref{reflections} that;\\

$\bigtriangledown'(f^{g})_{i}|_{\overline{x}'}={\partial f^{g}\over \partial x_{i}'}|_{\overline{x}'}$\\

$=\sum_{j=1}^{3}g_{ij}{\partial f\over \partial x_{j}}|_{\overline{x}}$\\

$=\sum_{j=1}^{3}g_{ij}\bigtriangledown(f)_{j}|_{\overline{x}}$\\

$=(\bigtriangledown(f))^{g}_{i}|_{\overline{x}'}$\\

For the third part, we have that;\\

$\bigtriangledown'\centerdot\overline{F}^{g}|_{\overline{x}'}=\sum_{i=1}^{3}{\partial (\overline{F}^{g})_{i}\over \partial x_{i}'}|_{\overline{x}'}$\\

$=\sum_{i=1}^{3}{\partial\over\partial x_{i}'}(\sum_{j=1}^{3}g_{ij}f_{j}^{g})|_{\overline{x}'}$\\

$=\sum_{j=1}^{3}\sum_{i=1}^{3}g_{ij}{\partial f_{j}^{g}\over \partial x_{i}'}|_{\overline{x}'}$\\

$=\sum_{j=1}^{3}\sum_{i=1}^{3}g_{ij}\sum_{k=1}^{3}g_{ik}{\partial f_{j}\over \partial x_{k}}|_{\overline{x}}$\\

$=\sum_{j=1}^{3}\sum_{i,k=1}^{3}(g^{-1})_{ji}g_{ik}{\partial f_{j}\over \partial x_{k}}|_{\overline{x}}$\\

$=\sum_{j=1}^{3}\sum_{k=1}^{3}\delta_{jk}{\partial f_{j}\over \partial x_{k}}|_{\overline{x}}$\\

$=\sum_{j=1}^{3}{\partial f_{j}\over \partial x_{j}}|_{\overline{x}}$\\

$=(\bigtriangledown\centerdot\overline{F})^{g}|_{\overline{x}'}$\\

For the fourth part, we let $\sigma$ be the permutation of $(1,2,3)$, with $\sigma(1)=2$, $\sigma(2)=3$ and $\sigma(3)=1$. Then, for $1\leq i\leq 3$, we have;\\

$(\bigtriangledown'\times\overline{F}^{g})_{i}|_{\overline{x}'}=({\partial (\overline{F}^{g})_{\sigma^{2}(i)}\over \partial x_{\sigma(i)}'}-{\partial (\overline{F}^{g})_{\sigma(i)}\over \partial x_{\sigma^{2}(i)}'})|_{\overline{x}'}$\\

$=({\partial \over \partial x_{\sigma(i)}'}(\sum_{j=1}^{3}g_{\sigma^{2}(i),j}f_{j}^{g})-{\partial \over \partial x_{\sigma^{2}(i)}'}(\sum_{j=1}^{3}g_{\sigma(i),j}f_{j}^{g}))|_{\overline{x}'}$\\

$=(\sum_{j=1}^{3}g_{\sigma^{2}(i),j}(\sum_{k=1}^{3}g_{\sigma(i),k}{\partial f_{j}\over \partial x_{k}})-\sum_{j=1}^{3}g_{\sigma(i),j}(\sum_{k=1}^{3}g_{\sigma^{2}(i),k}{\partial f_{j}\over \partial x_{k}}))|_{\overline{x}}$\\

$=\sum_{j,k=1}^{3}(g_{\sigma^{2}(i),j}g_{\sigma(i),k}-g_{\sigma(i),j}g_{\sigma^{2}(i),k}){\partial f_{j}\over \partial x_{k}}|_{\overline{x}}$\\

$=\sum_{j,k=1,j\neq k}^{3}(g_{\sigma^{2}(i),j}g_{\sigma(i),k}-g_{\sigma(i),j}g_{\sigma^{2}(i),k}){\partial f_{j}\over \partial x_{k}}|_{\overline{x}}$\\

$=\sum_{j,k=1,j\neq k}^{3}\tau(j,k)cof(g)_{i,jk}{\partial f_{j}\over \partial x_{k}}|_{\overline{x}}$\\

$=sign(g)\sum_{j,k=1,j\neq k}^{3}\tau(j,k)g_{i,jk}{\partial f_{j}\over \partial x_{k}}|_{\overline{x}}$\\

$=sign(g)\sum_{l=1}^{3}g_{il}({\partial f_{\sigma^{2}(l)}\over \partial x_{\sigma(l)}}-{\partial f_{\sigma(l)}\over \partial x_{\sigma^{2}(l)}})|_{\overline{x}}$\\

$=sign(g)\sum_{l=1}^{3}g_{il}(\bigtriangledown\times\overline{F})_{l}|_{\overline{x}}$\\

$=sign(g)(\bigtriangledown\times\overline{F})^{g}_{i}|_{\overline{x}'}$\\

where $\tau(1,3)=\tau(2,1)=\tau(3,2)=1$ and $\tau(1,2)=\tau(2,3)=\tau(3,1)=-1$, $jk$ denotes the remaining element in the tuple $(1,2,3)$, $cof(g)_{ij}$ is the representation of the cofactor matrix of $g$, and we have used the fact that $g_{ij}=(g^{-1})_{ji}=sign(g)(cof(g)^{*})_{ji}=sign(g)(cof(g))_{ij}$.\\

The proof of the fifth part is similar to the fourth, replacing the components of $\overline{F}$ with $\overline{H}$ and those of $\bigtriangledown$ with $\overline{F}$, using the fact that $\bigtriangledown'=\bigtriangledown^{g}$.
\end{proof}

\begin{lemma}
\label{maxwells}
Given a tuple $(\rho,\overline{J},\overline{E},\overline{B})$ satisfying Maxwell's equations in the rest frame $S$, and $g\in O(3)$, then the tuple $(\rho^{g},\overline{J}^{g},\overline{E}^{g},sign(g)\overline{B}^{g})$ also satisfies Maxwell's equations, and the tuple $(\rho^{g},\overline{J}^{g})$ satisfies the continuity equation in the rotated or reflected frame $S'$. Moreover;\\

$\bigtriangledown'\centerdot(\overline{E}^{g}\times sign(g)\overline{B}^{g})=(\bigtriangledown\centerdot(\overline{E}\times\overline{B}))^{g}$\\

In particular;\\

$\bigtriangledown'\centerdot(\overline{E}^{g}\times sign(g)\overline{B}^{g})=0$ iff $\bigtriangledown\centerdot(\overline{E}\times\overline{B})=0$\\
\end{lemma}

\begin{proof}
The second claim follows immediately from the first. For the first claim, we check the conditions using Lemma \ref{divcurl}. We have that;\\

$(i).$ $\bigtriangledown'\centerdot\overline{E}^{g}=(\bigtriangledown\centerdot\overline{E})^{g}=({\rho\over \epsilon_{0}})^{g}={\rho^{g}\over \epsilon_{0}}$\\

$(ii).$ $\bigtriangledown'\times \overline{E}^{g}=sign(g)(\bigtriangledown\times \overline{E})^{g}=sign(g)(-{\partial \overline{B}\over \partial t})^{g}=-{\partial (sign(g)\overline{B}^{g})\over \partial t'}$\\

$(iii).$ $\bigtriangledown'\centerdot (sign(g)\overline{B}^{g})=sign(g)(\bigtriangledown\centerdot\overline{B})^{g}=0^{g}=0$\\

$(iv).$ $\bigtriangledown'\times (sign(g)\overline{B}^{g})=sign(g)sign(g)(\bigtriangledown\times\overline{B})^{g}$\\

$=(\mu_{0}\overline{J}+\mu_{0}\epsilon_{0}{\partial\overline{E}\over \partial t})^{g}=\mu_{0}\overline{J}^{g}+\mu_{0}\epsilon_{0}{\partial\overline{E}^{g}\over \partial t'}$\\

For the penultimate claim, using Lemma \ref{divcurl} again, we have that;\\

$\bigtriangledown'\centerdot(\overline{E}^{g}\times sign(g)\overline{B}^{g})
=\bigtriangledown'\centerdot(sign(g)sign(g)(\overline{E}\times\overline{B})^{g})$\\

$=\bigtriangledown'\centerdot(\overline{E}\times\overline{B})^{g}=(\bigtriangledown\centerdot(\overline{E}\times\overline{B}))^{g}$\\

The final claim follows immediately from the penultimate claim, applied to the transformations $g$ and $g^{-1}$.\\
\end{proof}

\begin{lemma}
\label{conjugation}
Let $g\in O(3)$, and let $\overline{v}$ define a boost, with matrices $R_{g}$ and $B_{\overline{v}}$ respectively in the Lorentz group, then;\\

$R_{g}B_{\overline{v}}=B_{g(\overline{v})}R_{g}$\\

Moreover, the representation is unique, in the sense that if;\\

$R_{g}B_{\overline{v}}=R_{h}B_{\overline{w}}$\\

for $\{g,h\}\subset O(3)$, and $\{\overline{v},\overline{w}\}$ velocities defining boosts, then $g=h$ and $\overline{v}=\overline{w}$.\\

\end{lemma}
\begin{proof}
We first prove this as a footnote in the case when $\overline{v}=v\overline{e}_{1}$, (\footnote{$B_{v\overline{e}_{1}}$ is given in coordinates by;\\

$(B_{v\overline{e}_{1}})_{00}=(B_{\overline{v}})_{11}=\gamma(v)$\\

$(B_{v\overline{e}_{1}})_{22}=(B_{\overline{v}})_{33}=1$\\

$(B_{v\overline{e}_{1}})_{10}=(B_{\overline{v}})_{01}=-{v\gamma(v)\over c}$\\

$(B_{v\overline{e}_{1}})_{ij}=0$ otherwise, $0\leq i\leq j\leq 3$\\

and $R_{g}$ is given in coordinates by;\\

$(R_{g})_{00}=1$\\

$(R_{g})_{ij}=g_{ij}$, $1\leq i\leq j\leq 3$\\

$(R_{g})_{ij}=0$ otherwise, $0\leq i\leq j\leq 3$\\

where;\\

$\sum_{i=1}^{3}g_{ij}^{2}=1$, for $1\leq j\leq 3$\\

$\sum_{i=1}^{3}g_{ij}g_{ik}=0$, for $1\leq j,k\leq 3$, $j\neq k$ $(*)$\\

so that;\\

$(R_{g}B_{v\overline{e}_{1}})_{00}=\gamma$\\

$(R_{g}B_{v\overline{e}_{1}})_{01}=-{v\gamma\over c}$\\

$(R_{g}B_{v\overline{e}_{1}})_{02}=(R_{g}B_{\overline{v}})_{03}=0$\\

$(R_{g}B_{v\overline{e}_{1}})_{i0}=-{v\gamma g_{i1}\over c}$, for $1\leq i\leq 3$\\

$(R_{g}B_{v\overline{e}_{1}})_{i1}=\gamma g_{i1}$, for $1\leq i\leq 3$\\

$(R_{g}B_{v\overline{e}_{1}})_{ij}=g_{ij}$, for $1\leq i\leq 3$, $2\leq j\leq 3$\\

Using the formula in \cite{L} for a general boost with velocity $\overline{w}$, using $(*)$, we can compute $B_{g(v\overline{e}_{1})}^{-1}=B_{-vg(\overline{e}_{1})}$, to obtain;\\

$(B_{-vg(\overline{e}_{1})})_{00}=\gamma$\\

$(B_{-vg(\overline{e}_{1})})_{0i}={\gamma v\over c}g_{i1}\sum_{k=1}^{3}g_{k1}^{2}={\gamma v\over c}g_{i1}$, $1\leq i\leq 3$\\

$(B_{-vg(\overline{e}_{1})})_{i0}={\gamma v\over c}g_{i1}$, $1\leq i\leq 3$\\

$(B_{-vg(\overline{e}_{1})})_{ii}=(\gamma-1)g_{i1}^{2}+1$, $1\leq i\leq 3$\\

$(B_{-vg(\overline{e}_{1})})_{21}=(B_{-vg(\overline{e}_{1})})_{12}=(\gamma-1)g_{11}g_{21}$\\

$(B_{-vg(\overline{e}_{1})})_{31}=(B_{-vg(\overline{e}_{1})})_{13}=(\gamma-1)g_{11}g_{31}$\\

$(B_{-vg(\overline{e}_{1})})_{32}=(B_{-vg(\overline{e}_{1})})_{23}=(\gamma-1)g_{21}g_{31}$\\

Finally, using $(*)$ and the identity $\gamma^{2}(1-{v^{2}\over c^{2}})=1$, we compute $B_{-vg(\overline{e}_{1})}R_{g}B_{v\overline{e}_{1}}$ in coordinates, to obtain;\\

$(B_{-vg(\overline{e}_{1})}R_{g}B_{v\overline{e}_{1}})_{00}=\gamma^{2}-{v^{2}\gamma^{2}\over c^{2}}(\sum_{k=1}^{3}g_{k1}^{2})=1$\\

$(B_{-vg(\overline{e}_{1})}R_{g}B_{v\overline{e}_{1}})_{01}=-{v\gamma^{2}\over c}+{v\gamma^{2}\over c}(\sum_{k=1}^{3}g_{k1}^{2})=0$\\

$(B_{-vg(\overline{e}_{1})}R_{g}B_{v\overline{e}_{1}})_{0j}={\gamma v\over c}(\sum_{k=1}^{3}g_{k1}g_{kj})=0$ $(2\leq j\leq 3)$\\

$(B_{-vg(\overline{e}_{1})}R_{g}B_{v\overline{e}_{1}})_{i0}=g_{i1}({\gamma^{2} v\over c}-{\gamma v\over c}-{(\gamma-1)\gamma v\over c})=0$ $(1\leq i\leq 3)$\\

$(B_{-vg(\overline{e}_{1})}R_{g}B_{v\overline{e}_{1}})_{i1}=g_{i1}(\gamma(\gamma-1)+\gamma-{\gamma^{2} v^{2}\over c^{2}})=g_{i1}$, $(1\leq i\leq 3)$\\

$(B_{-vg(\overline{e}_{1})}R_{g}B_{v\overline{e}_{1}})_{ij}=(\gamma-1)g_{i1}(\sum_{k=1}^{3}g_{k1}g_{kj})+g_{ij}=g_{ij}$\\

$(1\leq i\leq 3, 2\leq j\leq 3)$\\

as required}). For the general case, let $\overline{v}$ be an arbitrary velocity, and choose $g\in SO(3)$ with $\overline{v}=g(v\overline{e}_{1})$. By the result proved in the footnote, we have that $B_{\overline{v}}R_{g}=R_{g}B_{v\overline{e}_{1}}$, $(**)$. Now let $h\in O(3)$, then, using $(**)$;\\

$R_{h}B_{\overline{v}}=R_{h}R_{g}B_{v\overline{e}_{1}}R_{g}^{-1}$ $(***)$\\

and the claim that $R_{h}B_{\overline{v}}=B_{h(\overline{v})}R_{h}$ is equivalent, by $(***)$, to;\\

$R_{h}R_{g}B_{v\overline{e}_{1}}R_{g}^{-1}=B_{h(\overline{v})}R_{h}$ or $R_{g}B_{v\overline{e}_{1}}R_{g}^{-1}=R_{h}^{-1}B_{h(\overline{v})}R_{h}$ $(****)$\\

We have that $h(\overline{v})=hg(v\overline{e}_{1})$ and $hg\in O(3)$, so it is sufficient, from $(****)$ to prove that;\\

$R_{g}B_{v\overline{e}_{1}}R_{g}^{-1}=R_{h}^{-1}B_{hg(v\overline{e}_{1})}R_{h}$ or $B_{v\overline{e}_{1}}=R_{g}^{-1}R_{h}^{-1}B_{hg(v\overline{e}_{1})}R_{h}R_{g}=R_{hg}^{-1}B_{hg(v\overline{e}_{1})}R_{hg}$\\

$(*****)$\\

Clearly, the claim $(*****)$ or equivalently, $R_{hg}B_{v\overline{e}_{1}}=B_{hg(v\overline{e}_{1})}R_{hg}$ follows from the proof in the footnote, as required. The second claim is noted in \cite{U2}, and is straightforward to prove. We have that;\\

$R_{h}^{-1}R_{g}=R_{h^{-1}g}=B_{\overline{w}}B_{\overline{v}}^{-1}=B_{\overline{w}}B_{-\overline{v}}$\\

The following formula is given in \cite{U2} for the boost matrix $B_{\overline{w}}$;\\

$B_{\overline{w}}=I+{\gamma_{w}b_{\overline{w}}\over c}+{\gamma_{w}^{2}b_{\overline{w}}^{2}\over c^{2}(\gamma_{w}+1)}$\\

where, (\footnote{Ungar's definition of $b_{\overline{w}}$ differs by a minus sign, as he relates unprimed to primed coordinates in the definition of the boost matrix, which seems to go slightly against the usual convention. We have also changed his formula slightly for the case when $x_{0}=tc$, rather than $x_{0}=t$. });\\

$(b_{\overline{w}})_{0j}=(b_{\overline{w}})_{j0}=-w_{j}$ for $1\leq j\leq 3$\\

$(b_{\overline{w}})_{ij}=0$ otherwise $(0\leq i,j\leq 3)$\\

It follows that;\\

$B_{\overline{w}}B_{-\overline{v}}=(I+{\gamma_{w}b_{\overline{w}}\over c}+{\gamma_{w}^{2}b_{\overline{w}}^{2}\over c^{2}(\gamma_{w}+1)})(I-{\gamma_{v}b_{\overline{v}}\over c}+{\gamma_{v}^{2}b_{\overline{v}}^{2}\over c^{2}(\gamma_{v}+1)})$\\

$=I-{\gamma_{v}b_{\overline{v}}\over c}+{\gamma_{w}b_{\overline{w}}\over c}+{\gamma_{v}^{2}b_{\overline{v}}^{2}\over c^{2}(\gamma_{v}+1)}+{\gamma_{w}^{2}b_{\overline{w}}^{2}\over c^{2}(\gamma_{w}+1)}-{\gamma_{v}\gamma_{w}b_{\overline{v}}b_{\overline{w}}\over c^{2}}+{\gamma_{v}^{2}\gamma_{w}b_{\overline{v}}^{2}b_{\overline{w}}\over c^{3}(\gamma_{v}+1)}-{\gamma_{w}^{2}\gamma_{v}b_{\overline{w}}^{2}b_{\overline{v}}\over c^{3}(\gamma_{w}+1)}$\\

$+{\gamma_{w}^{2}\gamma_{v}^{2}b_{\overline{w}}^{2}b_{\overline{v}}^{2}\over c^{4}(\gamma_{w}+1)(\gamma_{v}+1)}$\\

Using the fact that we must have;\\

$(R_{h^{-1}g})_{00}=(B_{\overline{w}}B_{-\overline{v}})_{00}=1$\\

$(R_{h^{-1}g})_{0j}=(B_{\overline{w}}B_{-\overline{v}})_{0j}=0$, $(1\leq j\leq 3)$\\

we obtain the equations;\\

${\gamma_{v}^{2}v^{2}\over c^{2}(\gamma_{v}+1)}+{\gamma_{w}^{2}v^{2}\over c^{2}(\gamma_{w}+1)}-{\gamma_{v}\gamma_{w}\overline{v}\centerdot\overline{w}\over c^{2}}+{\gamma_{v}^{2}\gamma_{w}^{2}v^{2}w^{2}\over c^{4}(\gamma_{v}+1)(\gamma_{w}+1)}=0$ $(\dag)$\\

${\gamma_{w}w_{j}\over c}-{\gamma_{v}v_{j}\over c}=0$ $(1\leq j\leq 3)$ $(\dag\dag)$\\

It follows from $(\dag\dag)$ that;\\

$v^{2}={\gamma_{w}^{2}w^{2}\over \gamma_{v}^{2}}$ and $\overline{v}\centerdot\overline{w}={\gamma_{w} w^{2}\over \gamma_{v}}$\\

and, substituting into $(\dag)$, using the relation ${\gamma_{w}^{2}w^{2}\over c^{2}}={\gamma_{w}^{2}-1}$, we obtain that $\gamma_{w}=\gamma_{v}$, so that, from $(\dag\dag)$, $\overline{w}=\overline{v}$ as required.

\end{proof}

\begin{defn}
\label{rotationsreflections}
Let $g\in O(3)$ be a rotation or reflection, then, as in Definition \ref{reflections}, if $f\in C^{1}(\mathcal{R}^{3},\mathcal{R}_{>0})$, we let $f^{g}$ be defined by;\\

$f^{g}(t',x',y',z')=f(t,x,y,z)$\\

where $t'=t$ and $(x',y',z')=g(x,y,z)$. For a vector field $\overline{V}$ with components $(v_{1},v_{2},v_{3})$, we let $\overline{V}^{g}$ be defined by $(v_{1}',v_{2}',v_{3}')$ where;\\

$v_{i}'=\sum_{1\leq j\leq 3}g_{ij}v^{g}_{j}$, $(1\leq i\leq 3)$\\

and $(g_{ij})_{1\leq i,j\leq 3}$ is the matrix representation of $g$. For a four vector $\overline{W}$ with components $(w_{0},w_{1},w_{2},w_{3})$, we, similarly, let $\overline{W}^{g}$ be defined by $(w_{0}',w_{1}',w_{2}',w_{3}')$, where;\\

$w_{0}'=w_{0}^{g}$\\

$w_{i}'=\sum_{1\leq j\leq 3}g_{ij}w^{g}_{j}$, $(1\leq i\leq 3)$\\

We introduce rotated frames $S'$ relative to a fixed frame $S$, with coordinates $(t',x',y',z')$ linked to the coordinates $(t,x,y,z)$ in $S$, by the relations;\\

$t'=t$\\

$(x',y',z')=g(x,y,z)$\\

where $g\in O(3)$. We define the Lorentz group as generated by boosts in a given direction $\overline{v}$ and by rotations defined by $g\in SO(3)$. By the augmented Lorentz group, we mean the group generated by boosts and elements $g\in O(3)$. It is shown in \cite{U2} that the composition of boosts is equivalent to a boost followed or preceded by a rotation, the so called Thomas rotation. It follows that any element $L$ of the (augmented) Lorentz group can be written (uniquely) as $L=R_{g}B_{\overline{v}}=B_{g(\overline{v})}R_{g}$, where $g\in SO(3)$ (or $g\in O(3)$).

\end{defn}

\begin{lemma}
\label{potential}
Given potentials $(V,\overline{A})$ in the rest frame $S$, satisfying the relations;\\

$\overline{E}=-(\bigtriangledown(V)-{\partial \overline{A}\over \partial t})$\\

$\overline{B}=\bigtriangledown\times\overline{A}$ $(*)$\\

for electric and magnetic fields $\{\overline{E},\overline{B}\}$, the four vector $({V\over c},\overline{A})$ transforms covariantly with respect to elements of the Lorentz group, augmented by $O(3)$, when the transformation rules are given by;\\

$({V'\over c},\overline{A}')=({V\over c},\overline{A})^{g}$\\

when $g\in O(3)$ is a rotation or a reflection, and by;\\

$({V'\over c},\overline{A}')=(\gamma({V\over c}-{<\overline{v},\overline{A}_{||}>\over c}),\gamma(\overline{A}_{||}-{V\over c^{2}}\overline{v})+\overline{A}_{\perp})$\\

when $\overline{v}$ defines a boost, see \cite{L}. Let the frame $S'$ be defined relative to the base frame $S$ by an element $\tau$ of the Lorentz group, which is either a reflection, rotation or a boost. Let $(\overline{E}',\overline{B}')$ be defined by;\\

$(\overline{E}',\overline{B}')=(\overline{E}^{\tau},sign(\tau)\overline{B}^{\tau})$\\

when $\tau\in O(3)$ is a rotation or a reflection, and by;\\

$(\overline{E}',\overline{B}')=(\overline{E}_{||}+\gamma(\overline{E}_{\perp}+\overline{v}\times\overline{B}),
\overline{B}_{||}+\gamma(\overline{B}_{\perp}-{\overline{v}\over c^{2}}\times\overline{E}))$\\

when $\tau$ defines a boost with velocity $\overline{v}$, see \cite{L}. Then with the transformed vector $({V'\over c},\overline{A}')$, we still have that;\\

$\overline{E}'=-(\bigtriangledown'(V')-{\partial \overline{A'}\over \partial t'})$\\

$\overline{B}'=\bigtriangledown'\times\overline{A'}$ $(**)$\\

There exists a well defined transformation of $(\overline{E},\overline{B})$ for an arbitrary element $\tau$ of the Lorentz group, augmented by orthogonal transformations. Moreover, if $\tau$ is represented by two distinct products of boosts, rotations and reflections, then the transformation of $(\overline{E},\overline{B})$ coincides with the two transformations obtained by iteration in these representations.
\end{lemma}

\begin{proof}
The fact that the $({V\over c},\overline{A})$ transforms covariantly with respect to Lorentz boosts is noted in \cite{L}, alternatively it can be shown by writing the transformation rule in matrix form and comparing it with the corresponding Lorentz matrix, the details are left to the reader. For rotations or reflections, the result is immediate from the definition. To check the second claim in the case of a rotation or reflection $\tau\in O(3)$, we have, using the definitions in the statement of the Lemma, Definition \ref{reflections} and Lemma \ref{divcurl}, that;\\

$-(\bigtriangledown'(V')+{\partial\overline{A}'\over \partial t'})$\\

$=-((\bigtriangledown(V))^{\tau}+({\partial\overline{A}\over \partial t})^{\tau})$\\

$=(-(\bigtriangledown(V)+{\partial\overline{A}\over \partial t}))^{\tau}$\\

$=\overline{E}^{\tau}=\overline{E}'$\\

and, similarly;\\

$\bigtriangledown'\times\overline{A}'$\\

$=sign(\tau)(\bigtriangledown\times \overline{A})^{\tau}$\\

$=sign(\tau)\overline{B}^{\tau}=\overline{B}'$\\

We check the second claim in the case of a boost when $\overline{v}=v\overline{e}_{1}$. We have, using the transformation rules above, the transformation rules for derivatives, see \cite{L}, with components $\{(a_{1},a_{2},a_{3}),(a_{1}',a_{2}',a_{3}')\}$ for $\{\overline{A},\overline{A}'\}$ respectively, that;\\

$(a_{1}',a_{2}',a_{3}')=(\gamma a_{1}-{\gamma vV\over c^{2}},a_{2},a_{3})$\\

$V'=\gamma V-\gamma va_{1}$\\

$\bigtriangledown'=(\gamma({\partial\over\partial x}+{v\over c^{2}}{\partial\over\partial t},{\partial\over\partial y},{\partial\over\partial z}))$\\

${\partial\over\partial t'}=\gamma({\partial\over\partial t}+v{\partial\over\partial x})$\\

Using $(*)$ from the statement of the Lemma, with components $\{(e_{1},e_{2},e_{3}),(b_{1},b_{2},b_{3})\}$ for $\{\overline{E},\overline{B}\}$ respectively, we have;\\

$(e_{1},e_{2},e_{3})=(-{\partial a_{1}\over\partial t}-{\partial V\over\partial x},-{\partial a_{2}\over\partial t}-{\partial V\over\partial y},-{\partial a_{3}\over\partial t}-{\partial V\over\partial z})$\\

$(b_{1},b_{2},b_{3})=({\partial a_{3}\over\partial y}-{\partial a_{2}\over\partial z},{\partial a_{1}\over\partial z}-{\partial a_{3}\over\partial x},{\partial a_{2}\over\partial x}-{\partial a_{1}\over\partial y})$\\

We then compute:\\

$(-{\partial \overline{A}'\over \partial t'}-\bigtriangledown'(V'))_{1}=-\gamma({\partial \over\partial t}+v{\partial \over\partial x})(\gamma a_{1}-{\gamma vV\over c^{2}})-\gamma({\partial \over\partial x}+{v\over c^{2}}{\partial \over\partial t})(\gamma V-\gamma va_{1})$\\

$=-\gamma^{2}{\partial a_{1}\over \partial t}+{\gamma^{2}v\over c^{2}}{\partial V\over \partial t}-\gamma^{2}v{\partial a_{1}\over \partial x}+{\gamma^{2}v^{2}\over c^{2}}{\partial V\over \partial x}-\gamma^{2}{\partial V\over \partial x}-{\gamma^{2}v\over c^{2}}{\partial V\over \partial t}+\gamma^{2}v{\partial a_{1}\over \partial x}+{\gamma^{2}v^{2}\over c^{2}}{\partial a_{1}\over \partial t}$\\

$=-{\partial V\over \partial x}-\gamma^{2}{\partial a_{1}\over \partial t}+{\gamma^{2}v^{2}\over c^{2}}{\partial a_{1}\over \partial t}$\\

$=-{\partial V\over \partial x}-{\partial a_{1}\over \partial t}=e_{1}$\\

$(-{\partial \overline{A}'\over \partial t'}-\bigtriangledown'(V'))_{2}=-\gamma({\partial \over\partial t}+v{\partial \over\partial x})a_{2}-{\partial \over\partial y}(\gamma V-\gamma va_{1})$\\

$=-\gamma{\partial a_{2}\over \partial t}-\gamma v{\partial a_{2}\over \partial x}-\gamma{\partial V\over \partial y}+\gamma v{\partial a_{1}\over \partial y}$\\

$=\gamma e_{2}+\gamma v({\partial a_{1}\over \partial y}-{\partial a_{2}\over \partial x})$\\

$=\gamma e_{2}-\gamma vb_{3}$\\

$(-{\partial \overline{A}'\over \partial t'}-\bigtriangledown'(V'))_{3}=-\gamma({\partial \over\partial t}+v{\partial \over\partial x})a_{3}-{\partial \over\partial z}(\gamma V-\gamma va_{1})$\\

$=-\gamma{\partial a_{3}\over \partial t}-\gamma v{\partial a_{3}\over \partial x}-\gamma{\partial V\over \partial z}+\gamma v{\partial a_{1}\over \partial z}$\\

$=\gamma e_{3}+\gamma v({\partial a_{1}\over \partial z}-{\partial a_{3}\over \partial x})$\\

$=\gamma e_{3}+\gamma vb_{2}$\\

so that;\\

$-{\partial \overline{A}'\over \partial t'}-\bigtriangledown'(V')=(e_{1},\gamma e_{2}-\gamma vb_{3},\gamma e_{3}+\gamma vb_{2})=\overline{E}'$\\

as required. Similarly;\\

$(\bigtriangledown'\times \overline{A}')_{1}={\partial a_{3}\over \partial y}-{\partial a_{2}\over\partial z}=b_{1}$\\

$(\bigtriangledown'\times \overline{A}')_{2}={\partial\over\partial z}(\gamma a_{1}-{\gamma vV\over c^{2}})-\gamma({\partial\over\partial x}+{v\over c^{2}}{\partial\over\partial t})a_{3}=\gamma b_{2}-{\gamma v\over c^{2}}({\partial V\over\partial z}+{\partial a_{3}\over \partial t})$\\

$=\gamma b_{2}+{\gamma v e_{3}\over c^{2}}$\\

$(\bigtriangledown'\times \overline{A}')_{3}=\gamma({\partial\over\partial x}+{v\over c^{2}}{\partial\over\partial t})a_{2}-{\partial\over\partial y}(\gamma a_{1}-{\gamma vV\over c^{2}})=\gamma b_{3}+{\gamma v\over c^{2}}({\partial a_{2}\over\partial t}+{\partial V\over\partial y})$\\

$=\gamma b_{3}-{\gamma ve_{2}\over c^{2}}$\\

so that;\\

$\bigtriangledown'\times \overline{A}'=(b_{1},\gamma b_{2}+{\gamma v e_{3}\over c^{2}},\gamma b_{3}-{\gamma ve_{2}\over c^{2}})=\overline{B}'$\\

as required. In the general case of a boost defined by a velocity vector $\overline{v}$, we can find $g\in SO(3)$, with $\overline{v}=g(v\overline{e}_{1})$. Then, by Lemma \ref{conjugation}, we have that;\\

 $B_{\overline{v}}=R_{g}B_{v\overline{e}_{1}}R_{g}^{-1}=R_{g}B_{v\overline{e}_{1}}R_{g^{-1}}$\\

 By the two cases checked so far, we then have that the property $(**)$, in the statement of the Lemma, holds by iteration. If $\tau$ is an arbitrary element of the Lorentz group, it can be written uniquely as the product of a boost and a rotation or a reflection. By iteration, we can then define a transformation $(\overline{E}^{\tau},\overline{B}^{\tau})$, of $(\overline{E},\overline{B})$. If $\tau$ is represented by two distict products, then by the covariant property of the potential $({V\over c},\overline{A})$, the two transformations, $({V_{1}\over c},\overline{A}_{1})$ and $({V_{2}\over c},\overline{A}_{2})$, obtained by iteration in the representations, coincide. In particularly, we have that $V_{1}=V_{2}$ and $\overline{A}_{1}=\overline{A}_{2}$. Again, iterating the relation $(**)$, we must have that;\\

$\overline{E}_{1}=-(\bigtriangledown'(V_{1})+{\partial \overline{A}_{1}\over \partial t'})=-(\bigtriangledown'(V_{2})+{\partial \overline{A}_{2}\over \partial t'})=\overline{E}_{2}$\\

$\overline{B}_{1}=\bigtriangledown'\times \overline{A}_{1}=\bigtriangledown'\times \overline{A}_{2}=\overline{B}_{2}$\\

where $\bigtriangledown'$ and ${\partial\over \partial t'}$ are differential operators in the new frame $S'$ and $\{\overline{E}_{1},\overline{E}_{2},\overline{B}_{1},\overline{B}_{2}\}$ are again obtained by iteration in the representations.

\end{proof}
\begin{lemma}
\label{twist1}

Let $(\rho,\overline{J},\overline{E},\overline{B})$ be a solution to Maxwell's equations in the rest frame $S$, let $g\in O(3)$, and let $(\rho^{g},\overline{J}^{g},\overline{E}^{g},sign(g)\overline{B}^{g})$ in $S'$ be as given in Lemma \ref{maxwells}. Let $\overline{v}$ be a velocity, with corresponding image $\overline{w}=g(\overline{v})$. Let $S''$ and $S'''$ be the frames defined by $\overline{v}$ and $\overline{w}$, relative to $S$ and $S'$ respectively. Then, if $(\rho'',\overline{J}'',\overline{E}'',\overline{B}'')$ is the solution to Maxwell's equations, corresponding to $(\rho,\overline{J},\overline{E},\overline{B})$ in $S''$ and $(\rho''',\overline{J}''',\overline{E}''',\overline{B}''')$ is the solution to Maxwell's equations, corresponding to $(\rho^{g},\overline{J}^{g},\overline{E}^{g},sign(g)\overline{B}^{g})$, in $S'''$, then;\\

$[\bigtriangledown''\centerdot(\overline{E}''\times\overline{B}'')]^{g}=\bigtriangledown'''\centerdot(\overline{E}'''\times\overline{B}''')$\\

In particular, we have that;\\

$\bigtriangledown''\centerdot(\overline{E}''\times\overline{B}'')=0$ iff $\bigtriangledown'''\centerdot(\overline{E}'''\times\overline{B}''')=0$\\

\end{lemma}

\begin{proof}

The final claim follows immediately from the first. By Lemma \ref{conjugation}, we have that $B_{w}R_{g}=R_{g}B_{v}$, and by Lemma \ref{potential}, we have that $\overline{E}'''=\overline{E}''^{g}$, $\overline{B}'''=sign(g)\overline{B}''^{g}$. Then a straightforward calculation, using Lemma \ref{divcurl} shows that;\\

$\bigtriangledown'''\centerdot(\overline{E}'''\times\overline{B}''')
=\bigtriangledown'''\centerdot(\overline{E}''^{g}\times sign(g)\overline{B}''^{g})$\\

$=\bigtriangledown'''\centerdot(\overline{E}''\times\overline{B}'')^{g}$\\

$=[\bigtriangledown''\centerdot(\overline{E}''\times\overline{B}'')]^{g}$\\

\end{proof}

\begin{lemma}
\label{velocity}
Let the frame $S'$ move with velocity $\overline{u}$ relative to $S$, and let the frame $S''$ move with velocity $\overline{v}$ relative to $S'$, then the velocity of $S''$ computed in $S$ is given by;\\

$\overline{u}*\overline{v}={\overline{u}+\overline{v}\over 1+{\overline{u}\centerdot\overline{v}\over c^{2}}}+{\gamma_{u}\over c^{2}(\gamma_{u}+1)}{\overline{u}\times (\overline{u}\times\overline{v})\over 1+{\overline{u}\centerdot\overline{v}\over c^{2}}}$\\

Moreover, there exists $g\in SO(3)$ defining a rotation $R_{g}$, such that;\\

$B_{\overline{v}}B_{\overline{u}}=R_{g}B_{\overline{u}*\overline{v}}=B_{\overline{v}*\overline{u}}R_{g}$\\

where $R_{g}(\overline{u}*\overline{v})=\overline{v}*\overline{u}$ and the velocity of $S$ computed in $S''$ is $-(\overline{v}*\overline{u})$.

\end{lemma}

\begin{proof}
The first formula is given in \cite{U2}. We use the original formula there for the boost matrix;\\

$B_{\overline{u}}=I+\gamma_{u}b+{\gamma_{u}^{2}\over \gamma_{u}+1}b^{2}$\\

as it is unnecessary to introduce the variable $x_{0}=ct$, but see the footnote in Lemma \ref{conjugation}. The frame $S$ moves with velocity $-\overline{u}$ relative to $S'$ and, using primed coordinates for $S$, and unprimed for $S'$, we have, using the matrix $B_{-\overline{u}}$ to relate the two frames, that;\\

$dt'=\gamma_{u}dt+c^{-2}\gamma_{u}u_{1}dx_{1}+c^{-2}\gamma_{u}u_{2}dx_{2}+c^{-2}\gamma_{u}u_{3}dx_{3}$\\

$dx_{1}'=\gamma_{u}u_{1}dt+(1+{\gamma_{u}^{2}u_{1}^{2}\over c^{2}(\gamma_{u}+1)})dx_{1}+{\gamma_{u}^{2}u_{1}u_{2}\over c^{2}(\gamma_{u}+1)}dx_{2}+{\gamma_{u}^{2}u_{1}u_{3}\over c^{2}(\gamma_{u}+1)}dx_{3}$\\

$dx_{2}'=\gamma_{u}u_{2}dt+{\gamma_{u}^{2}u_{1}u_{2}\over c^{2}(\gamma_{u}+1)}dx_{1}+(1+{\gamma_{u}^{2}u_{2}^{2}\over c^{2}(\gamma_{u}+1)})dx_{2}+{\gamma_{u}^{2}u_{2}u_{3}\over c^{2}(\gamma_{u}+1)}dx_{3}$\\

$dx_{3}'=\gamma_{u}u_{3}dt+{\gamma_{u}^{2}u_{1}u_{3}\over c^{2}(\gamma_{u}+1)}dx_{1}+{\gamma_{u}^{2}u_{2}u_{3}\over c^{2}(\gamma_{u}+1)}dx_{2}+(1+{\gamma_{u}^{2}u_{3}^{2}\over c^{2}(\gamma_{u}+1)})dx_{3}$\\

Using the facts that $v_{i}={dx_{i}\over dt}$, for $1\leq i\leq 3$, ${1\over \gamma_{u}}=1-{u^{2}\gamma_{u}\over c^{2}(\gamma_{u}+1)}$ and the formula $\overline{u}\times (\overline{u}\times\overline{v})=\overline{u}(\overline{u}\centerdot\overline{v})-\overline{v}(\overline{u}\centerdot\overline{u})$\\

we compute;\\

$v_{1}'={dx_{1}'\over dt'}={\gamma_{u}u_{1}dt+(1+{\gamma_{u}^{2}u_{1}^{2}\over c^{2}(\gamma_{u}+1)})dx_{1}+{\gamma_{u}^{2}u_{1}u_{2}\over c^{2}(\gamma_{u}+1)}dx_{2}+{\gamma_{u}^{2}u_{1}u_{3}\over c^{2}(\gamma_{u}+1)}dx_{3}\over \gamma_{u}dt+c^{-2}\gamma_{u}u_{1}dx_{1}+c^{-2}\gamma_{u}u_{2}dx_{2}+c^{-2}\gamma_{u}u_{3}dx_{3}}$\\

$={\gamma_{u}u_{1}+(1+{\gamma_{u}^{2}u_{1}^{2}\over c^{2}(\gamma_{u}+1)})v_{1}+{\gamma_{u}^{2}u_{1}u_{2}\over c^{2}(\gamma_{u}+1)}v_{2}+{\gamma_{u}^{2}u_{1}u_{3}\over c^{2}(\gamma_{u}+1)}v_{3}\over \gamma_{u}+c^{-2}\gamma_{u}u_{1}v_{1}+c^{-2}\gamma_{u}u_{2}v_{2}+c^{-2}\gamma_{u}u_{3}v_{3}}$\\

$={v_{1}+u_{1}(\gamma_{u}+{\gamma_{u}^{2}\overline{u}\centerdot\overline{v}\over c^{2}(\gamma_{u}+1)})\over \gamma_{u}(1+{\overline{u}\centerdot\overline{v}\over c^{2}})}$\\

$={{v_{1}\over \gamma_{u}}+u_{1}(1+{\gamma_{u}\overline{u}\centerdot\overline{v}\over c^{2}(\gamma_{u}+1)})\over 1+{\overline{u}\centerdot\overline{v}\over c^{2}}}$\\

$={(1-{u^{2}\gamma_{u}\over c^{2}(\gamma_{u}+1)})v_{1}\over 1+{\overline{u}\centerdot\overline{v}\over c^{2}}}+{(1+{\gamma_{u}\overline{u}\centerdot\overline{v}\over c^{2}(\gamma_{u}+1)})u_{1}\over 1+{\overline{u}\centerdot\overline{v}\over c^{2}}}$\\

$={u_{1}+v_{1}\over 1+{\overline{u}\centerdot\overline{v}\over c^{2}}}+{\gamma_{u}((\overline{u}\centerdot\overline{v})u_{1}-u^{2}v_{1})\over 1+{\overline{u}\centerdot\overline{v}\over c^{2}}}$\\

$={u_{1}+v_{1}\over 1+{\overline{u}\centerdot\overline{v}\over c^{2}}}+{\gamma_{u}(\overline{u}\times(\overline{u}\times\overline{v}))_{1}\over 1+{\overline{u}\centerdot\overline{v}\over c^{2}}}$\\

and similarly;\\

$v_{i}'={u_{i}+v_{i}\over 1+{\overline{u}\centerdot\overline{v}\over c^{2}}}+{\gamma_{u}(\overline{u}\times(\overline{u}\times\overline{v}))_{i}\over 1+{\overline{u}\centerdot\overline{v}\over c^{2}}}$, for $2\leq i\leq 3$\\

so the result follows. The first of the second set of formulae is also given in \cite{U2}, (\footnote{In fact Ungar claims that $B_{\overline{u}}B_{\overline{v}}=B(\overline{u}*\overline{v})R_{h}$, for some $h\in SO(3)$. Remembering that the boost matrix we use in this paper, for $t$ coordinates, reverses the signs of $\overline{u}$, ${\overline{v}}$ and $\overline{u}*\overline{v}$, we obtain;\\

$B_{\overline{u}}^{-1}B_{\overline{v}}^{-1}=B_{-\overline{u}}B_{-\overline{v}}=B_{-(\overline{u}*\overline{v})}R_{h}=
B_{\overline{u}*\overline{v}}^{-1}R_{h}$\\

so that;\\

$B_{\overline{v}}B_{\overline{u}}=R_{h^{-1}}B_{\overline{u}*\overline{v}}$ $(*)$\\

and we can take $g=h^{-1}$. This formula also holds for the boost matrices with $x_{0}=ct$ coordinates, as letting $A_{c}$ be defined by;\\

$(A_{c})_{00}=c$\\

$(A_{c})_{ii}=1$, for $1\leq i\leq 3$\\

$(A_{c})_{ij}=0$, otherwise, for $0\leq i,j\leq 3$\\

we obtain from $(*)$ that;\\

$(A_{c}B_{\overline{v}}A_{c}^{-1})(A_{c}B_{\overline{u}}A_{c}^{-1})=(A_{c}R_{g}A_{c}^{-1})(A_{c}B_{\overline{u}*\overline{v}}A_{c}^{-1})=R_{g}(A_{c}B_{\overline{u}*\overline{v}}A_{c}^{-1})$\\

and the boost matrices in $x_{0}=ct$ coordinates are given by $\{A_{c}B_{\overline{u}}A_{c}^{-1},A_{c}B_{\overline{v}}A_{c}^{-1},A_{c}B_{\overline{u}*\overline{v}}A_{c}^{-1}\}$\\

The explicit representation of $h$ and, therefore, $g$ is known, and included by Ungar. The formula is given by;\\

$R_{g}=R_{h^{-1}}=I-c_{1}\Omega+c_{2}\Omega^{2}$\\

where;\\

$(\Omega)_{ii}=0$, $1\leq i\leq 3$\\

$(\Omega)_{ij}=(-1)^{i+j}\omega_{ij}$, $1\leq i<j\leq 3$\\

$(\Omega)_{ij}=-(\Omega)_{ji}$, $1\leq j<i\leq 3$\\

$\omega=\overline{u}\times\overline{v}=(\omega_{1},\omega_{2},\omega_{3})$\\

$c_{1}={-\gamma_{u}\gamma_{v}(\gamma_{u}+\gamma_{v}+\gamma_{u*v}+1)\over c^{2}(\gamma_{u}+1)(\gamma_{v}+1)(\gamma_{u*v}+1)}$\\

$c_{2}={\gamma_{u}^{2}\gamma_{v}^{2}\over c^{4}(\gamma_{u}+1)(\gamma_{v}+1)(\gamma_{u*v}+1)}$\\

}). We can deduce the second formula from the first. We have that;\\

$B_{\overline{u}*\overline{v}}B_{-\overline{u}}B_{-\overline{v}}=R_{g^{-1}}$\\

and, as this holds for any velocities $\{\overline{u},\overline{v}\}$, making the substitutions $-\overline{v}$ for $\overline{u}$ and $-\overline{u}$ for $\overline{v}$, and, using the fact that $-(\overline{v}*\overline{u})=-\overline{v}*-\overline{u}$, we can find $h\in SO(3)$ with;\\

$B_{-(\overline{v}*\overline{u})}B_{\overline{v}}B_{\overline{u}}=B_{-\overline{v}*-\overline{u}}B_{\overline{v}}B_{\overline{u}}=R_{h^{-1}}$\\

so that;\\

$B_{\overline{v}}B_{\overline{u}}=B_{\overline{v}*\overline{u}}R_{h^{-1}}$ $(\dag)$\\

By the first part of Lemma \ref{conjugation}, and using the first result, we have that;\\

$B_{\overline{v}}B_{\overline{u}}=R_{g}B_{\overline{u}*\overline{v}}=B_{R_{g}(\overline{u}*\overline{v})}R_{g}$\\

and by the uniqueness part of Lemma \ref{conjugation} and $(\dag)$, we have that;\\

$R_{h^{-1}}=R_{g}$ and $\overline{v}*\overline{u}=R_{g}(\overline{u}*\overline{v})$\\

It then follows from $(\dag)$ that;\\

$B_{\overline{v}}B_{\overline{u}}=B_{\overline{v}*\overline{u}}R_{g}$\\

as well. For the final claim, observe that $S'$ moves with velocity $-\overline{v}$ relative to $S''$ and $S$ moves with velocity $-\overline{u}$ relative to $S'$. Using the first claim, we have that the velocity of $S$ computed in the frame $S''$ is $-\overline{v}*-\overline{u}=-(\overline{v}*\overline{u})$ as required.

\end{proof}

\begin{rmk}
\label{paradox}
The presence of the Thomas rotation resolves the so called Mocanu paradox, explained in \cite{U1}, that the relative velocities of two frames $S$ and $S''$, connected by two successive boosts $\{\overline{u},\overline{v}\}$, do not differ by a minus sign, when computed in $S$ and $S''$ respectively. An interesting perspective on the relativistic effects of rotations is given in \cite{F}.

\end{rmk}

\begin{lemma}
\label{existence}
Given a scalar $v\in\mathcal{R}$, with $|v|<c$, and a velocity $\overline{u}$, with $u<c$, there exists a unique velocity $\overline{w}$, such that;\\

$B_{\overline{w}}B_{\overline{u}}=R_{g}B_{v\overline{e}_{1}}$\\

where $g\in SO(3)$. Moreover, if $\overline{u}$ is fixed and $v\rightarrow\infty$, when $u_{1}\neq 0$ and $\overline{u}\neq u_{1}\overline{e}_{1}$, we have that;\\

$\overline{w}\rightarrow iv(0,u_{2},u_{3})$\\

When, $u_{1}\neq 0$ and $\overline{u}=u_{1}\overline{e}_{1}$, we have that;\\

$\overline{w}\rightarrow -{c^{2}\overline{e}_{1}\over u_{1}}$\\

When $u_{1}=0$\\

$\overline{w}\rightarrow v(1-{iu_{1}\over c},-{iu_{2}\over c},-{iu_{3}\over c})$\\

\end{lemma}

\begin{proof}
By Lemma \ref{velocity}, we have that;\\

$B_{\overline{u}}B_{-v\overline{e}_{1}}=R_{g}B_{-v\overline{e}_{1}*\overline{u}}$\\

where $g\in SO(3)$, and, rearranging;\\

$B_{v\overline{e}_{1}*-\overline{u}}B_{\overline{u}}=B_{-(-v\overline{e}_{1}*\overline{u})}B_{\overline{u}}=B_{-v\overline{e}_{1}*\overline{u}}^{-1}B_{\overline{u}}=
R_{g}B_{-v\overline{e}_{1}}^{-1}=R_{g}B_{v\overline{e}_{1}}$\\

so it is sufficient to take $\overline{w}=v\overline{e}_{1}*-\overline{u}$. The uniqueness claim is clear by the second part of Lemma \ref{conjugation}. By the formula in Lemma \ref{velocity}, we have that;\\

$\overline{w}={v\overline{e}_{1}-\overline{u}\over 1-{v{\overline{e}}_{1}\centerdot\overline{u}\over c^{2}}}-{\gamma_{v}(v\overline{e}_{1}\times(v\overline{e}_{1}\times \overline{u}))\over c^{2}(\gamma_{v}+1)({1-v{\overline{e}}_{1}\centerdot\overline{u}\over c^{2}})}$\\

$={v\overline{e}_{1}-\overline{u}\over 1-{u_{1}v\over c^{2}}}-{u_{1}v^{2}\overline{e}_{1}-v^{2}\overline{u}\over c^{2}(1-{u_{1}v\over c^{2}})(1+\sqrt{1-{v^{2}\over c^{2}}})}$\\

For the first case, taking the limit as $v\rightarrow\infty$, keeping $\overline{u}$ fixed, we obtain;\\

$\overline{w}_{\infty}={\overline{e}_{1}\over -{u_{1}\over c^{2}}}-{vu_{1}\over {-iu_{1}\over c}}\overline{e}_{1}+{v\over {-iu_{1}\over c}}\overline{u}$\\

$=-icv\overline{e}_{1}+icv\overline{e}_{1}+icv(0,{u_{2}\over u_{1}},{u_{3}\over u_{1}})$\\

$=iv(0,u_{2},u_{3})$\\

In the second case, we obtain the finite limit;\\

$\overline{w}_{\infty}={\overline{e}_{1}\over -{u_{1}\over c^{2}}}-icv\overline{e}_{1}+icv\overline{e}_{1}=-{c^{2}\overline{e}_{1}\over u_{1}}$\\

and, in the final case, we obtain;\\

$\overline{w}_{\infty}=lim_{v\rightarrow\infty}(v\overline{e}_{1}-\overline{u}+{v^{2}\overline{u}\over c^{2}(1+\sqrt{1-{v^{2}\over c^{2}}})})$\\

$=v\overline{e}_{1}+{v\overline{u}\over c^{2}({i\over c})}$\\

$=v(1-{iu_{1}\over c},-{iu_{2}\over c},-{iu_{3}\over c})$\\

as required.

\end{proof}

\begin{defn}
\label{extension}
We extend the definition of boost matrices in Lemma \ref{conjugation} to include complex vectors $\overline{v}$, with $\overline{v}^{2}\neq c^{2}$, where $\overline{v}^{2}=v_{1}^{2}+v_{2}^{2}+v_{3}^{2}$, and $\gamma_{\overline{v}}={1\over \sqrt{1-{\overline{v}^{2}\over c^{2}}}}$, where we can take either square root, provided we do so consistently in the definition. We denote the two complex boost matrices obtained from a complex vector $\overline{v}$ by $B_{\overline{v}}^{1}$ and $B_{\overline{v}}^{2}$. The fact that $(B_{\overline{v}}^{i})^{-1}=(B_{-\overline{v}}^{i})$, for $1\leq i\leq 2$, follows easily from the real case and the fact that it holds generically, noting that $(-\overline{v})^{2}=\overline{v}^{2}$, so we can take compatible square roots. Similarly, we extend the definition of $\overline{u}*\overline{v}$, to include complex vectors $\overline{u}$ and $\overline{v}$, with $\overline{u}^{2}\notin \{0,c^{2}\}$, $\overline{v}^{2}\neq c^{2}$, and $1+{\overline{u}\centerdot\overline{v}\over c^{2}}\neq 0$, taking $\overline{u}\centerdot\overline{v}$ to be $u_{1}v_{1}+u_{2}v_{2}+u_{3}v_{3}$, and noting there are two possibilities, $(\overline{u}*\overline{v})^{1}$ and $(\overline{u}*\overline{v})^{2}$, depending on the choice of square root in $\gamma_{\overline{u}}$. We extend the orthogonal group $O(3)$, to consist of complex transformations $R_{g}$ in the spatial coordinates with $R_{g}R_{g}^{t}=R_{g}^{t}R_{g}=I$, where $t$ denotes transpose. We denote this group by $G(3)$, noting that it is a group, as $R_{g}R_{h}(R_{g}R_{h})^{t}=R_{g}R_{h}R_{h}^{t}R_{g}^{t}=R_{g}R_{g}^{t}=I$, and, similarly, $(R_{g}R_{h})^{t}R_{g}R_{h}=I$. We denote by $SG(3)$ the subgroup of $G(3)$ consisting of complex transformations with $det(R_{g})=1$. We let $C^{an}(\mathcal{C}^{4})$ denote the set of analytic functions in the variables $(x,y,z,t)$, and, given $f\in C^{an}(\mathcal{C}^{4}), g\in G(3)$, define $f^{g}$ by extension from the real case. Using analytic derivatives, we similarly define the transformation $\overline{F}^{g}$ of a complex vector field. We extend the definition of the operator $\bigtriangledown$ using analytic derivatives. Let the frame $S_{\overline{v}}$ be connected to the base frame $S$ by one of the boost matrices $B_{\overline{v}}^{i}$, $1\leq i\leq 2$, for the complex vector $\overline{v}$, with $\overline{v}^{2}\neq c^{2}$. Given a real valued tuple $(\rho,\overline{J})$, and a complex valued tuple $(\overline{E},\overline{B})$, such that $(\rho,\overline{J},\overline{E},\overline{B})$ satisfy Maxwell's equations in $S$, we extend the transformation rules, given in \cite{L}, for $(\rho,\overline{J})$ and $(\overline{E},\overline{B})$. We use the above definitions of $\overline{u}\centerdot\overline{v}$ and $\gamma_{\overline{v}}$, for complex vectors $\{\overline{u},\overline{v}\}$, with $\overline{v}^{2}\neq c^{2}$, and use the choice of square root, determined by $i$. We can link the coordinates with the boost matrix $B_{\overline{v}}^{i}$, and we are only interested in the coordinates of $\mathcal{C}^{4}$ determined by the image of $\mathcal{R}^{3}\times \mathcal{R}_{\geq 0}$, under the boost matrix. Similarly, we define the transformation of derivatives from $S_{\overline{v}}$ to $S$, using the boost matrix $(B_{\overline{v}}^{i})^{-1}$, in coordinates $(x_{0}',x_{1}',x_{2}',x_{3}')$ for the frame $S_{\overline{v}}$, by the rule ${\partial \over \partial x_{i}'}=\sum_{j=0}^{4}((B_{\overline{v}}^{i})^{-1})_{ji}{\partial\over\partial x_{j}}$.
\end{defn}

\begin{lemma}
\label{preserved}
With the notation of Definition \ref{extension}, we still have that the results of Lemmas \ref{divcurl}, \ref{conjugation} and \ref{velocity} hold. That is, the results $(i)-(v)$ of Lemma \ref{divcurl} hold, with analytic derivatives, the extension of the $\bigtriangledown$ operator, and taking $g\in G(3)$. Moreover, noting $g(\overline{v})^{2}=\overline{v}^{2}$, so we can take compatible square roots;\\

$R_{g}B_{\overline{v}}^{i}=B_{g(\overline{v})}^{i}R_{g}$, for $1\leq i\leq 2$\\

when $g\in G(3)$, $\overline{v}^{2}\neq c^{2}$, and we still have uniqueness of representation, that is if;\\

$R_{g}B_{\overline{v}}^{i_{1}}=R_{h}B_{\overline{w}}^{i_{2}}$\\

for $\{g,h\}\subset G(3)$, and $\{\overline{v},\overline{w}\}$, with $\overline{v}^{2}\neq c^{2}$ and $\overline{w}^{2}\neq c^{2}$, then $g=h$, $\overline{v}=\overline{w}$ and $i_{1}=i_{2}$ in the sense of taking compatible square roots in $\gamma_{\overline{v}}$ or $\gamma_{\overline{w}}$. Finally, for $\{\overline{u},\overline{v}\}$, with $\{\overline{u}^{2},\overline{v}^{2}\}\cap \{0,c^{2}\}=\emptyset$, $1+{\overline{u}\centerdot\overline{v}\over c^{2}}\neq 0$ and $(1+{\overline{u}\centerdot\overline{v}\over c^{2}})^{2}-(1-{\overline{u}^{2}\over c^{2}})(1-{\overline{v}^{2}\over c^{2}})\neq 0$, $(\dag)$,  taking;\\

$\gamma_{\overline{u}*\overline{v}}=\gamma_{\overline{v}*\overline{u}}=\gamma_{\overline{u}}\gamma_{\overline{v}}(1+{\overline{u}\centerdot\overline{v}\over c^{2}})$ $(*)$\\

we have that there exists $g_{i_{1}i_{2}}\in SG(3)$, with;\\

$B_{\overline{v}}^{i_{2}}B_{\overline{u}}^{i_{1}}=R_{g_{i_{1}i_{2}}}B_{(\overline{u}*\overline{v})^{i_{3}}}^{i_{4}}
=B_{(\overline{v}*\overline{u})^{i_{5}}}^{i_{4}}R_{g_{i_{1}i_{2}}}$\\

and;\\

$R_{g_{i_{1}i_{2}}}((\overline{u}*\overline{v})^{i_{3}})=(\overline{v}*\overline{u})^{i_{5}}$\\

where $i_{3}$ is determined by the choice of $i_{1}$, $i_{5}$ is determined by the choice of $i_{2}$, $i_{4}$ is determined by the formula $(*)$, and $g_{i_{1}i_{2}}$ is determined by the choices of $i_{1}$ and $i_{2}$.\\

Let the frames $S_{\overline{v}}$ and $S$ as in the above definition, be connected by the boost matrices $B_{\overline{v}}^{i}$ and $B_{-\overline{v}}^{i}$. Then, if $(\rho_{\overline{v}}^{i},\overline{J}_{\overline{v}}^{i})$ and $(\overline{E}_{\overline{v}}^{i},\overline{B}_{\overline{v}}^{i})$ are the transformed quantities, we have that;\\

$\rho=(\rho_{\overline{v}}^{i})_{-\overline{v}}^{i}$, $\overline{J}=(\overline{J}_{\overline{v}}^{i})_{-\overline{v}}^{i}$, $\overline{E}=(\overline{E}_{\overline{v}}^{i})_{-\overline{v}}^{i}$, $\overline{B}=(\overline{B}_{\overline{v}}^{i})_{-\overline{v}}^{i}$\\

Let $\{\overline{u},\overline{v}\}$ satisfy the conditions $(\dag)$ above, with frames $\{S,S',S'',S''',S''''\}$, such that $S'$ is connected to $S$ by $B_{\overline{u}}^{i_{1}}$, $S''$ is connected to $S'$ by $B_{\overline{v}}^{i_{2}}$, $S'''$ is connected to $S$ by $B_{(\overline{u}*\overline{v})_{i_{3}}}^{i_{4}}$, $S''''$ is connected to $S$ by $R_{g_{i_{1}i_{2}}}$, $S''$ is connected to $S'''$ by $R_{g_{i_{1}i_{2}}}$, $S''$ is connected to $S''''$ by $B_{(\overline{v}*\overline{u})_{i_{5}}}^{i_{4}}$, for the appropriate choice of $\{i_{1},i_{2},i_{3},i_{4},i_{5}\}$, then we have, for the transformations, corresponding to  the pairs $(\rho,\overline{J})$ and  $(\overline{E},\overline{B})$, that;\\

$((\rho_{\overline{u}}^{i_{1}})_{\overline{v}}^{i_{2}},(\overline{J}_{\overline{u}}^{i_{1}})_{\overline{v}}^{i_{2}})
=((\rho_{(\overline{u}*\overline{v})_{i_{3}}}^{i_{4}})_{g_{{i_{1}}i_{2}}},(\overline{J}_{(\overline{u}*\overline{v})_{i_{3}}}^{i_{4}})_{g_{i_{1}{i_{2}}}})
=((\rho_{g_{{i_{1}}i_{2}}})_{(\overline{v}*\overline{u})_{i_{5}}}^{i_{4}},(\overline{J}_{g_{{i_{1}}i_{2}}})_{(\overline{v}*\overline{u})_{i_{5}}}^{i_{4}})$\\

$((\overline{E}_{\overline{u}}^{i_{1}})_{\overline{v}}^{i_{2}},(\overline{B}_{\overline{u}}^{i_{1}})_{\overline{v}}^{i_{2}})
=((\overline{E}_{(\overline{u}*\overline{v})_{i_{3}}}^{i_{4}})_{g_{{i_{1}}i_{2}}},(\overline{B}_{(\overline{u}*\overline{v})_{i_{3}}}^{i_{4}})_{g_{i_{1}{i_{2}}}})
=((\overline{E}_{g_{{i_{1}}i_{2}}})_{(\overline{v}*\overline{u})_{i_{5}}}^{i_{4}},(\overline{B}_{g_{{i_{1}}i_{2}}})_{(\overline{v}*\overline{u})_{i_{5}}}^{i_{4}})$\\

Moreover, the transformation of derivatives from $S''$ to $S$ is independent of the path taken, in the sense that for coordinates $\{x'',y'',z'',t''\}$ in $S''$;\\

$(B_{(\overline{u}*\overline{v})_{i_{3}}}^{i_{4}})^{-1}(R_{g_{i_{1}i_{2}}})^{-1}({\partial\over \partial x_{i}''})=(B_{\overline{u}}^{i_{1}})^{-1}(B_{\overline{v}}^{i_{2}})^{-1}({\partial\over \partial x_{i}''})=(R_{g_{i_{1}i_{2}}})^{-1}(B_{(\overline{v}*\overline{u})_{i_{5}}}^{i_{4}})^{-1}({\partial\over \partial x_{i}''})$\\

\end{lemma}
\begin{proof}
For the first part, in $(i)$, we use the complex linearity of the analytic derivatives. For $(ii)$, we still have the property $(g^{-1})_{ji}=g_{ij}$ for the complex entries of $g\in G(3)$, the rest follows using the chain rule for analytic functions. For $(iii)$, we again use the properties mentioned to prove $(i),(ii)$. $(iv)$ is similar, using the definition of the inverse of a complex matrix $A^{-1}={1\over det(A)}(cof(A))^{t}$. $(v)$ is similar to $(iv)$. For the second part, the fact that $g(\overline{v})^{2}=\overline{v}^{2}$ for $g\in G(3)$ is a simple calculation using the definitions. We first prove the footnote relations;\\

$R_{g}B_{v\overline{e}_{1}}^{i}=B_{g(v\overline{e}_{1})}^{i}R_{g}$ $(1\leq i\leq 2)$\\

for any $v\in\mathcal{C}$ with $v^{2}\neq c^{2}$. These follow, using the properties of $g\in G(3)$ and the fact that the identity $\gamma_{v\overline{e}_{1}}^{2}(1-{v^{2}\over c^{2}})=1$ holds for either choice of square root in $\gamma_{v\overline{e}_{1}}$. The rest of the proof follows similarly to Lemma \ref{conjugation}. We note that if $\overline{v}^{2}\neq c^{2}$, then for a choice $v\in\mathcal{C}$ of the square root of $\overline{v}^{2}$, we have that $({\overline{v}\over v})^{2}=1$. We can then find $g\in SG(3)$ with $g(\overline{e}_{1})={\overline{v}\over v}$, by taking a matrix with first column ${\overline{v}\over v}$, then choosing $\overline{h}$, with ${\overline{v}\over v}\centerdot \overline{h}=0$, and $h_{1}^{2}+h_{2}^{2}+h_{3}^{2}=1$ as the second column, and finally taking the complex cross product ${\overline{v}\over v}\times \overline{h}$ as the third column. Interchanging the second and third columns if necessary, we can ensure $det(g)=1$. Then $g(v\overline{e}_{1})=\overline{v}$ by linearity and $v^{2}\neq c^{2}$. For the uniqueness claim, we can compute $B_{\overline{w}}^{i_{2}}B_{-\overline{v}}^{i_{2}}$ and using the definition of $G(3)$, obtain the same relations, including;\\

$\gamma_{\overline{w}}w_{j}=\gamma_{\overline{v}}v_{j}$, for $(1\leq j\leq 3)$ $(*)$\\

for the appropriate choices of a square root in $\gamma_{\overline{v}}$ and $\gamma_{\overline{w}}$. As in the proof, we conclude that $\gamma_{\overline{v}}=\gamma_{\overline{w}}$, which from $(*)$ implies that $\overline{v}=\overline{w}$ and that we have taken compatible square roots $i_{1}$ and $i_{2}$. For the final part, we note that, using the real formula for $\overline{u}*\overline{v}$, the matrices $\{B_{\overline{u}},B_{\overline{v}},B_{\overline{u}*\overline{v}},R_{g}\}$ depend algebraically and rationally on the parameters $\{\overline{u},\overline{v},\gamma_{\overline{u}},\gamma_{\overline{v}},\gamma_{\overline{u}*\overline{v}}\}$, where $\{\overline{u},\overline{v}\}$ are real vectors with $max(\overline{u}^{2},\overline{v}^{2})<c^{2}$. It follows that the identity;\\

$B_{\overline{v}}B_{\overline{u}}=R_{g}B_{\overline{u}*\overline{v}}$ $(**)$\\

amounts to a set of rational  algebraic identities $R_{i}=0$, $1\leq i\leq 16$, involving $\{\overline{u},\overline{v},\gamma_{\overline{u}},\gamma_{\overline{v}},\gamma_{\overline{u}*\overline{v}}\}$ as well. As noted in \cite{U2}, these identities are still true when we make the substitution $\gamma_{\overline{u}}\gamma_{\overline{v}}(1+{\overline{u}\centerdot\overline{v}\over c^{2}})$, $(***)$, for $\gamma_{\overline{u}*\overline{v}}$. We let $V\subset \mathcal{C}^{6}$, be the open subvariety defined by;\\

$V=\{(\overline{u},\overline{v})\in \mathcal{C}^{6}:\{\overline{u}^{2},\overline{v}^{2}\}\cap \{0,c^{2}\}=\emptyset,1+{\overline{u}\centerdot\overline{v}\over c^{2}}\neq 0,(1+{\overline{u}\centerdot\overline{v}\over c^{2}})^{2}$\\

$-(1-{\overline{u}^{2}\over c^{2}})(1-{\overline{v}^{2}\over c^{2}})\neq 0\}$\\

and $C\subset V\times \mathcal{C}^{2}$ be the double cover of $V$ defined by;\\

$C=\{(\overline{u},\overline{v},w_{1},w_{2})\in V\times \mathcal{C}^{2}:w_{1}^{2}={1\over 1-{\overline{u}^{2}\over c^{2}}},w_{2}^{2}={1\over 1-{\overline{v}^{2}\over c^{2}}}\}$\\

Making the substitutions $w_{1}$ for $\gamma_{\overline{u}}$ and $w_{2}$ for $\gamma_{\overline{v}}$, the closed rational algebraic relations $R_{i}(\overline{u},\overline{v},w_{1},w_{2})=0$ hold generically on $C$. The conditions on $\{\overline{u},\overline{v}\}$ are necessary to ensure the denominators in the definitions of $\{B_{\overline{u}},B_{\overline{v}},B_{\overline{u}*\overline{v}},R_{g}\}$ are non-zero on $C$, so, by construction the rational functions $R_{i}$ have no poles. It follows that the $R_{i}$ are identically zero  on $C$. In particular, the identity $(**)$ holds for all complex vectors $\{\overline{u},\overline{v}\}$ with $\{\overline{u}^{2},\overline{v}^{2}\}\cap \{0,c^{2}\}=\emptyset$, $1+{\overline{u}\centerdot\overline{v}\over c^{2}}\neq 0$, $(1+{\overline{u}\centerdot\overline{v}\over c^{2}})^{2}-(1-{\overline{u}^{2}\over c^{2}})(1-{\overline{v}^{2}\over c^{2}})\neq 0$, and choices of square root in $\{\gamma_{\overline{u}},\gamma_{\overline{v}}\}$, $(\dag)$. We need to take $\gamma_{\overline{u}*\overline{v}}=\gamma_{\overline{u}}\gamma_{\overline{v}}(1+{\overline{u}\centerdot\overline{v}\over c^{2}})$, and the coefficients $\{c_{1},c_{2}\}$ in $g_{i_{1}i_{2}}$, see the footnote in Lemma \ref{velocity}, must be determined by the choices in $(\dag)$. Again, formulating the property $g\in SG(3)$ as a set of closed conditions, and using the fact that they holds generically, we obtain that $g_{i_{1}i_{2}}\in SG(3)$. Finally, we need to check that the identity $(***)$ holds up to a minus sign, for any choice of root in $\{\gamma_{\overline{u}},\gamma_{\overline{v}},\gamma_{\overline{u}*\overline{v}}\}$, and for $\{\overline{u},\overline{v}\}$ satisfying the usual conditions. This follows by verifying the identity given in \cite{U2};\\

$(\overline{u}*\overline{v})^{2}={(\overline{u}+\overline{v})^{2}\over (1+{\overline{u}\centerdot\overline{v}\over c^{2}})^{2}}-{1\over c^{2}}{(\overline{u}\times\overline{v})^{2}\over (1+{\overline{u}\centerdot\overline{v}\over c^{2}})^{2}}$\\

This is a straightforward calculation involving verifying the identities;\\

$\overline{u}\times (\overline{u}\times\overline{v})=(\overline{u}\centerdot\overline{v})^{2}-\overline{u}^{2}\overline{v}^{2}$,  $(\overline{u}\times\overline{v})^{2}=\overline{u}^{2}\overline{v}^{2}-(\overline{u}\centerdot\overline{v})^{2}$\\

for complex vectors. Then, computing $\gamma_{\overline{u}*\overline{v}}$ gives the result. We thus obtain the relation;\\

$B_{\overline{v}}^{i_{2}}B_{\overline{u}}^{i_{1}}=R_{g_{i_{1}i_{2}}}B_{(\overline{u}*\overline{v})^{i_{3}}}^{i_{4}}$ $(\dag\dag)$\\

for some $g_{i_{1}i_{2}}\in SG(3)$. For the rest of the proof, we can follow the argument in Lemma \ref{velocity}. Note that the conditions in this Lemma on $\{\overline{u},\overline{v}\}$ are symmetric, so that making the substitutions, $-\overline{v}$ for $\overline{u}$ and $-\overline{u}$ for $\overline{v}$, observing the choice of root for $\gamma_{\overline{v}*\overline{u}}$ is admissible for $\gamma_{-\overline{v}*-\overline{u}}$ or $\gamma_{-(\overline{v}*\overline{u})}$, we can use the first part of this lemma, to conclude that there exists $h\in SG(3)$ with;\\

$B_{\overline{v}}^{i_{2}}B_{\overline{u}}^{i_{1}}=B_{(\overline{v}*\overline{u})^{i_{5}}}^{i_{6}}R_{h}$ $(\dag\dag\dag)$\\

where $i_{6}$ is determined by the formula $(*)$ in the statement of the Lemma. Then, we can use the conjugation result $(\dag\dag)$, applied to $(\dag)$, noting that $((\overline{u}*\overline{v})^{i_{3}})^{2}\neq c^{2}$, by $(*)$, to obtain;\\

$B_{\overline{v}}^{i_{2}}B_{\overline{u}}^{i_{1}}
=B_{g_{i_{1}i_{2}}((\overline{u}*\overline{v})^{i_{3}})}^{i_{4}}R_{g_{i_{1}i_{2}}}$\\

By the uniqueness of representation, noting that $((\overline{v}*\overline{v})^{i_{5}})^{2}\neq c^{2}$ and $(g_{i_{1}i_{2}}((\overline{u}*\overline{v})^{i_{3}}))^{2}\neq c^{2}$, we conclude that $h=g_{i_{1}i_{2}}$, $i_{4}=i_{6}$, and $g_{i_{1}i_{2}}((\overline{u}*\overline{v})^{i_{3}})=(\overline{v}*\overline{u})^{i_{5}}$ as required.\\

The next claim can be seen by following through the computation for real $\overline{v}$, with $|\overline{v}|<c$. For the penultimate claim, we can use the fact that $(\rho,\overline{J})$ transforms as a four-vector, and the components of $\{\overline{E},\overline{B}\}$ transform as part of the covariant field tensor, see \cite{G}. Then we can use the result that the identities holds for generically independent real $\{\overline{u},\overline{v}\}$, with $|\overline{u}|<c$ and $|\overline{v}|<c$. For the last claim, we can just use the identification of matrices in $(**)$ of the lemma.
\end{proof}

\begin{lemma}{Limit Frames}
\label{limit}

Given a series of frames $S_{f(s)\overline{v}}$, where $|\overline{v}|=1$, connected to the base frame $S$ by the boost matrices $B_{f(s)\overline{v}}$, with $lim_{s\rightarrow \infty}f(s)=\infty$, $f$  smooth and positive real-valued on $\mathcal{R}_{>0}$, the boost matrix $lim_{s\rightarrow \infty,f(s)\neq c}B_{f(s)\overline{v}}$ exists, with a given choice of square root, and defines a limit frame $S_{\infty}$. Given a series of transformations $(\rho_{f(s)\overline{v}},\overline{J}_{f(s)\overline{v}})$ and $(\overline{E}_{f(s)\overline{v}},\overline{B}_{f(s)\overline{v}})$ of $(\rho,\overline{J})$ and $(\overline{E},\overline{B})$ from the base frame $S$, we define the transformation to $S_{\infty}$ by taking the limit as $s\rightarrow\infty$, of the transformation rules, see \cite{L}, for the pairs $(\rho_{f(s)\overline{v}},\overline{J}_{f(s)\overline{v}})$ and $(\overline{E}_{f(s)\overline{v}},\overline{B}_{f(s)\overline{v}})$. If $((\rho,\overline{J})$ is a real pair and $(\overline{E},\overline{B})$ is a complex valued pair, then the limit exists at the corresponding coordinates in $S_{\infty}$. Defining the transformation of derivatives from $S_{\infty}$ to $S$ in the usual way by $(lim_{s\rightarrow \infty,f(s)\neq c}B_{f(s)\overline{v}})^{-1}$, we have that the transformation is given by $lim_{s\rightarrow \infty,f(s)\neq c}((B_{f(s)\overline{v}})^{-1})$. In particularly, the limit definitions are independent of the choice of $f$ with $lim_{s\rightarrow \infty}f(s)=\infty$.

\end{lemma}

\begin{proof}
For the first claim, using the formula for the boost matrix $B_{f(s)\overline{v}}$, given in Lemma \ref{conjugation};\\

$B_{f(s)\overline{v}}=I+{\gamma_{s} b_{f(s)\overline{v}}\over c}+{\gamma_{s}^{2}b_{f(s)\overline{v}}^{2}\over c^{2}(\gamma_{s}+1)}$\\

where $\gamma_{s}={1\over \sqrt{1-{f(s)^{2}\over c^{2}}}}$. We can compute the limit, taking a positive square root;\\

$lim_{s\rightarrow \infty,f(s)\neq c}(\gamma_{s}+1)=1+lim_{s\rightarrow \infty,f(s)\neq c}{1\over \sqrt{1-{f(s)^{2}\over c^{2}}}}$\\

$=1+lim_{s\rightarrow \infty,f(s)\neq c}{1\over f(s)\sqrt{{1\over f(s)}^{2}-{1\over c^{2}}}}=1$\\

$lim_{s\rightarrow \infty,f(s)\neq c}f(s)\gamma_{s}=lim_{s\rightarrow \infty,f(s)\neq c}{f(s)\over \sqrt{1-{f(s)^{2}\over c^{2}}}}$\\

$=lim_{s\rightarrow \infty,f(s)\neq c}{1\over \sqrt{{1\over f(s)^{2}}-{1\over c^{2}}}}=-ic$\\

$lim_{s\rightarrow \infty,f(s)\neq c}f(s)^{2}\gamma_{s}^{2}=lim_{s\rightarrow \infty,f(s)\neq c}{f(s)^{2}\over 1-{f(s)^{2}\over c^{2}}}=-c^{2}$\\

It is then straightforward to see that that $lim_{s\rightarrow \infty,f(s)\neq c}({\gamma_{s} b_{f(s)\overline{v}}\over c})_{ij}$ and $lim_{s\rightarrow \infty,f(s)\neq c}({\gamma_{s}^{2}b_{f(s)\overline{v}}^{2}\over c^{2}(\gamma_{s}+1)})_{ij}$ exist for $0\leq i,j\leq 4$, as required. For the second claim, we have that;\\

$\rho_{f(s)\overline{v}}=\gamma_{s}(\rho-{<f(s)\overline{v},\overline{J}>\over c^{2}})$\\

$\overline{J}_{f(s)\overline{v}}=\gamma_{s}(\overline{J}_{||,s}-f(s)\overline{v}\rho)+\overline{J}_{\perp,s}$\\

We have , for a vector field $\overline{F}$ and scalar $\rho$, that;\\

$lim_{s\rightarrow\infty}\overline{F}_{||,s}=lim_{s\rightarrow\infty}{<\overline{F},f(s)\overline{v}>f(s)\overline{v}\over f(s)^{2}}=<\overline{F},\overline{v}>\overline{v}=\overline{F}_{||}$\\

so that;\\

 $lim_{s\rightarrow\infty,f(s)\neq c}\overline{F}_{\perp,s}=\overline{F}-\overline{F}_{||,s}=\overline{F}_{\perp}$\\

and $lim_{s\rightarrow\infty,f(s)\neq c}\gamma_{s}\overline{F}_{||,s}=lim_{s\rightarrow\infty}\gamma_{s}{<\overline{F},f(s)\overline{v}>f(s)\overline{v}\over f(s)^{2}}=0$\\

As above, we have that $lim_{s\rightarrow\infty,f(s)\neq c}\gamma_{s}\rho=0$, and\\

$lim_{s\rightarrow\infty,f(s)\neq c}\gamma_{s}<f(s)\overline{v},\overline{F}>=lim_{s\rightarrow\infty}\gamma_{s}f(s)<\overline{v},\overline{F}>=-ic<\overline{v},\overline{F}>$\\

$lim_{s\rightarrow\infty,f(s)\neq c}\gamma_{s}f(s)\overline{v}\rho=-ic\overline{v}\rho$\\

so that;\\

$lim_{s\rightarrow\infty,f(s)\neq c}\rho_{f(s)\overline{v}}={i\over c}<\overline{v},\overline{J}>$\\

$lim_{s\rightarrow\infty,f(s)\neq c}\overline{J}_{f(s)\overline{v}}=ic\overline{v}\rho+\overline{J}_{\perp}$\\

Similarly, we have that;\\

$\overline{E}_{f(s)\overline{v}}=\overline{E}_{||,s}+\gamma_{s}(\overline{E}_{\perp,s}+f(s)\overline{v}\times \overline{B})$\\

$\overline{B}_{f(s)\overline{v}}=\overline{B}_{||,s}+\gamma_{s}(\overline{B}_{\perp,s}-{f(s)\overline{v}\times \overline{E}\over c^{2}})$\\

so that, using the above computations again, replacing $<\overline{v},\overline{F}>$ by $\overline{v}\times \overline{F}$;\\

$lim_{s\rightarrow\infty,f(s)\neq c}\overline{E}_{f(s)\overline{v}}=\overline{E}_{||}-ic(\overline{v}\times \overline{B})$\\

$lim_{s\rightarrow\infty,f(s)\neq c}\overline{B}_{f(s)\overline{v}}=\overline{B}_{||}+{i\over c}(\overline{v}\times \overline{E})$\\

For the penultimate claim we have that $lim_{s\rightarrow \infty,f(s)\neq c}B_{f(s)\overline{v}}$ is invertible, which can be seen from the inverse boost matrix, replacing $\overline{v}$ by $-\overline{v}$, and noting the limit exists again. Then, we can use the formula for the inverse of a matrix, and elementary properties of limits. The last claim is clear from the above calculation.

\end{proof}

\begin{defn}
\label{stressenergy}
In the context of special relativity, we choose coordinates $x^{0}=ct$, $x^{1}=x$, $x^{2}=y$, $x^{3}=z$. We have the coordinate relationship for the Lorentz transformation;\\

$x^{0'}=\gamma(x^{0}-{vx^{1}\over c})$\\

$x^{1'}=\gamma(x^{1}-{vx^{0}\over c})$\\

$x^{2'}=x^{2}$\\

$x^{3'}=x^{3}$\\

where $v$ is the velocity of a boost in the $x$-direction, and $\gamma={1\over \sqrt{1-{v^{2}\over c^{2}}}}$.\\

We can encode the transformation with the Lorentz matrix given by $\overline{\Lambda}$, which is defined by;\\

$(\overline{\Lambda})_{00}=(\overline{\Lambda})_{11}=\gamma$\\

$(\overline{\Lambda})_{01}=(\overline{\Lambda})_{10}=-\gamma\beta$\\

$(\overline{\Lambda})_{22}=(\overline{\Lambda})_{33}=1$\\

$(\overline{\Lambda})_{ij}=0$, otherwise, for $0\leq i,j\leq 3$\\

where $\beta={v\over c}$. We let;\\

$\sigma={\epsilon_{0}\over 2}(e^{2}+c^{2}b^{2})$\\

where $e=|\overline{E}|$, $b=|\overline{B}|$, and $\{\overline{E},\overline{B}\}$ are electric and magnetic fields, satisfying Maxwell's equations, in the rest frame. We let;\\

$\overline{g}=(g_{1},g_{2},g_{3})=\epsilon_{0}(\overline{E}\times\overline{B})$\\

For $1\leq i,j\leq 3$, we let;\\

$p_{ij}=-\epsilon_{0}(e_{i}e_{j}+c^{2}b_{i}b_{j}-{1\over 2}\delta_{ij}(e^{2}+c^{2}b^{2}))$\\

be Maxwell's stress tensor, where $\overline{E}=(e_{1},e_{2},e_{3})$ and $\overline{B}=(b_{1},b_{2},b_{3})$. The stress energy tensor is given by $\overline{M}$, defined by;\\

$(\overline{M})_{00}=\sigma$\\

$(\overline{M})_{i0}=cg_{i}$, for $1\leq i\leq 3$\\

$(\overline{M})_{0j}=cg_{j}$, for $1\leq j\leq 3$\\

$(\overline{M})_{ij}=p_{ij}$, for $1\leq i,j\leq 3$\\

It transforms between inertial frames, using the summation rule, see \cite{R};\\

$(\overline{M}')_{i'j'}=(\overline{\Lambda})_{ii'}(\overline{\Lambda})_{jj'}(\overline{M})_{ij}$\\

\end{defn}

\begin{lemma}
\label{families}
Suppose that $(\rho,\overline{J})$ satisfies the continuity equation and is surface non-radiating in the sense of Definition 2.8 of \cite{dep1}, then there exist $3$ real families of electric and magnetic fields, indexed by $v\in\mathcal{R}$, $(\overline{E}^{1}_{v},\overline{B}^{1}_{v})$, $(\overline{E}^{2}_{v},\overline{B}^{2}_{v})$, $(\overline{E}^{3}_{v},\overline{B}^{3}_{v})$, satisfying Maxwell's equations in the rest frame $S$ and the additional equations;\\

$\alpha_{v}{\partial \sigma_{v}^{1}\over \partial x}+\beta_{v}{\partial g_{1,v}^{1}\over \partial x}+\gamma_{v}{\partial p_{11,v}^{1}\over \partial x}+\delta_{v}{\partial \sigma_{v}^{1}\over \partial t}+\epsilon_{v}{\partial g_{1,v}^{1}\over \partial t}+\xi_{v}{\partial p_{11,v}^{1}\over \partial t}+\eta_{v} div(\overline{g}_{v}^{1})$\\

$+\theta_{v}(f_{1,v}^{1}+{\partial g_{1,v}^{1}\over \partial t})=0$\\

$\alpha_{v}{\partial \sigma_{v}^{2}\over \partial y}+\beta_{v}{\partial g_{2,v}^{2}\over \partial y}+\gamma_{v}{\partial p_{22,v}^{2}\over \partial y}+\delta_{v}{\partial \sigma_{v}^{2}\over \partial t}+\epsilon_{v}{\partial g_{2,v}^{2}\over \partial t}+\xi_{v}{\partial p_{22,v}^{2}\over \partial t}+\eta_{v} div(\overline{g}_{v}^{2})$\\

$+\theta_{v}(f_{2,v}^{2}+{\partial g_{2,v}^{2}\over \partial t})=0$\\

$\alpha_{v}{\partial \sigma_{v}^{3}\over \partial z}+\beta_{v}{\partial g_{3,v}^{3}\over \partial z}+\gamma_{v}{\partial p_{33,v}^{3}\over \partial z}+\delta_{v}{\partial \sigma_{v}^{3}\over \partial t}+\epsilon_{v}{\partial g_{3,v}^{3}\over \partial t}+\xi_{v}{\partial p_{33,v}^{3}\over \partial t}+\eta_{v} div(\overline{g}_{v}^{3})$\\

$+\theta_{v}(f_{3,v}^{3}+{\partial g_{3,v}^{3}\over \partial t})=0$ $(\dag)$\\

where;\\

$\alpha_{v}={-\beta\gamma^{3}\over c}, \beta_{v}=((\beta^{2}+1)\gamma^{3}-\gamma), \gamma_{v}=({\gamma \beta\over c}-{\gamma^{3}\beta\over c}), \delta_{v}={-\beta\gamma^{3} v\over c^{3}}$\\

$\epsilon_{v}={(\beta^{2}+1)\gamma^{3} v\over c^{2}}, \xi_{v}={-\beta\gamma^{3} v\over c^{3}}, \eta_{v}=\gamma, \theta_{v}={\gamma\beta\over c}$\\

Moreover, we can take a fixed pair $(\overline{E},\overline{B})$ in the rest frame, with $div(\overline{E}\times\overline{B})=0$, such that $\overline{E}=\overline{E}^{1}_{0}=\overline{E}^{2}_{0}=\overline{E}^{3}_{0}$ and $\overline{B}=\overline{B}^{1}_{0}=\overline{B}^{2}_{0}=\overline{B}^{3}_{0}$.

\end{lemma}

\begin{proof}

Transforming between frames, and corresponding fields $(\overline{E}',\overline{B}')$ in $S'$ and $(\overline{E}^{1}_{v},\overline{B}^{1}_{v})$ in $S$, we have, dropping the index $v$ throughout the proof, that;\\

$\sigma'=(\overline{M}')_{00}=(\overline{\Lambda})_{i0}(\overline{\Lambda})_{j0}(\overline{M})_{ij}$\\

$=\gamma^{2}\sigma-\gamma^{2}\beta cg_{1}-\gamma^{2}\beta cg_{1}+\gamma^{2}\beta^{2}p_{11}$\\

$=\gamma^{2}(\sigma-2\beta cg_{1}+\beta^{2}p_{11})$\\

$cg_{1}'=(\overline{M}')_{10}=(\overline{\Lambda})_{i1}(\overline{\Lambda})_{j0}(\overline{M})_{ij}$\\

$=-\gamma^{2}\beta\sigma+\gamma^{2}\beta^{2}cg_{1}+\gamma^{2}cg_{1}-\gamma^{2}\beta p_{11}$\\

$=\gamma^{2}(-\beta\sigma+(\beta^{2}c+c)g_{1}-\beta p_{11})$\\

$cg_{2}'=(\overline{M}')_{20}=(\overline{\Lambda})_{i2}(\overline{\Lambda})_{j0}(\overline{M})_{ij}$\\

$=\gamma cg_{2}-\gamma\beta p_{21}$\\

$cg_{3}'=(\overline{M}')_{30}=(\overline{\Lambda})_{i3}(\overline{\Lambda})_{j0}(\overline{M})_{ij}$\\

$=\gamma cg_{3}-\gamma\beta p_{31}$ $(*)$\\

The condition that $div'(\overline{E}'\times\overline{B}')=0$ in the frame $S'$ is equivalent to;\\

$\bigtriangledown'\centerdot (g_{1}',g_{2}',g_{3}')=0$ $(**)$\\

We have the transformation rule for $\bigtriangledown'$, given in \cite{L};\\

${\partial\over\partial x'}=\gamma({\partial\over\partial x}+{v\over c^{2}}{\partial\over\partial t})$\\

${\partial\over\partial y'}={\partial\over\partial y}$\\

${\partial\over\partial z'}={\partial \over\partial z}$\\

Applying this to $(*),(**)$, we obtain;\\

$\gamma({\partial\over\partial x}+{v\over c^{2}}{\partial \over\partial t})g_{1}'+{\partial g_{2}'\over \partial y}+{\partial g_{3}'\over \partial z}$\\

$={\gamma\over c}{\partial\over\partial x}(\gamma^{2}(-\beta\sigma+(\beta^{2}c+c)g_{1}-\beta p_{11}))$\\

$+{\gamma v\over c^{3}}{\partial\over\partial t}(\gamma^{2}(-\beta\sigma+(\beta^{2}c+c)g_{1}-\beta p_{11}))$\\

$+{1\over c}{\partial \over \partial y}(\gamma cg_{2}-\gamma\beta p_{21})$\\

$+{1\over c}{\partial \over \partial z}(\gamma cg_{3}-\gamma\beta p_{31})=0$, $(***)$\\

and, rearranging, we have that;\\

${-\beta\gamma^{3}\over c}{\partial \sigma\over \partial x}+((\beta^{2}+1)\gamma^{3}-\gamma){\partial g_{1}\over \partial x}+\gamma div(\overline{g})+({\gamma \beta\over c}-{\gamma^{3}\beta\over c}){\partial p_{11}\over \partial x}$\\

$-{\gamma\beta\over c}div(\overline{T}_{1})-{\beta\gamma^{3} v\over c^{3}}{\partial \sigma\over \partial t}+{(\beta^{2}+1)\gamma^{3} v\over c^{2}}{\partial g_{1}\over \partial t}-{\beta\gamma^{3} v\over c^{3}}{\partial p_{11}\over \partial t}=0$, $(****)$\\

where $\overline{T}$ is the Maxwell stress tensor. We have, see \cite{G}, that;\\

$div(\overline{T}_{1})+{\partial g_{1}\over\partial t}=-f_{1}$\\

where $\overline{f}=(f_{1},f_{2},f_{3})$ is the force applied by the fields $\{\overline{E}^{1}_{v},\overline{B}^{1}_{v}\}$ relative to the charge and current $(\rho,\overline{J})$ in the rest frame $S$.\\

Rearranging again, we obtain;\\

${-\beta\gamma^{3}\over c}{\partial \sigma\over \partial x}+((\beta^{2}+1)\gamma^{3}-\gamma){\partial g_{1}\over \partial x}+({\gamma \beta\over c}-{\gamma^{3}\beta\over c}){\partial p_{11}\over \partial x}-{\beta\gamma^{3} v\over c^{3}}{\partial \sigma\over \partial t}$\\

$+{(\beta^{2}+1)\gamma^{3} v\over c^{2}}{\partial g_{1}\over \partial t}-{\beta\gamma^{3} v\over c^{3}}{\partial p_{11}\over \partial t}+\gamma div(\overline{g})+{\gamma\beta\over c}(f_{1}+{\partial g_{1}\over \partial t})=0$ $(*****)$\\

By symmetry, for boosts with velocity $v$ in the $y$ and $z$ directions, we obtain the relations for $(\overline{E}^{2}_{v},\overline{B}^{2}_{v})$ and $(\overline{E}^{3}_{v},\overline{B}^{3}_{v})$ in the rest frame $S$;\\

${-\beta\gamma^{3}\over c}{\partial \sigma\over \partial y}+((\beta^{2}+1)\gamma^{3}-\gamma){\partial g_{2}\over \partial y}+({\gamma \beta\over c}-{\gamma^{3}\beta\over c}){\partial p_{22}\over \partial y}-{\beta\gamma^{3} v\over c^{3}}{\partial \sigma\over \partial t}$\\

$+{(\beta^{2}+1)\gamma^{3} v\over c^{2}}{\partial g_{2}\over \partial t}-{\beta\gamma^{3} v\over c^{3}}{\partial p_{22}\over \partial t}+\gamma div(\overline{g})+{\gamma\beta\over c}(f_{2}+{\partial g_{2}\over \partial t})=0$\\

and;\\

${-\beta\gamma^{3}\over c}{\partial \sigma\over \partial z}+((\beta^{2}+1)\gamma^{3}-\gamma){\partial g_{3}\over \partial z}+({\gamma \beta\over c}-{\gamma^{3}\beta\over c}){\partial p_{33}\over \partial z}-{\beta\gamma^{3} v\over c^{3}}{\partial \sigma\over \partial t}$\\

$+{(\beta^{2}+1)\gamma^{3} v\over c^{2}}{\partial g_{3}\over \partial t}-{\beta\gamma^{3} v\over c^{3}}{\partial p_{33}\over \partial t}+\gamma div(\overline{g})+{\gamma\beta\over c}(f_{3}+{\partial g_{3}\over \partial t})=0$\\

The final claim is clear by the definition of surface non-radiating.

\end{proof}

\begin{lemma}
\label{geometry}
For $(\overline{x_{0}},t_{0})$ in the rest frame $S$, there exists, at $(\overline{x_{0}},t_{0})$ a polynomial approximation of $(\overline{E}_{v}^{1},\overline{B}_{v}^{1})$ from Lemma \ref{families}, satisfying Maxwell's equations, and the additional equations there. Moreover, the conditions are algebraic.
\end{lemma}

\begin{proof}
Fix $(\overline{x}_{0},t_{0})$ in the rest frame $S$ and define the vector fields $(\overline{L}^{1}_{v},\overline{M}^{1}_{v})$ by;\\

$l^{1}_{i,v}=\sum_{0\leq j+k+l+m\leq 1}l^{1}_{i,jklm,v}(x-x_{0})^{j}(y-y_{0})^{k}(z-z_{0})^{l}(t-t_{0})^{m}$\\

$m^{1}_{i,v}=\sum_{0\leq j+k+l+m\leq 1}m^{1}_{i,jklm,v}(x-x_{0})^{j}(y-y_{0})^{k}(z-z_{0})^{l}(t-t_{0})^{m}$\\

for $1\leq i\leq 3$, where $\overline{L}^{1}_{v}=(l^{1}_{1,v},l^{1}_{2,v},l^{1}_{3,v})$, $\overline{M}^{1}_{v}=(m^{1}_{1,v},m^{1}_{2,v},m^{1}_{3,v})$ and;\\

$l^{1}_{i,jklm,v}={\partial^{(j+k+l+m)}e^{1}_{i,v}\over \partial x^{j}\partial y^{k}\partial z^{l}\partial t^{m}}|_{(\overline{x}_{0},t_{0})}$\\

$m^{1}_{i,jklm,v}={\partial^{(j+k+l+m)}b^{1}_{i,v}\over \partial x^{j}\partial y^{k}\partial z^{l}\partial t^{m}}|_{(\overline{x}_{0},t_{0})}$\\

Then, for $v\in\mathcal{R}$, at $(\overline{x}_{0},t_{0})$, $(\overline{L}^{1}_{v},\overline{M}^{1}_{v})$ satisfy Maxwell's equations, see the proof in \cite{L}, and the first equation of $(\dag)$ in Lemma \ref{families}. The satisfaction of Maxwell's equations at $(x_{0},t_{0})$, defines $8$ linear conditions on the $30$ coefficients $l^{1}_{i,jklm,v}$ and $m^{1}_{i,jklm,v}$ given by;\\

$l^{1}_{1,1000,v}+l^{1}_{2,0100,v}+l^{1}_{3,0010,v}-{\rho(\overline{x}_{0},t_{0})\over\epsilon_{0}}=0$ $(i)$\\

$l^{1}_{3,0100,v}-l^{1}_{2,0010,v}+m^{1}_{1,0001,v}=0$\\

$l^{1}_{3,1000,v}-l^{1}_{1,0010,v}+m^{1}_{2,0001,v}=0$\\

$l^{1}_{2,1000,v}-l^{1}_{1,0100,v}+m^{1}_{3,0001,v}=0$ $(ii)$\\

$m^{1}_{1,1000,v}+m^{1}_{2,0100,v}+m^{1}_{3,0010,v}=0$ $(iii)$\\

$m^{1}_{3,0100,v}-m^{1}_{2,0010,v}-\mu_{0}\epsilon_{0}l^{1}_{1,0001,v}-\mu_{0}j_{1}(\overline{x}_{0},t_{0})=0$\\

$m^{1}_{3,1000,v}-m^{1}_{1,0010,v}-\mu_{0}\epsilon_{0}l^{1}_{2,0001,v}-\mu_{0}j_{2}(\overline{x}_{0},t_{0})=0$\\

$m^{1}_{2,1000,v}-m^{1}_{1,0100,v}-\mu_{0}\epsilon_{0}l^{1}_{3,0001,v}-\mu_{0}j_{3}(\overline{x}_{0},t_{0})=0$ $(iv)$\\

where $\overline{J}=(j_{1},j_{2},j_{3})$. For the first equation of $(\dag)$ in Lemma \ref{families}, a simple computation using the product rule and the formula given in \cite{G};\\

$\overline{f}^{1}_{v}=\rho\overline{E}^{1}_{v}+\overline{J}\times\overline{B}^{1}_{v}$\\

we obtain that;\\

$\alpha_{v}{\partial \sigma_{v}^{1}\over \partial x}=\epsilon_{0}\alpha_{v}(l^{1}_{1,0000,v}l^{1}_{1,1000,v}+l^{1}_{2,0000,v}l^{1}_{2,1000,v}+l^{1}_{3,0000,v}l^{1}_{3,1000,v}$\\

$+c^{2}m^{1}_{1,0000,v}m^{1}_{1,1000,v}+c^{2}m^{1}_{2,0000,v}m^{1}_{2,1000,v}+c^{2}m^{1}_{3,0000,v}m^{1}_{3,1000,v})$\\

$\beta_{v}{\partial g_{1,v}^{1}\over \partial x}=\epsilon_{0}\beta_{v}(l^{1}_{2,1000,v}m^{1}_{3,0000,v}+
l^{1}_{2,0000,v}m^{1}_{3,1000,v}-l^{1}_{3,1000,v}m^{1}_{2,0000,v}$\\

$-l^{1}_{3,0000,v}m^{1}_{2,1000,v})$\\

$\gamma_{v}{\partial p_{11,v}^{1}\over \partial x}=-\epsilon_{0}\gamma_{v}(l^{1}_{1,0000,v}l^{1}_{1,1000,v}+
c^{2}m^{1}_{1,0000,v}m^{1}_{1,1000,v}-l^{1}_{2,0000,v}l^{1}_{2,1000,v}$\\

$-l^{1}_{3,0000,v}l^{1}_{3,1000,v}
-c^{2}m^{1}_{2,0000,v}m^{1}_{2,1000,v}
-c^{2}m^{1}_{3,0000,v}m^{1}_{3,1000,v})$\\

$\delta_{v}{\partial \sigma_{v}^{1}\over \partial t}=\epsilon_{0}\delta_{v}(l^{1}_{1,0000,v}l^{1}_{1,0001,v}
+l^{1}_{2,0000,v}l^{1}_{2,0001,v}
+l^{1}_{3,0000,v}l^{1}_{3,0001,v}
+c^{2}m^{1}_{1,0000,v}m^{1}_{1,0001,v}$\\

$+c^{2}m^{1}_{2,0000,v}m^{1}_{2,0001,v}
+c^{2}m^{1}_{3,0000,v}m^{1}_{3,0001,v})$\\

$\epsilon_{v}{\partial g_{1,v}^{1}\over \partial t}=\epsilon_{0}\epsilon_{v}(l^{1}_{2,0001,v}m^{1}_{3,0000,v}
+l^{1}_{2,0000,v}m^{1}_{3,0001,v}
-l^{1}_{3,0001,v}m^{1}_{2,0000,v}$\\

$-l^{1}_{3,0000,v}m^{1}_{2,0001,v})$\\

$\xi_{v}{\partial p_{11,v}^{1}\over \partial t}=-\epsilon_{0}\xi_{v}(l^{1}_{1,0000,v}l^{1}_{1,0001,v}
+c^{2}m^{1}_{1,0000,v}m^{1}_{1,0001,v}
-l^{1}_{2,0000,v}l^{1}_{2,0001,v}$\\

$-l^{1}_{3,0000,v}l^{1}_{3,0001,v}
-c^{2}m^{1}_{2,0000,v}m^{1}_{2,0001,v}
-c^{2}m^{1}_{3,0000,v}m^{1}_{3,0001,v})$\\

$\eta_{v} div(\overline{g}_{v}^{1})=\epsilon_{0}\eta_{v}(l^{1}_{2,1000,v}m^{1}_{3,0000,v}
+l^{1}_{2,0000,v}m^{1}_{3,1000,v}
-l^{1}_{3,1000,v}m^{1}_{2,0000,v}$\\

$-l^{1}_{3,0000,v}m^{1}_{2,1000,v}
+l^{1}_{3,0100,v}m^{1}_{1,0000,v}
+l^{1}_{3,0000,v}m^{1}_{1,0100,v}
-l^{1}_{1,0100,v}m^{1}_{3,0000,v}$\\

$-l^{1}_{1,0000,v}m^{1}_{3,0100,v}
+l^{1}_{1,0010,v}m^{1}_{2,0000,v}
+l^{1}_{1,0000,v}m^{1}_{2,0010,v}
-l^{1}_{2,0010,v}m^{1}_{1,0000,v}$\\

$-l^{1}_{2,0000,v}m^{1}_{1,0010,v})$\\

$\theta_{v}f_{1,v}^{1}=\theta_{v}(p(\overline{x}_{0},t_{0})l^{1}_{1,0000,v}+j_{2}(\overline{x}_{0},t_{0})m^{1}_{3,0000,v}
-j_{3}(\overline{x}_{0},t_{0})m^{1}_{2,0000,v})$\\

$\theta_{v}{\partial g_{1,v}^{1}\over \partial t}=\epsilon_{0}\theta_{v}(l^{1}_{2,0001,v}m^{1}_{3,0000,v}
+l^{1}_{2,0000,v}m^{1}_{3,0001,v}
-l^{1}_{3,0001,v}m^{1}_{2,0000,v}$\\

$-l^{1}_{3,0000,v}m^{1}_{2,0001,v})$ $(\dag\dag)$\\

Combining $(\dag)$ and $(\dag\dag)$ we obtain $1$ condition on the coefficients. We introduce $30$ new variables $\{x^{1}_{i,jklm},y^{1}_{i,jklm}\}$, for $1\leq i\leq 3$ and $0\leq j+k+l+m\leq 1$. Substituting the variables for the corresponding $\{l^{1}_{i,jklm,v},m^{1}_{i,jklm,v}\}$ in the equations $(\dag),(\dag\dag)$ and $(i)-(iv)$, when $v\in{\mathcal{C}\setminus \{-c,c\}}$, we obtain algebraic conditions.
\end{proof}

\begin{lemma}{Limit Equations}
\label{infinity}
We can take a limit as $v\rightarrow\infty$ of the equations $(\dag)$ in Lemma \ref{families}, to obtain the limit relations;\\

 ${\partial\over\partial t}(\overline{E}_{\infty}^{1}\times\overline{B}_{\infty}^{1})_{1}=-{1\over\epsilon_{0}}({\partial p_{12,\infty}^{1}\over \partial y}+{\partial p_{13,\infty}^{1}\over \partial z})$\\

${\partial\over\partial t}(\overline{E}_{\infty}^{2}\times\overline{B}_{\infty}^{2})_{2}=-{1\over\epsilon_{0}}({\partial p_{12,\infty}^{2}\over \partial x}+{\partial p_{23,\infty}^{2}\over \partial z})$\\

${\partial\over\partial t}(\overline{E}_{\infty}^{3}\times\overline{B}_{\infty}^{3})_{3}=-{1\over\epsilon_{0}}({\partial p_{13,\infty}^{3}\over \partial x}+{\partial p_{23,\infty}^{3}\over \partial y})$\\

for the transferred fields $(\overline{E}_{\infty}^{i},\overline{B}_{\infty}^{i})$, $1\leq i\leq 3$, see Lemma \ref{limit}.

\end{lemma}

\begin{proof}
Observe that, taking compatible square roots;\\

$lim_{v\rightarrow\infty}\gamma=lim_{v\rightarrow\infty}\alpha_{v}=lim_{v\rightarrow\infty}\beta_{v}=lim_{v\rightarrow\infty}\delta_{v}=lim_{v\rightarrow\infty}\xi_{v}$\\

$=lim_{v\rightarrow\infty}\eta_{v}=0$\\

$lim_{v\rightarrow\infty}\gamma_{v}=lim_{v\rightarrow\infty}\theta_{v}=lim_{v\rightarrow\infty}{\gamma\beta\over c}$\\

$=lim_{v\rightarrow\infty}{v\over c^{2}\sqrt{1-{v^{2}\over c^{2}}}}$\\

$={1\over c^{2}}lim_{v\rightarrow\infty}{1\over \sqrt{{1\over v^{2}}-{1\over c^{2}}}}$\\

$={1\over c^{2}}{c\over i}=-{i\over c}$\\

$lim_{v\rightarrow\infty}\epsilon_{v}=lim_{v\rightarrow\infty}{\beta^{2}v\gamma^{3}\over c^{2}}$\\

$=lim_{v\rightarrow\infty}{v^{3}\over c^{4}(1-{v^{2}\over c^{2}})^{3\over 2}}$\\

$={1\over c^{4}}lim_{v\rightarrow\infty}{1\over({1\over v^{2}}-{1\over c^{2}})^{3\over 2}}$\\

$={1\over c^{4}}{c^{3}\over i^{3}}={i\over c}$\\

Taking the limit of the first equation in $(\dag)$ of Lemma \ref{families}, we obtain;\\

$-{i\over c}{\partial p_{11,\infty}^{1}\over \partial x}+{i\over c}{\partial g_{1,\infty}^{1}\over \partial t}-{i\over c}(f_{1,\infty}^{1}+{\partial g_{1,\infty}^{1}\over \partial t})=0$\\

which simplifies to;\\

$f_{1,\infty}^{1}=-{\partial p_{11,\infty}^{1}\over \partial x}$ $(*)$\\

Similarly, taking limits of the second and third equations in $(\dag)$ of Lemma \ref{families}, we obtain;\\

$f_{2,\infty}^{2}=-{\partial p_{22,\infty}^{2}\over \partial y}$\\

$f_{3,\infty}^{3}=-{\partial p_{33,\infty}^{3}\over \partial z}$ $(**)$\\

Using the definition of $p_{ii}$, for $1\leq i\leq 3$, we can rearrange $(*), (**)$ to obtain;\\

$f_{1,\infty}^{1}=(\epsilon_{0}{\partial\over \partial x}((e_{1,\infty}^{1})^{2}+c^{2}(b_{1,\infty}^{1})^{2})-{\epsilon_{0}\over 2}{\partial\over \partial x}((e_{\infty}^{1})^{2}+c^{2}(b_{\infty}^{1})^{2}))$\\

$f_{2,\infty}^{2}=(\epsilon_{0}{\partial\over \partial y}((e_{2,\infty}^{2})^{2}+c^{2}(b_{2,\infty}^{2})^{2})-{\epsilon_{0}\over 2}{\partial\over \partial y}((e_{\infty}^{2})^{2}+c^{2}(b_{\infty}^{2})^{2}))$\\

$f_{3,\infty}^{3}=(\epsilon_{0}{\partial\over \partial z}((e_{3,\infty}^{3})^{2}+c^{2}(b_{3,\infty}^{3})^{2})-{\epsilon_{0}\over 2}{\partial\over \partial z}((e_{\infty}^{3})^{2}+c^{2}(b_{\infty}^{3})^{2}))$ $(***)$\\

Using the definition of force density $\overline{f}$, we have, see \cite{G}, the formula;\\

$\overline{f}=-{1\over 2}\bigtriangledown(\epsilon_{0}E^{2}+{1\over \mu_{0}}B^{2})-\epsilon_{0}{\partial\over\partial t}(\overline{E}\times\overline{B})+\epsilon_{0}((\bigtriangledown\centerdot\overline{E})\overline{E}+(\overline{E}\centerdot\bigtriangledown)\overline{E})$\\

$+{1\over \mu_{0}}((\bigtriangledown\centerdot\overline{B})\overline{B}+(\overline{B}\centerdot\bigtriangledown)\overline{B})$\\

and, substituting for the first equation in $(***)$, using the product rule, we obtain;\\

$-\epsilon_{0}{\partial\over\partial t}(\overline{E}_{\infty}^{1}\times\overline{B}_{\infty}^{1})_{1}={\partial\over \partial x}(\epsilon_{0}(e_{1,\infty}^{1})^{2}+{1\over \mu_{0}}(b_{1,\infty}^{1})^{2})$\\

$-\epsilon_{0}((\bigtriangledown\centerdot\overline{E})\overline{E}+(\overline{E}\centerdot\bigtriangledown)\overline{E})_{1}-{1\over \mu_{0}}((\bigtriangledown\centerdot\overline{B})\overline{B}+(\overline{B}\centerdot\bigtriangledown)\overline{B})_{1}$\\

$={\partial\over \partial x}(\epsilon_{0}(e_{1,\infty}^{1})^{2}+{1\over \mu_{0}}(b_{1,\infty}^{1})^{2})-\epsilon_{0}({\partial e_{1,\infty}^{1}\over \partial x}+{\partial e_{2,\infty}^{1}\over \partial y}+{\partial e_{3,\infty}^{1}\over \partial z})e_{1,\infty}^{1}$\\

$-\epsilon_{0}(e_{1,\infty}^{1}{\partial\over\partial x}+e_{2,\infty}^{1}{\partial\over\partial y}+e_{3,\infty}^{1}{\partial\over\partial z})e_{1,\infty}^{1}-{1\over \mu_{0}}({\partial b_{1,\infty}^{1}\over \partial x}+{\partial b_{2,\infty}^{1}\over \partial y}+{\partial b_{3,\infty}^{1}\over \partial z})b_{1,\infty}^{1}$\\

$-{1\over\mu_{0}}(b_{1,\infty}^{1}{\partial\over\partial x}+b_{2,\infty}^{1}{\partial\over\partial y}+b_{3,\infty}^{1}{\partial\over\partial z})b_{1,\infty}^{1}$\\

$=-\epsilon_{0}({\partial (e_{1,\infty}^{1}e_{2,\infty}^{1})\over\partial y}+{\partial (e_{1,\infty}^{1}e_{3,\infty}^{1})\over\partial z})-{1\over \mu_{0}}({\partial (b_{1,\infty}^{1}b_{2,\infty}^{1})\over\partial y}+{\partial (b_{1,\infty}^{1}b_{3,\infty}^{1})\over\partial z})$\\

and, rearranging;\\

${\partial\over\partial t}(\overline{E}_{\infty}^{1}\times\overline{B}_{\infty}^{1})_{1}=({\partial (e_{1,\infty}^{1}e_{2,\infty}^{1}+c^{2}b_{1,\infty}^{1}b_{2,\infty}^{1})\over\partial y}+{\partial (e_{1,\infty}^{1}e_{3,\infty}^{1}+c^{2}b_{1,\infty}^{1}b_{3,\infty}^{1})\over\partial z})$ $(\dag)$\\

Similarly, substituting into the second and third equations of $(***)$, we obtain;\\

${\partial\over\partial t}(\overline{E}_{\infty}^{2}\times\overline{B}_{\infty}^{2})_{2}=({\partial (e_{1,\infty}^{2}e_{2,\infty}^{2}+c^{2}b_{1,\infty}^{2}b_{2,\infty}^{2})\over\partial x}+{\partial (e_{2,\infty}^{2}e_{3,\infty}^{2}+c^{2}b_{2,\infty}^{2}b_{3,\infty}^{2})\over\partial z})$\\

${\partial\over\partial t}(\overline{E}_{\infty}^{3}\times\overline{B}_{\infty}^{3})_{3}=({\partial (e_{1,\infty}^{3}e_{3,\infty}^{3}+c^{2}b_{1,\infty}^{3}b_{3,\infty}^{3})\over\partial x}+{\partial (e_{2,\infty}^{3}e_{3,\infty}^{3}+c^{2}b_{2,\infty}^{3}b_{3,\infty}^{3})\over\partial y})$ $(\dag\dag)$\\

Using the definition of the stress tensor, we can write this as;\\

${\partial\over\partial t}(\overline{E}_{\infty}^{1}\times\overline{B}_{\infty}^{1})_{1}=-{1\over\epsilon_{0}}({\partial p_{12,\infty}^{1}\over \partial y}+{\partial p_{13,\infty}^{1}\over \partial z})$\\

${\partial\over\partial t}(\overline{E}_{\infty}^{2}\times\overline{B}_{\infty}^{2})_{2}=-{1\over\epsilon_{0}}({\partial p_{12,\infty}^{2}\over \partial x}+{\partial p_{23,\infty}^{2}\over \partial z})$\\

${\partial\over\partial t}(\overline{E}_{\infty}^{3}\times\overline{B}_{\infty}^{3})_{3}=-{1\over\epsilon_{0}}({\partial p_{13,\infty}^{3}\over \partial x}+{\partial p_{23,\infty}^{3}\over \partial y})$ $(\dag\dag\dag)$\\

\end{proof}

\begin{lemma}
\label{infinity1}

For the construction of $(\overline{E},\overline{B})$ in the base frame $S$, transferred from the limit $(\overline{E}_{\infty},\overline{B}_{\infty})$ in the frame $S_{infty}$, see Lemmas \ref{polynomial},\ref{bounded},\ref{approxinfinity} and \ref{triangles}, we obtain a set of equations $(******)$, as in the proof of the Lemma, for the quantities $\{\sigma,g_{i},p_{jk}\}$, $1\leq i\leq 3$, $1\leq j\leq k\leq 3$, valid for $\overline{u}$ real, with $u_{1}\neq 0$, $\overline{u}\neq u_{1}\overline{e}_{1}$, and with coefficients as defined in the proof of the Lemma.
\end{lemma}

\begin{proof}

Fixing $\overline{u}$ real, with $u_{1}\neq 0$ and $\overline{u}\neq u_{1}\overline{e}_{1}$, choose $g\in SO(3)$ with $g(\overline{e}_{1})={1\over (u_{2}^{2}+u_{3}^{2})^{1\over 2}}(0,u_{2},u_{3})$, so that, for $v\in\mathcal{R}$,  $g(iv(u_{2}^{2}+u_{3}^{2})^{1\over 2}\overline{e}_{1})=iv(0,u_{2},u_{3})$. By the result of Lemma \ref{preserved}, we have that;\\

$B_{iv(0,u_{2},u_{3})}=R_{g}B_{iv(u_{2}^{2}+u_{3}^{2})^{1\over 2}\overline{e}_{1}}R_{g}^{-1}$\\

and, by Lemma \ref{existence}, there exists a unique $\overline{w}$, with $\overline{w}\rightarrow iv(0,u_{2},u_{3})$, such that;\\

$B_{\overline{w}}B_{\overline{u}}=R_{h}B_{v\overline{e}_{1}}$\\

where $h\in SG(3)$. Let $(\overline{E}_{v},\overline{B}_{v})$, be fields in the frames $S_{v}$, travelling with velocity $v\overline{e}_{1}$ relative to $S$, such that $div(\overline{E}_{v}\times\overline{B}_{v})=0$, then, in the rotated frames $R_{h}(S_{v})$, we have, by Lemma \ref{preserved}, that $div(\overline{E}_{v}^{h}\times \overline{B}_{v}^{h})=0$, and this property is preserved in the limit frame $R_{h_{\infty}}(S_{\infty})$, see Lemma \ref{limit} and Lemmas \ref{polynomial},\ref{bounded},\ref{approxinfinity}, \ref{triangles}, so that $div(\overline{E}_{\infty}^{h_{\infty}}\times \overline{B}_{\infty}^{h_{\infty}})=0$ as well. By the same argument, in the rotated frame, $R_{g^{-1}}R_{h_{\infty}}(S_{\infty})$, we have that $div(\overline{E}_{\infty}^{g^{-1}h_{\infty}}\times \overline{B}_{\infty}^{g^{-1}h_{\infty}})=0$. Following the same argument as above, and taking compatible square roots, we have;\\

$lim_{v\rightarrow\infty}\gamma=lim_{v\rightarrow\infty}\alpha_{iv}=lim_{v\rightarrow\infty}\beta_{iv}=lim_{v\rightarrow\infty}\delta_{iv}=lim_{v\rightarrow\infty}\xi_{iv}
=lim_{v\rightarrow\infty}\eta_{iv}=0$\\

$lim_{v\rightarrow\infty}\gamma_{iv}=lim_{v\rightarrow\infty}\theta_{iv}=lim_{iv\rightarrow\infty}{\gamma\beta\over c}$\\

$=lim_{v\rightarrow\infty}{iv\over c^{2}\sqrt{1+{v^{2}\over c^{2}}}}$\\

$={i\over c^{2}}lim_{v\rightarrow\infty}{1\over \sqrt{{1\over v^{2}}+{1\over c^{2}}}}$\\

$={i\over c^{2}}c={i\over c}$\\

$lim_{v\rightarrow\infty}\epsilon_{iv}=lim_{v\rightarrow\infty}{\beta^{2}iv\gamma^{3}\over c^{2}}$\\

$=lim_{v\rightarrow\infty}{-iv^{3}\over c^{4}(1+{v^{2}\over c^{2}})^{3\over 2}}$\\

$={-i\over c^{4}}lim_{v\rightarrow\infty}{1\over({1\over v^{2}}+{1\over c^{2}})^{3\over 2}}$\\

$={-i\over c^{4}}c^{3}={-i\over c}$\\

Taking the limit again of the first equation in $(\dag)$ of Lemma \ref{families}, with $iv$ replacing $v$, see Lemma \ref{infinity}, we obtain, for the transformed quantities in the limit frame $S'$, connected to $R_{g^{-1}}R_{h_{\infty}}(S_{\infty})$ as the limit of boosts with velocity vector $iv(0,u_{2},u_{3})$ ;\\

${i\over c}{\partial p_{11,\infty}\over \partial x}-{i\over c}{\partial g_{1,\infty}\over \partial t}+{i\over c}(f_{1,\infty}+{\partial g_{1,\infty}\over \partial t})=0$\\

which simplifies again to;\\

$f_{1,\infty}=-{\partial p_{11,\infty}\over \partial x}$ $(\dag)$\\

Following the same argument as Lemma \ref{infinity}, and using Maxwell's equations, we have;\\

${\partial\over\partial t}(\overline{E}_{\infty'}\times\overline{B}_{\infty'})_{1}=-{1\over\epsilon_{0}}({\partial p_{12,\infty}\over \partial y}+{\partial p_{13,\infty}\over \partial z})$ $(\dag\dag)$\\

for the transformed fields $(\overline{E}_{\infty'},\overline{B}_{\infty'})$ in $S'$. Let $S''$ be the frame connected to $S'$ by the relation $R_{g^{-1}}(S'')=S'$, where $R_{g^{-1}}(0,u_{2},u_{3})=w\overline{e}_{1}$ and $w=(u_{2}^{2}+u_{3}^{2})^{1\over 2}$. Let $R_{f}(\overline{e}_{1})={\overline{u}\over u}$, where $u=(u_{1}^{2}+u_{2}^{2}+u_{3}^{2})^{1\over 2}$. By the same conjugation result, we have that;\\

$R_{f}B_{u\overline{e}_{1}}R_{f}^{-1}=B_{\overline{u}}$\\

Let $S'''$ be the frame connected to $S''$ by the relation $R_{f}(S''')=S''$, the derivatives on $S'$ transform to $S'''$ by the relations;\\

${\partial\over \partial t}\mapsto {\partial\over \partial t}$\\

${\partial\over \partial y}\mapsto (R_{g^{-1}}R_{f})^{-1}({\partial\over \partial y})=R_{f}^{-1}R_{g}({\partial\over \partial y})$\\

${\partial\over \partial z}\mapsto (R_{g^{-1}}R_{f})^{-1}({\partial\over \partial z})=R_{f}^{-1}R_{g}({\partial\over \partial z})$ $(\sharp)$\\

We have that;\\

$({(0,u_{2},u_{3})\over w},\overline{u})={w^{2}\over w}=w$\\

Let;\\

$T=\{\overline{\theta}:\overline{\theta}\centerdot u\overline{e}_{1}\}=w$\\

so that, as $f\in SO(3)$, $\theta_{1}u=w$, where $\overline{\theta}=(\theta_{1},\theta_{2},\theta_{3})$, and $R_{f}^{-1}R_{g}(\overline{e}_{1})=\overline{\theta}$, with $|\overline{\theta}|=1$. We have that $R_{g}(\overline{e}_{2})=\overline{v}$, with $\overline{v}\in (0,u_{2},u_{3})^{\perp}$ and $|\overline{v}|=1$. Observing that;\\

$(\overline{v},(u_{1},u_{2},u_{3}))=v_{1}u_{1}$\\

$(\overline{v},{(0,u_{2},u_{3})\over w})=0$\\

Let;\\

$T'=\{\overline{\theta}':\overline{\theta}'\centerdot u\overline{e}_{1}=v_{1}u_{1},\overline{\theta}'\centerdot\overline{\theta}=0\}$\\

so that, again as $f\in SO(3)$, $\theta_{1}'u=v_{1}u_{1}$, $\theta'={v_{1}u_{1}\over u}$, $\overline{\theta}'\centerdot\overline{\theta}=0$, and $R_{f}^{-1}R_{g}(\overline{e}_{2})=\overline{\theta}'$, where $\overline{\theta}'=(\theta_{1}',\theta_{2}',\theta_{3}')$ and $|\overline{\theta}'|=1$. As $\{f,g\}\subset SO(3)$ and $\overline{e}_{1}\times\overline{e}_{2}=\overline{e}_{3}$, we have that $R_{h}^{-1}R_{g}(\overline{e}_{3})=\theta''=\theta\times\theta'$. Transforming from $S'$ to $S'''$, we have, using $(\sharp)$, that;\\

${\partial\over\partial t}((\overline{E}_{\infty'}\times\overline{B}_{\infty'},\overline{e}_{1}))\mapsto {\partial\over\partial t}((\overline{E}_{\infty'''}\times\overline{B}_{\infty'''},\overline{\theta}))$\\

$={\partial\over\partial t}({g_{1}\over \epsilon_{0}}{w\over u}+{g_{2}\over \epsilon_{0}}\alpha+{g_{3}\over \epsilon_{0}}\beta)$\\

$={w\over \epsilon_{0}u}{\partial g_{1}\over \partial t}+{\alpha\over \epsilon_{0}}{\partial g_{2}\over \partial t}+{\beta\over \epsilon_{0}}{\partial g_{3}\over \partial t}$ $(\sharp\sharp)$\\

where $\overline{g}=(g_{1},g_{2},g_{3})=\epsilon_{0}(\overline{E}_{\infty''}\times\overline{B}_{\infty''})$, for the transformed fields $\{\overline{E}_{\infty'''},\overline{B}_{\infty'''}\}$ in $S'''$, and $\overline{\theta}=({w\over u},\alpha,\beta)$.\\

$(\overline{E}_{\infty'},\overline{e}_{1})\mapsto (\overline{E}_{\infty'''},\overline{\theta})$\\

$(\overline{E}_{\infty'},\overline{e}_{2})\mapsto (\overline{E}_{\infty'''},\overline{\theta}')$\\

$(\overline{E}_{\infty'},\overline{e}_{3})\mapsto (\overline{E}_{\infty'''},\overline{\theta}\times\overline{\theta}')$\\

$(\overline{E}_{\infty'},\overline{e}_{1})(\overline{E}_{\infty'},\overline{e}_{2})\mapsto (\overline{E}_{\infty'''},\overline{\theta})(\overline{E}_{\infty'''},\overline{\theta}')$\\

$=({e_{1}w\over u}+e_{2}\alpha+e_{3}\beta)({e_{1}v_{1}u_{1}\over u}+e_{2}\gamma+e_{3}\delta)$\\

$={wv_{1}u_{1}\over u^{2}}e_{1}^{2}+({w\gamma \over u}+{v_{1}u_{1}\alpha\over u})e_{1}e_{2}+({w\delta \over u}+{v_{1}u_{1}\beta\over u})e_{1}e_{3}+\alpha\gamma e_{2}^{2}+(\alpha\delta+\beta\gamma)e_{2}e_{3}+\beta\delta e_{3}^{2}$\\

where $\overline{\theta}'=({v_{1}u_{1}\over u},\gamma,\delta)$, $\overline{E}_{\infty'''}=(e_{1},e_{2},e_{3})$. Using the same reasoning for $\overline{B}_{\infty'}$, we have that;\\

$p_{12,S'}\mapsto {wv_{1}u_{1}\over u^{2}}p_{11}+({w\gamma \over u}+{v_{1}u_{1}\alpha\over u})p_{12}+({w\delta \over u}+{v_{1}u_{1}\beta\over u})p_{13}+\alpha\gamma p_{22}+(\alpha\delta+\beta\gamma)p_{23}+\beta\delta p_{33}$\\

$-{\epsilon_{0}\over 2}{wv_{1}u_{1}\over u^{2}}{2\sigma\over \epsilon_{0}}-{\epsilon_{0}\over 2}\alpha\gamma{2\sigma\over \epsilon_{0}}-{\epsilon_{0}\over 2}\beta\delta{2\sigma\over \epsilon_{0}}$\\

$={wv_{1}u_{1}\over u^{2}}p_{11}+({w\gamma \over u}+{v_{1}u_{1}\alpha\over u})p_{12}+({w\delta \over u}+{v_{1}u_{1}\beta\over u})p_{13}+\alpha\gamma p_{22}+(\alpha\delta+\beta\gamma)p_{23}+\beta\delta p_{33}$ $(\sharp\sharp\sharp)$\\

as ${wv_{1}u_{1}\over u^{2}}+\alpha\gamma+\beta\delta=0$, where $(p_{ij})_{1\leq i\leq j\leq 3}$ are the components of the stress tensor and $\sigma$ is the energy term in $S'''$. We have that;\\

$(\overline{E}_{\infty'},\overline{e}_{1})(\overline{E}_{\infty'},\overline{e}_{3})\mapsto (\overline{E}_{\infty'''},\overline{\theta})(\overline{E}_{\infty'''},\overline{\theta}\times\overline{\theta}')$\\

$=({e_{1}w\over u}+e_{2}\alpha+e_{3}\beta)(e_{1}(\alpha\delta-\beta\gamma)+e_{2}({\beta v_{1}u_{1}\over u}-{\delta w\over u})+e_{3}({\gamma w\over u}-{\alpha v_{1}u_{1}\over u}))$\\

$={w\over u}(\alpha\delta-\beta\gamma)e_{1}^{2}$\\

$+({w\over u}({\beta v_{1}u_{1}\over u}-{\delta w\over u})+\alpha(\alpha\delta-\beta\gamma))e_{1}e_{2}$\\

$+({w\over u}({\gamma w\over u}-{\alpha v_{1}u_{1}\over u})+\beta(\alpha\delta-\beta\gamma))e_{1}e_{3}$\\

$+(\alpha({\beta v_{1}u_{1}\over u}-{\delta w\over u}))e_{2}^{2}$\\

$+(\alpha({\gamma w\over u}-{\alpha v_{1}u_{1}\over u})+\beta ({\beta v_{1}u_{1}\over u}-{\delta w\over u}))e_{2}e_{3}$\\

$+\beta({\gamma w\over u}-{\alpha v_{1}u_{1}\over u})e_{3}^{2}$\\

Again, following the same reasoning as above;\\

$p_{13,S'}\mapsto {w\over u}(\alpha\delta-\beta\gamma)p_{11}$\\

$+({w\over u}({\beta v_{1}u_{1}\over u}-{\delta w\over u})+\alpha(\alpha\delta-\beta\gamma))p_{12}$\\

$+({w\over u}({\gamma w\over u}-{\alpha v_{1}u_{1}\over u})+\beta(\alpha\delta-\beta\gamma))p_{13}$\\

$+(\alpha({\beta v_{1}u_{1}\over u}-{\delta w\over u}))p_{22}$\\

$+(\alpha({\gamma w\over u}-{\alpha v_{1}u_{1}\over u})+\beta ({\beta v_{1}u_{1}\over u}-{\delta w\over u}))p_{23}$\\

$+\beta({\gamma w\over u}-{\alpha v_{1}u_{1}\over u})p_{33}$\\

$-{\epsilon_{0}\over 2}{w\over u}(\alpha\delta-\beta\gamma){2\sigma\over \epsilon_{0}}$\\

$-{\epsilon_{0}\over 2}\alpha({\beta v_{1}u_{1}\over u}-{\delta w\over u}){2\sigma\over \epsilon_{0}}$\\

$-{\epsilon_{0}\over 2}\beta({\gamma w\over u}-{\alpha v_{1}u_{1}\over u}){2\sigma\over \epsilon_{0}}$\\

$={w\over u}(\alpha\delta-\beta\gamma)p_{11}$\\

$+({w\over u}({\beta v_{1}u_{1}\over u}-{\delta w\over u})+\alpha(\alpha\delta-\beta\gamma))p_{12}$\\

$+({w\over u}({\gamma w\over u}-{\alpha v_{1}u_{1}\over u})+\beta(\alpha\delta-\beta\gamma))p_{13}$\\

$+(\alpha({\beta v_{1}u_{1}\over u}-{\delta w\over u}))p_{22}$\\

$+(\alpha({\gamma w\over u}-{\alpha v_{1}u_{1}\over u})+\beta ({\beta v_{1}u_{1}\over u}-{\delta w\over u}))p_{23}$\\

$+\beta({\gamma w\over u}-{\alpha v_{1}u_{1}\over u})p_{33}$ $(\sharp\sharp\sharp\sharp)$\\

as ${w\over u}(\alpha\delta-\beta\gamma)+\alpha({\beta v_{1}u_{1}\over u}-{\delta w\over u})+\beta({\gamma w\over u}-{\alpha v_{1}u_{1}\over u})=0$.\\

By $(\sharp)$, we have that;\\

$({\partial\over \partial y})_{S'}\mapsto {v_{1}u_{1}\over u}{\partial\over\partial x}+\gamma {\partial\over\partial y}+\delta{\partial \over \partial z}$\\

$({\partial\over \partial z})_{S'}\mapsto (\alpha\delta-\beta\gamma){\partial\over\partial x}+({\beta v_{1}u_{1}\over u}-{\delta w\over u}){\partial\over\partial y}+({\gamma w\over u}-{\alpha v_{1}u_{1}\over u}){\partial\over\partial z}$ $(\sharp\sharp\sharp\sharp\sharp)$\\

Combining, $(\dag\dag)$, $(\sharp\sharp)$,$(\sharp\sharp\sharp)$,$(\sharp\sharp\sharp\sharp)$,$(\sharp\sharp\sharp\sharp\sharp)$, we obtain the relation in the frame $S'''$;\\

${w\over \epsilon_{0}u}{\partial g_{1}\over \partial t}+{\alpha\over \epsilon_{0}}{\partial g_{2}\over \partial t}+{\beta\over \epsilon_{0}}{\partial g_{3}\over \partial t}$\\

$=-{1\over\epsilon_{0}}[{v_{1}u_{1}\over u}{\partial\over\partial x}+\gamma {\partial\over\partial y}+\delta{\partial \over \partial z}]({wv_{1}u_{1}\over u^{2}}p_{11}+({w\gamma \over u}+{v_{1}u_{1}\alpha\over u})p_{12}+({w\delta \over u}+{v_{1}u_{1}\beta\over u})p_{13}+\alpha\gamma p_{22}+(\alpha\delta+\beta\gamma)p_{23}+\beta\delta p_{33})$\\

$-{1\over\epsilon_{0}}[(\alpha\delta-\beta\gamma){\partial\over\partial x}+({\beta v_{1}u_{1}\over u}-{\delta w\over u}){\partial\over\partial y}+({\gamma w\over u}-{\alpha v_{1}u_{1}\over u}){\partial\over\partial z}]({w\over u}(\alpha\delta-\beta\gamma)p_{11}$\\

$+({w\over u}({\beta v_{1}u_{1}\over u}-{\delta w\over u})+\alpha(\alpha\delta-\beta\gamma))p_{12}$\\

$+({w\over u}({\gamma w\over u}-{\alpha v_{1}u_{1}\over u})+\beta(\alpha\delta-\beta\gamma))p_{13}$\\

$+(\alpha({\beta v_{1}u_{1}\over u}-{\delta w\over u}))p_{22}$\\

$+(\alpha({\gamma w\over u}-{\alpha v_{1}u_{1}\over u})+\beta ({\beta v_{1}u_{1}\over u}-{\delta w\over u}))p_{23}$\\

$+\beta({\gamma w\over u}-{\alpha v_{1}u_{1}\over u})p_{33})$ $(\dag\dag\dag)$\\

Let $S''''$ be the frame connected to $S'''$ by the relation $B_{u\overline{e}_{1}}(S'''')=S'''$. Using the formula for the boost matrix, we have that;\\

$({\partial \over \partial x})_{S'''}\mapsto \gamma_{u}({\partial\over\partial x}+{u\over c^{2}}{\partial \over \partial t})$\\

$({\partial \over \partial y})_{S'''}\mapsto {\partial \over \partial y}$\\

$({\partial \over \partial z})_{S'''}\mapsto {\partial \over \partial z}$\\

$({\partial \over \partial t})_{S'''}\mapsto \gamma_{u}({\partial\over\partial t}+u{\partial\over\partial x})$ $(\dag\dag\dag\dag)$\\

and, using the energy stress tensor;\\

$(cg_{1})_{S'''}\mapsto({-{u\gamma_{u}^{2}\over c}})\sigma+(c\gamma_{u}^{2}+{u^{2}\gamma_{u}^{2}\over c})g_{1}-({u\gamma_{u}^{2}\over c})p_{11}$\\

$(cg_{2})_{S'''}\mapsto\gamma_{u}cg_{2}-{\gamma_{u}u\over c}p_{21}$\\

$(cg_{3})_{S'''}\mapsto\gamma_{u}cg_{3}-{\gamma_{u}u\over c}p_{31}$\\

$(p_{11})_{S'''}\mapsto \gamma_{u}^{2}({u^{2}\sigma\over c^{2}}-2ug_{1}+p_{11})$\\

$(p_{12})_{S'''}\mapsto \gamma_{u}(-ug_{2}+p_{12})$\\

$(p_{13})_{S'''}\mapsto \gamma_{u}(-ug_{2}+p_{13})$\\

$(p_{22})_{S'''}\mapsto p_{22}$\\

$(p_{23})_{S'''}\mapsto p_{23}$\\

$(p_{33})_{S'''}\mapsto p_{33}$ $(\dag\dag\dag\dag\dag)$\\

Using the relations $(\dag\dag\dag),(\dag\dag\dag\dag),(\dag\dag\dag\dag\dag)$, we obtain the following relation in $S''''$;\\

${w\over \epsilon_{0}u}\gamma_{u}({\partial\over\partial t}+u{\partial\over\partial x}){1\over c}(({-{u\gamma_{u}^{2}\over c}})\sigma+(c\gamma_{u}^{2}+{u^{2}\gamma_{u}^{2}\over c})g_{1}-({u\gamma_{u}^{2}\over c})p_{11})$\\

$+{\alpha\over \epsilon_{0}u}\gamma_{u}({\partial\over\partial t}+u{\partial\over\partial x}){1\over c}(\gamma_{u}cg_{2}-{\gamma_{u}u\over c}p_{21})$\\

$+{\beta\over \epsilon_{0}u}\gamma_{u}({\partial\over\partial t}+u{\partial\over\partial x}){1\over c}(\gamma_{u}cg_{3}-{\gamma_{u}u\over c}p_{31})$\\

$=-{1\over\epsilon_{0}}[{v_{1}u_{1}\gamma_{u}\over u}({\partial\over\partial x}+{u\over c^{2}}{\partial \over \partial t})+\gamma{\partial \over \partial y}+\delta{\partial \over \partial z}]({wv_{1}u_{1}\over u^{2}}\gamma_{u}^{2}({u^{2}\sigma\over c^{2}}-2ug_{1}+p_{11})+({w\gamma \over u}+{v_{1}u_{1}\alpha\over u})\gamma_{u}(-ug_{2}+p_{12})+({w\delta \over u}+{v_{1}u_{1}\beta\over u})\gamma_{u}(-ug_{2}+p_{13})+\alpha\gamma p_{22}+(\alpha\delta+\beta\gamma)p_{23}+\beta\delta p_{33})$\\

$-{1\over\epsilon_{0}}[(\alpha\delta-\beta\gamma)\gamma_{u}({\partial\over\partial x}+{u\over c^{2}}{\partial \over \partial t})+({\beta v_{1}u_{1}\over u}-{\delta w\over u}){\partial\over\partial y}+({\gamma w\over u}-{\alpha v_{1}u_{1}\over u}){\partial\over\partial z}]$\\

$({w\over u}(\alpha\delta-\beta\gamma)\gamma_{u}^{2}({u^{2}\sigma\over c^{2}}-2ug_{1}+p_{11})+({w\over u}({\beta v_{1}u_{1}\over u}-{\delta w\over u})+\alpha(\alpha\delta-\beta\gamma))\gamma_{u}(-ug_{2}+p_{12})$\\

$+({w\over u}({\gamma w\over u}-{\alpha v_{1}u_{1}\over u})+\beta(\alpha\delta-\beta\gamma))\gamma_{u}(-ug_{3}+p_{13})+(\alpha({\beta v_{1}u_{1}\over u}-{\delta w\over u}))p_{22}$\\

$+(\alpha({\gamma w\over u}-{\alpha v_{1}u_{1}\over u})+\beta ({\beta v_{1}u_{1}\over u}-{\delta w\over u}))p_{23}+\beta({\gamma w\over u}-{\alpha v_{1}u_{1}\over u})p_{33})$ $(*****)$\\

Let;\\

$\lambda_{ij}=(\theta_{i}\theta_{j}'+\theta_{i}'\theta_{j})$ for $1\leq i<j\leq 3$\\

$\mu_{ij}=(\theta_{i}\theta_{j}''+\theta_{i}''\theta_{j})$ for $1\leq i<j\leq 3$\\

$\nu_{ij}=(\theta_{i}'\theta_{j}'+\theta_{i}''\theta_{j}'')$ for $1\leq i\leq j\leq 3$\\

Then, using Lemma \ref{algebra}, we can rearrange $(*****)$, to obtain;\\

${-w^{3}\gamma_{u}^{3}\over \epsilon_{0}uc^{2}}{\partial\sigma\over \partial x}+{w\gamma_{u}^{2}u\nu_{12}\over \epsilon_{0}c^{2}}{\partial\sigma\over \partial y}+{w\gamma_{u}^{2}u\nu_{13}\over \epsilon_{0}c^{2}}{\partial\sigma\over \partial z}+{w\gamma_{u}^{3}(-1+{1\over c^{2}}{u_{1}^{2}\over u^{2}})\over \epsilon_{0}c^{2}}{\partial\sigma\over \partial t}$\\

$+{w\gamma_{u}^{3}(u^{2}+{2w^{2}c^{2}\over u^{2}}-c^{2})\over \epsilon_{0}c^{2}}{\partial g_{1}\over \partial x}-{2w\gamma_{u}^{2}\nu_{12}\over\epsilon_{0}}{\partial g_{1}\over \partial y}-{2w\gamma_{u}^{2}\nu_{13}\over\epsilon_{0}}{\partial g_{1}\over \partial z}+{w\gamma_{u}^{3}(2w^{2}-u^{2}+c^{2})\over \epsilon_{0} uc^{2}}{\partial g_{1}\over \partial t}$\\

$+{u\gamma_{u}^{2}(\theta_{2}-\theta_{1}'\lambda_{12}-\theta_{1}''\mu_{12})\over\epsilon_{0}}{\partial g_{2}\over \partial x}+{u\gamma_{u}(-\theta_{2}'\lambda_{12}-\theta_{2}''\mu_{12})\over\epsilon_{0}}{\partial g_{2}\over \partial y}+{u\gamma_{u}(-\theta_{3}'\lambda_{12}-\theta_{3}''\mu_{12})\over\epsilon_{0}}{\partial g_{2}\over \partial z}$\\

$+{\gamma_{u}^{2}(\theta_{2}-{u^{2}\over c^{2}}(\theta_{1}'\lambda_{12}+\theta_{1}''\mu_{12}))\over\epsilon_{0}}{\partial g_{2}\over \partial t}+{u\gamma_{u}^{2}(\theta_{3}-\theta_{1}'\lambda_{13}-\theta_{1}''\mu_{13})\over\epsilon_{0}}{\partial g_{3}\over \partial x}+{u\gamma_{u}(-\theta_{2}'\lambda_{13}-\theta_{2}''\mu_{13})\over\epsilon_{0}}{\partial g_{3}\over \partial y}$\\

$+{u\gamma_{u}(-\theta_{3}'\lambda_{13}-\theta_{3}''\mu_{13})\over\epsilon_{0}}{\partial g_{3}\over \partial z}+{\gamma_{u}^{2}(\theta_{3}-{u^{2}\over c^{2}}(\theta_{1}'\lambda_{13}+\theta_{1}''\mu_{13}))\over\epsilon_{0}}{\partial g_{3}\over \partial t}+{\gamma_{u}^{3}w(1-{w^{2}\over c^{2}}-{u^{2}\over c^{2}})\over\epsilon_{0}u}{\partial p_{11}\over \partial x}$\\

$+{\gamma_{u}^{2}w\nu_{12}\over\epsilon_{0}u}{\partial p_{11}\over \partial y}+{\gamma_{u}^{2}w\nu_{13}\over\epsilon_{0}u}{\partial p_{11}\over \partial z}-{\gamma_{u}^{3}w^{3}\over\epsilon_{0}c^{2}u^{2}}{\partial p_{11}\over \partial t}+{\gamma_{u}^{2}(-\theta_{2}u^{2}+c^{2}\theta_{1}'\lambda_{12}+c^{2}\theta_{1}''\mu_{12})\over c^{2}\epsilon_{0}}{\partial p_{12}\over \partial x}$\\

$+{\gamma_{u}(\theta_{2}'\lambda_{12}+\theta_{2}''\mu_{12})\over\epsilon_{0}}{\partial p_{12}\over \partial y}+{\gamma_{u}(\theta_{3}'\lambda_{12}+\theta_{3}''\mu_{12})\over\epsilon_{0}}{\partial p_{12}\over \partial z}+{\gamma_{u}^{2}u(-\theta_{2}+\theta_{1}'\lambda_{12}+\theta_{1}''\mu_{12})\over c^{2}\epsilon_{0}}{\partial p_{12}\over \partial t}$\\

$+{\gamma_{u}^{2}(-\theta_{3}u^{2}+c^{2}\theta_{1}'\lambda_{13}+c^{2}\theta_{1}''\mu_{13})\over c^{2}\epsilon_{0}}{\partial p_{13}\over \partial x}+{\gamma_{u}(\theta_{2}'\lambda_{13}+\theta_{2}''\mu_{13})\over\epsilon_{0}}{\partial p_{13}\over \partial y}+{\gamma_{u}(\theta_{3}'\lambda_{13}+\theta_{3}''\mu_{13})\over\epsilon_{0}}{\partial p_{13}\over \partial z}$\\

$+{\gamma_{u}^{2}u(-\theta_{3}+\theta_{1}'\lambda_{13}+\theta_{1}''\mu_{13})\over c^{2}\epsilon_{0}}{\partial p_{13}\over \partial t}+{\theta_{2}\gamma_{u}\nu_{12}\over\epsilon_{0}}{\partial p_{22}\over \partial x}+{\theta_{2}\nu_{22}\over\epsilon_{0}}{\partial p_{22}\over \partial y}+{\theta_{2}\nu_{23}\over\epsilon_{0}}{\partial p_{22}\over \partial z}+{\theta_{2} u \gamma_{u}\nu_{12}\over\epsilon_{0}c^{2}}{\partial p_{22}\over \partial t}$\\

$+{\gamma_{u}(\theta_{1}'\lambda_{23}+\theta_{1}''\mu_{23})\over\epsilon_{0}}{\partial p_{23}\over \partial x}+{(\theta_{2}'\lambda_{23}+\theta_{2}''\mu_{23})\over\epsilon_{0}}{\partial p_{23}\over \partial y}+{(\theta_{3}'\lambda_{23}+\theta_{3}''\mu_{23})\over\epsilon_{0}}{\partial p_{23}\over \partial z}+{\gamma_{u}^{2}u(\theta_{1}'\lambda_{23}+\theta_{1}''\mu_{23})\over c^{2}\epsilon_{0}}{\partial p_{23}\over \partial t}$\\

$+{\theta_{3}\gamma_{u}\nu_{13}\over\epsilon_{0}}{\partial p_{33}\over \partial x}+{\theta_{3}\nu_{23}\over\epsilon_{0}}{\partial p_{33}\over \partial y}+{\theta_{3}\nu_{33}\over\epsilon_{0}}{\partial p_{33}\over \partial z}+{\theta_{3} u\gamma_{u}\over\epsilon_{0} c^{2}}{\partial p_{33}\over \partial t}=0$ $(******)$\\

\end{proof}

\begin{lemma}
\label{algebra}
With notation as in the previous Lemma \ref{infinity1}, we have that;\\

$(\alpha\delta-\beta\gamma)^{2}+{(v_{1}u_{1})^{2}\over u^{2}}+{w^{2}\over u^{2}}=1$\\

\end{lemma}

\begin{proof}

By the fact that $|\theta''|=1$, and the definition of $\theta\times\theta'$, we have that;\\

$(\alpha\delta-\beta\gamma)^{2}+({v_{1}u_{1}\beta\over u}-{w\delta\over u})^{2}+({w\gamma\over u}-{\alpha v_{1}u_{1}\over u})^{2}$\\

$=(\alpha\delta-\beta\gamma)^{2}+{(v_{1}u_{1})^{2}\over u^{2}}(1-{w^{2}\over u^{2}})+{w^{2}\over u^{2}}(1-{{u_{1}v_{2}}^{2}\over u^{2}})-2{w\over u^{2}}(u_{1}v_{1})(\alpha\gamma+\beta\delta)=1$ $(*)$\\

As $\theta$ and $\theta'$ are orthogonal, we have that $\alpha\gamma+\beta\delta+{wv_{1}u_{1}\over u^{2}}=0$. Substituting into $(*)$, we obtain that;\\

$(\alpha\delta-\beta\gamma)^{2}+(v_{1}u_{1})^{2}(1-{w^{2}\over u^{2}})+{w^{2}\over u^{2}}(1-{{u_{1}v_{2}}^{2}\over u^{2}})-2{w\over u}(u_{1}v_{1})(-{wv_{1}u_{1}\over u^{2}})$\\

$=(\alpha\delta-\beta\gamma)^{2}+{(v_{1}u_{1})^{2}\over u^{2}}+{w^{2}\over u^{2}}=1$\\

\end{proof}

\begin{lemma}
\label{newequations}
With notation as in Lemma \ref{infinity}, for the equations $(******)$, we must have that;\\

${\partial\sigma\over \partial y}={\partial\sigma\over \partial z}={\partial g_{1}\over \partial x}={\partial g_{2}\over \partial y}={\partial g_{2}\over \partial z}={\partial g_{2}\over \partial t}={\partial g_{3}\over \partial y}={\partial g_{3}\over \partial z}={\partial g_{3}\over \partial t}={\partial p_{12}\over \partial x}={\partial p_{13}\over \partial x}$\\

$={\partial p_{22}\over \partial y}={\partial p_{22}\over \partial z}={\partial p_{22}\over \partial t}={\partial p_{23}\over \partial y}={\partial p_{23}\over \partial z}={\partial p_{33}\over \partial y}={\partial p_{33}\over \partial z}={\partial p_{33}\over \partial t}=0$\\

\end{lemma}
\begin{proof}
Note that as $u\rightarrow\infty$, $\gamma_{u}\sim {c\over iu}$. We let ${u_{2}\over u_{1}}=\kappa$ and ${u_{3}\over u_{1}}=\lambda$, with $u_{1}\neq 0$, so that $w=su$, with $s=({1+\kappa^{2}\over 1+\kappa^{2}+\lambda^{2}})^{1\over 2}$. The equations $(******)$ in Lemma \ref{infinity} hold for all $u$ with $0<u<c$, and are algebraic, so letting $u\rightarrow \infty$, we obtain that;\\

${-\nu_{12}s\over \epsilon_{0}}{\partial\sigma\over \partial y}-{\nu_{13}s\over \epsilon_{0}}{\partial\sigma\over \partial z}+{sic\over \epsilon_{0}}{\partial g_{1}\over \partial x}-{ic(-\theta_{2}'\lambda_{12}-\theta_{2}''\mu_{12})\over \epsilon_{0}}{\partial g_{2}\over \partial y}-{ic(-\theta_{3}'\lambda_{12}-\theta_{3}''\mu_{12})\over\epsilon_{0}}{\partial g_{2}\over \partial z}$\\

$+{(\theta_{1}'\lambda_{12}+\theta_{1}''\mu_{12})\over\epsilon_{0}}{\partial g_{2}\over \partial t}-{ic(-\theta_{2}'\lambda_{13}-\theta_{2}''\mu_{13})\over\epsilon_{0}}{\partial g_{3}\over \partial y}-{ic(-\theta_{3}'\lambda_{13}-\theta_{3}''\mu_{13})\over\epsilon_{0}}{\partial g_{3}\over \partial z}+{(\theta_{1}'\lambda_{13}+\theta_{1}''\mu_{13})\over\epsilon_{0}}{\partial g_{3}\over \partial t}$\\

$+{\theta_{2}\over \epsilon_{0}}{\partial p_{12}\over \partial x}+{\theta_{3}\over \epsilon_{0}}{\partial p_{13}\over \partial x}+{\theta_{2}\nu_{22}\over\epsilon_{0}}{\partial p_{22}\over \partial y}+{\theta_{2}\nu_{23}\over\epsilon_{0}}{\partial p_{22}\over \partial z}-{i\theta_{2}\nu_{12}\over\epsilon_{0}c}{\partial p_{22}\over \partial t}+{(\theta_{2}'\lambda_{23}+\theta_{2}''\mu_{23})\over\epsilon_{0}}{\partial p_{23}\over \partial y}$\\

$+{(\theta_{3}'\lambda_{23}+\theta_{3}''\mu_{23})\over\epsilon_{0}}{\partial p_{23}\over \partial z}+{\theta_{3}\nu_{23}\over\epsilon_{0}}{\partial p_{33}\over \partial y}+{\theta_{3}\nu_{33}\over\epsilon_{0}}{\partial p_{33}\over \partial z}-{i\over\epsilon_{0} c}{\partial p_{33}\over \partial t}=0$ $(*******)$\\

where;\\

$\theta_{1}=s$, $\theta_{2}=\alpha$, $\theta_{3}=\beta$\\

$\theta_{1}'={sv_{1}\over 1+\kappa^{2}}$, $\theta_{2}'=\gamma$, $\theta_{3}'=\delta$\\

$\theta_{1}''=\alpha\delta-\beta\gamma$, $\theta_{2}''={s\beta v_{1}\over \sqrt{1+\kappa^{2}}}-s\delta$, $\theta_{3}''=s\gamma-{s\alpha v_{1}\over \sqrt{1+\kappa^{2}}}$\\

$\lambda_{12}=s\gamma+{s\alpha v_{1}\over \sqrt{1+\kappa^{2}}}$\\

$\lambda_{13}=s\delta+{s\beta v_{1}\over \sqrt{1+\kappa^{2}}}$\\

$\lambda_{23}=\alpha\delta+\beta\gamma$\\

$\mu_{12}=s({s\beta v_{1}\over \sqrt{1+\kappa^{2}}}-s\delta)+\alpha(\alpha\delta-\beta\gamma)$\\

$\mu_{13}=s(s\gamma-{s\alpha v_{1}\over \sqrt{1+\kappa^{2}}})+\beta(\alpha\delta-\beta\gamma)$\\

$\mu_{23}=\alpha(s\gamma-{s\alpha v_{1}\over \sqrt{1+\kappa^{2}}})+\beta({s\beta v_{1}\over \sqrt{1+\kappa^{2}}}-s\delta)$\\

$\nu_{11}={v_{1}^{2}s^{2}\over 1+\kappa^{2}}+(\alpha\delta-\beta\gamma)^{2}$\\

$\nu_{12}={s\gamma v_{1}\over \sqrt{1+\kappa^{2}}}+(\alpha\delta-\beta\gamma)({s\beta v_{1}\over \sqrt{1+\kappa^{2}}}-s\delta)$\\

$\nu_{13}={s\delta v_{1}\over \sqrt{1+\kappa^{2}}}+(\alpha\delta-\beta\gamma)(s\gamma-{s\alpha v_{1}\over \sqrt{1+\kappa^{2}}})$\\

$\nu_{22}=\gamma^{2}+({s\beta v_{1}\over \sqrt{1+\kappa^{2}}}-s\delta)^{2}$\\

$\nu_{23}=\gamma\delta+({s\beta v_{1}\over \sqrt{1+\kappa^{2}}}-s\delta)(s\gamma-{s\alpha v_{1}\over \sqrt{1+\kappa^{2}}})$\\

$\nu_{33}=\delta^{2}+(s\gamma-{s\alpha v_{1}\over \sqrt{1+\kappa^{2}}})^{2}$\\

and, by the orthonormality relations between $\overline{\theta}$ and $\overline{\theta}'$;\\

$s^{2}+\alpha^{2}+\beta^{2}=1$\\

${v_{1}^{2}\over 1+\kappa^{2}}+\gamma^{2}+\delta^{2}=1$\\

${sv_{1}\over \sqrt{1+\kappa^{2}}}+\alpha\gamma+\beta\delta=0$\\

$|v_{1}|\leq 1$ $(\dag')$\\

Now, take $\kappa=0$, $v_{1}=0$, $0\leq s\leq 1$, $\tau=(1-s^{2})^{1\over 2}$, $\alpha=\tau cos(\theta)$, $\beta=\tau sin(\theta)$, $\gamma=-sin(\theta)$, $\delta=cos(\theta)$, $0\leq \theta<2\pi$, then it is easily verified that the conditions of $(\dag')$ are satisfied. Substituting into the equations $(*******)$, and taking the power series expansions of the functions involving $\theta$, we can equate coefficients in $\{1,\theta,\theta^{2}\}$ respectively, to obtain the following set of equations;\\

$(i)$ ${s^{2}\tau\over \epsilon_{0}}{\partial \sigma\over \partial y}+{sic\over \epsilon_{0}}{\partial g_{1}\over \partial x}-{sic(\tau^{2}-s^{2})\over \epsilon_{0}}{\partial g_{2}\over \partial y}+{\tau(\tau^{2}-s^{2})\over \epsilon_{0}}{\partial g_{2}\over \partial t}+{sic\over\epsilon_{0}}{\partial g_{3}\over \partial z}+{\tau\over\epsilon_{0}}{\partial p_{12}\over \partial x}+{s^{2}\tau\over\epsilon_{0}}{\partial p_{22}\over \partial y}$\\

$+{is\tau^{2}\over c\epsilon_{0}}{\partial p_{22}\over \partial t}+{\tau\over\epsilon_{0}}{\partial p_{23}\over \partial z}-{i\over c\epsilon_{0}}{\partial p_{33}\over \partial t}=0$\\

$(ii)$ ${s^{2}\tau\over \epsilon_{0}}{\partial \sigma\over \partial z}-{ics(1+(\tau^{2}-s^{2}))\over \epsilon_{0}}{\partial g_{2}\over \partial z}-{ics(1+(\tau^{2}-s^{2}))\over \epsilon_{0}}{\partial g_{3}\over \partial y}+{\tau(\tau^{2}-s^{2})\over \epsilon_{0}}{\partial g_{3}\over \partial t}+{\tau\over \epsilon_{0}}{\partial p_{13}\over \partial x}$\\

$+{\tau(s^{2}-1)\over \epsilon_{0}}{\partial p_{22}\over \partial z}+{\tau(2s^{2}-1)\over \epsilon_{0}}{\partial p_{23}\over \partial y}+{\tau\over \epsilon_{0}}{\partial p_{33}\over \partial z}=0$\\

$(iii)$ ${-s^{2}\tau\over 2\epsilon_{0}}{\partial \sigma\over \partial y}+{sic(1+(\tau^{2}-s^{2}))\over \epsilon_{0}}{\partial g_{2}\over \partial y}-{\tau(\tau^{2}-s^{2})\over 2\epsilon_{0}}{\partial g_{2}\over \partial t}-{sic(1+(\tau^{2}-s^{2}))\over \epsilon_{0}}{\partial g_{3}\over \partial z}-{\tau\over 2\epsilon_{0}}{\partial p_{12}\over \partial x}$\\

$+{2\tau-3\tau s^{2}\over 2\epsilon_{0}}{\partial p_{22}\over \partial y}-{si\tau^{2}\over c\epsilon_{0}}{\partial p_{22}\over \partial t}+{\tau(4s^{2}-1)\over 2\epsilon_{0}}{\partial p_{23}\over \partial z}+{\tau(s^{2}-1)\over \epsilon_{0}}{\partial p_{33}\over \partial y}=0$\\

Using $\tau^{2}=1-s^{2}$, we can simplify $(i)$ to obtain;\\

 ${s^{2}\tau\over \epsilon_{0}}({\partial \sigma\over \partial y}+{\partial p_{22}\over \partial y})+{sic\over \epsilon_{0}}({\partial g_{1}\over \partial x}+{\partial g_{3}\over \partial z})-{sic(1-2s^{2})\over \epsilon_{0}}{\partial g_{2}\over \partial y}+{\tau(1-2s^{2})\over \epsilon_{0}}{\partial g_{2}\over \partial t}$\\

 $+{\tau\over\epsilon_{0}}({\partial p_{12}\over \partial x}+{\partial p_{23}\over \partial z})+{is(1-s^{2})\over c\epsilon_{0}}{\partial p_{22}\over \partial t}-{i\over c\epsilon_{0}}{\partial p_{33}\over \partial t}=0$ $(*)$\\

 We can write $(*)$ in the form $\sum_{1=1}^{7}\lambda_{i}\mu_{i}=0$, where;\\

$\mu_{1}=({\partial \sigma\over \partial y}+{\partial p_{22}\over \partial y})$, $\mu_{2}={\partial g_{1}\over \partial x}+{\partial g_{3}\over \partial z}$, $\mu_{3}={\partial g_{2}\over \partial y}$, $\mu_{4}={\partial g_{2}\over \partial t}$, $\mu_{5}={\partial p_{12}\over \partial x}+{\partial p_{23}\over \partial z}$\\

$\mu_{6}={\partial p_{22}\over \partial t}$, $\mu_{7}={\partial p_{33}\over \partial t}$\\

Using Newton's expansion of;\\

$\tau=(1-s^{2})^{1\over 2}=1-{s^{2}\over 2}-{s^{4}\over 8}-{s^{6}\over 16}+O(s^{8})$\\

and equating coefficients up to $s^{6}$ to zero, we obtain the following equations;\\

$a(1)$. $\mu_{4}+\mu_{5}-{i\over c}\mu_{7}=0$\\

$b(s)$. $ic(\mu_{2}+\mu_{3})+{i\over c}\mu_{6}=0$\\

$c(s^{2})$. $\mu_{1}-{5\over 2}\mu_{4}-\mu_{5}=0$\\

$d(s^{3})$. $-2ic\mu_{3}-{i\over c}\mu_{6}=0$\\

$e(s^{4})$. ${-1\over 2}\mu_{1}+{7\over 8}\mu_{4}-{1\over 8}\mu_{5}=0$\\

$f(s^{6})$. ${-1\over 8}\mu_{1}+{3\over 16}\mu_{4}=0$\\

Using $c,e,f$, and solving the three simultaneous equations, we obtain that $\mu_{1}=\mu_{4}=\mu_{5}=0$. From $a$, we then obtain that $\mu_{7}=0$. Using $b,d$ and eliminating $\mu_{6}$, we obtain $\mu_{2}=\mu_{3}$ and $\mu_{6}=-2c^{2}\mu_{3}$, so we obtain;\\

${\partial \sigma\over \partial y}=-{\partial p_{22}\over \partial y}$, ${\partial p_{12}\over \partial x}=-{\partial p_{23}\over \partial z}$, ${\partial p_{22}\over \partial t}=-2c^{2}{\partial g_{2}\over \partial y}$\\

${\partial g_{2}\over \partial t}={\partial p_{33}\over \partial t}=0$\\

${\partial g_{1}\over \partial x}+{\partial g_{3}\over \partial z}={\partial g_{2}\over \partial y}$ $(\sharp)$\\

Using $\tau^{2}=1-s^{2}$ again, we can simplify $(ii)$ to obtain;\\

$(ii)$ ${s^{2}\tau\over \epsilon_{0}}{\partial \sigma\over \partial z}-{ics(2-2s^{2}))\over \epsilon_{0}}({\partial g_{2}\over \partial z}+{\partial g_{3}\over \partial y})+{\tau(1-2s^{2})\over \epsilon_{0}}({\partial g_{3}\over \partial t}-{\partial p_{23}\over \partial y})+{\tau\over \epsilon_{0}}({\partial p_{13}\over \partial x}-{\partial p_{33}\over \partial z})$\\

$+{\tau(s^{2}-1)\over \epsilon_{0}}{\partial p_{22}\over \partial z}=0$ $(\dag)$\\

We can write $(\dag)$ in the form $\sum_{i=1}^{5}\lambda_{i}\mu_{i}=0$, $(\dag\dag)$ where;\\

$\mu_{1}={\partial \sigma\over \partial z}$, $\mu_{2}={\partial g_{2}\over \partial z}+{\partial g_{3}\over \partial y}$, $\mu_{3}={\partial g_{3}\over \partial t}-{\partial p_{23}\over \partial y}$, $\mu_{4}={\partial p_{13}\over \partial x}-{\partial p_{33}\over \partial z}$, $\mu_{5}={\partial p_{22}\over \partial z}$\\

where $\lambda_{1}+\lambda_{3}+\lambda_{5}=0$ and $\lambda_{1}-\lambda_{4}-\lambda_{5}=0$. Using the first relation in $(\sharp)$, we can write $(\dag\dag)$ in the form;\\

$\lambda_{1}(\mu_{1}-\mu_{5})+\lambda_{2}\mu_{2}+\lambda_{3}(\mu_{3}-\mu_{5})+\lambda_{4}\mu_{4}=0$, $(\dag\dag\dag)$\\

Using the second relation in $(\sharp)$, we have that $2\lambda_{1}=\lambda_{4}-\lambda_{3}$, so that multiplying $(\dag\dag\dag)$ by $2$, and substituting, we obtain the relation;\\

$2\lambda_{2}\mu_{2}+\lambda_{3}(2(\mu_{3}-\mu_{5})-(\mu_{1}-\mu_{5}))+\lambda_{4}(2\mu_{4}+(\mu_{1}-\mu_{5}))=0$ $(\dag\dag\dag\dag)$\\

By taking a power series expansion of $(1-s^{2})^{1\over 2}$ or otherwise, it is easily checked that varying the coefficients $(2\lambda_{2},\lambda_{3},\lambda_{4})$, with $0\leq s\leq 1$, gives that;\\

$\mu_{2}=2(\mu_{3}-\mu_{5})-(\mu_{1}-\mu_{5})=2\mu_{4}+(\mu_{1}-\mu_{5})=0$\\

so we obtain;\\

${\partial g_{2}\over \partial z}=-{\partial g_{3}\over \partial y}$\\

${\partial \sigma\over \partial z}=2({\partial g_{3}\over \partial t}-{\partial p_{23}\over \partial y})-{\partial p_{22}\over \partial z}=-(2({\partial p_{13}\over \partial x}-{\partial p_{33}\over \partial z})-{\partial p_{22}\over \partial z})$\\

${\partial p_{22}\over \partial z}=({\partial g_{3}\over \partial t}-{\partial p_{23}\over \partial y})+({\partial p_{13}\over \partial x}-{\partial p_{33}\over \partial z})$\\

${\partial \sigma\over \partial z}=({\partial g_{3}\over \partial t}-{\partial p_{23}\over \partial y})-({\partial p_{13}\over \partial x}-{\partial p_{33}\over \partial z})$ $(\sharp\sharp)$\\

Using $\tau^{2}=1-s^{2}$ again, we can simplify $(iii)$ to obtain;\\

${-s^{2}\tau\over 2\epsilon_{0}}{\partial \sigma\over \partial y}+{sic(2-2s^{2}))\over \epsilon_{0}}({\partial g_{2}\over \partial y}-{\partial g_{3}\over \partial z})-{\tau(1-2s^{2})\over 2\epsilon_{0}}{\partial g_{2}\over \partial t}-{\tau\over 2\epsilon_{0}}{\partial p_{12}\over \partial x}$\\

$+{2\tau-3\tau s^{2}\over 2\epsilon_{0}}{\partial p_{22}\over \partial y}-{si(1-s^{2})\over c\epsilon_{0}}{\partial p_{22}\over \partial t}+{\tau(4s^{2}-1)\over 2\epsilon_{0}}{\partial p_{23}\over \partial z}+{\tau(s^{2}-1)\over \epsilon_{0}}{\partial p_{33}\over \partial y}=0$ $(\dag\dag\dag')$\\

We can write $(\dag\dag\dag')$ in the form $\sum_{i=1}^{8}\lambda_{i}\mu_{i}=0$, $(\dag\dag\dag\dag)$ where;\\

$\mu_{1}={\partial \sigma\over \partial y}$, $\mu_{2}={\partial g_{2}\over \partial y}-{\partial g_{3}\over \partial z}$, $\mu_{3}={\partial g_{2}\over \partial t}$, $\mu_{4}={\partial p_{12}\over \partial x}$, $\mu_{5}={\partial p_{22}\over \partial y}$, $\mu_{6}={\partial p_{22}\over \partial t}$\\

$\mu_{7}={\partial p_{23}\over \partial z}$, $\mu_{8}={\partial p_{23}\over \partial y}$\\

where $4\lambda_{1}-4\lambda_{4}+2\lambda_{8}=0$, $6\lambda_{1}-4\lambda_{4}-2\lambda_{5}=0$, $4\lambda_{1}-\lambda_{4}+\lambda_{7}=0$ and $2\lambda_{2}+\lambda_{3}-\lambda_{4}=0$. Using the first relation of $(\sharp\sharp)$, we can write $(\dag\dag\dag\dag)$ in the form;\\

$\lambda_{1}(\mu_{1}-2\mu_{8})+\sum_{i=2}^{3}\lambda_{i}\mu_{i}+\lambda_{4}(\mu_{4}+2\mu_{8})+\sum_{i=5}^{7}\lambda_{i}\mu_{i}=0$\\

Then, using the second relation of $(\sharp\sharp)$, we can write $(\dag\dag\dag\dag)$ in the form;\\

$\lambda_{1}(\mu_{1}+3\mu_{5}-2\mu_{8})+\sum_{i=2}^{3}\lambda_{i}\mu_{i}+\lambda_{4}(\mu_{4}-2\mu_{5}+2\mu_{8})+\sum_{i=6}^{7}\lambda_{i}\mu_{i}=0$\\

Using the third relation of $(\sharp\sharp)$, we obtain;\\

$\lambda_{1}(\mu_{1}+3\mu_{5}-4\mu_{7}-2\mu_{8})+\sum_{i=2}^{3}\lambda_{i}\mu_{i}+\lambda_{4}(\mu_{4}-2\mu_{5}+\mu_{7}+\mu_{7}+\mu_{8})$\\

$+\lambda_{6}\mu_{6}=0$\\

and, using the fourth relation of $(\sharp\sharp)$, we obtain;\\

$\lambda_{1}(\mu_{1}-2\mu_{3}+3\mu_{5}-4\mu_{7}-2\mu_{8})+\lambda_{2}\mu_{2}+\lambda_{4}(\mu_{3}+\mu_{4}-2\mu_{5}+\mu_{7}$\\

$+\mu_{8})+\lambda_{6}\mu_{6}=0$\\

Using Newton's expansion of $\tau$ again, and equating coefficients up to $s^{2}$, we obtain the equations;\\

$a(1)$. $-{\delta_{4}\over 2}=0$.\\

$b(s)$. $2ic\delta_{2}-{i\over c}\delta_{6}=0$\\

$c(s^{2})$. $-{\delta_{1}\over 2}+{\delta_{4}\over 4}=0$\\

where $\delta_{1}=\mu_{1}-2\mu_{3}+3\mu_{5}-4\mu_{7}-2\mu_{8}$, $\delta_{2}=\mu_{2}$, $\delta_{4}=\mu_{3}+\mu_{4}-2\mu_{5}+\mu_{7}+\mu_{8}$, $\delta_{6}=\mu_{6}$\\

From $b$, we obtain that $\mu_{2}={1\over 2c^{2}}\mu_{6}$, so that;\\

${\partial g_{2}\over \partial y}-{\partial g_{3}\over \partial z}={1\over 2c^{2}}{\partial p_{22}\over \partial t}$, $(****)$\\

From the rearrangement of $(i)$, we had that ${\partial p_{22}\over \partial t}=-2c^{2}{\partial g_{2}\over \partial y}$, so that, using $(****)$, we obtain;\\

${\partial g_{2}\over \partial y}-{\partial g_{3}\over \partial z}=-{\partial g_{2}\over \partial y}$ and $2{\partial g_{2}\over \partial y}={\partial g_{3}\over \partial z}$ $(!)$\\

It follows by symmetry that $2{\partial g_{3}\over \partial z}={\partial g_{2}\over \partial y}$, so that by $(!)$, $4{\partial g_{2}\over \partial y}={\partial g_{2}\over \partial y}$ and, using $(****)$, ${\partial g_{2}\over \partial y}={\partial g_{3}\over \partial z}={\partial p_{22}\over \partial t}=0$. It follows from the second equation in $(\sharp)$ that ${\partial g_{1}\over \partial x}=0$ as well, $(E)$.\\

From $a,c$, we obtain that $\delta_{1}=\delta_{4}=0$, so that $\mu_{1}-2\mu_{3}+3\mu_{5}-4\mu_{7}-2\mu_{8}=0$ and $\mu_{3}+\mu_{4}-2\mu_{5}+\mu_{7}+\mu_{8}=0$. It follows that;\\

${\partial \sigma\over \partial y}-2{\partial g_{2}\over \partial t}+3{\partial p_{22}\over \partial y}-4{\partial p_{23}\over \partial z}-2{\partial p_{23}\over \partial y}=0$\\

${\partial g_{2}\over \partial t}+{\partial p_{12}\over \partial x}-2{\partial p_{22}\over \partial y}+{\partial p_{23}\over \partial z}+{\partial p_{23}\over \partial y}=0$ $(!!!!)$\\

From the third and fourth equations in $(\sharp\sharp)$, we have;\\

${\partial g_{3}\over \partial t}={\partial p_{22}\over \partial z}+{\partial p_{22}\over \partial y}-{\partial p_{13}\over \partial x}+{\partial p_{33}\over \partial z}$\\

${\partial\sigma\over \partial z}={\partial g_{3}\over \partial t}-{\partial p_{23}\over \partial y}-{\partial p_{13}\over \partial x}+{\partial p_{33}\over \partial z}$\\

$={\partial p_{22}\over \partial z}+{\partial p_{22}\over \partial y}-{\partial p_{13}\over \partial x}+{\partial p_{33}\over \partial z}-{\partial p_{23}\over \partial y}-{\partial p_{13}\over \partial x}+{\partial p_{33}\over \partial z}$\\

$=-2{\partial p_{13}\over \partial x}+{\partial p_{22}\over \partial y}+{\partial p_{22}\over \partial z}-{\partial p_{23}\over \partial y}+2{\partial p_{33}\over \partial z}$, $(!!!)$\\

and, from $(!!!!)$, second equation of $(\sharp)$, we have that;\\

${\partial g_{2}\over \partial t}=-{\partial p_{12}\over \partial x}+2{\partial p_{22}\over \partial y}-{\partial p_{23}\over \partial y}-{\partial p_{23}\over \partial z}=0$\\

${\partial\sigma\over \partial y}=2{\partial g_{2}\over \partial t}-3{\partial p_{22}\over \partial y}+4{\partial p_{23}\over \partial z}+2{\partial p_{23}\over \partial y}$

$=2(-{\partial p_{12}\over \partial x}+2{\partial p_{22}\over \partial y}-{\partial p_{23}\over \partial y}-{\partial p_{23}\over \partial z})-3{\partial p_{22}\over \partial y}+4{\partial p_{23}\over \partial z}+2{\partial p_{23}\over \partial y}$\\

$=-2{\partial p_{12}\over \partial x}-7{\partial p_{22}\over \partial y}+2{\partial p_{23}\over \partial z}$, $(!!!!!)$\\

We also note, from the fact that ${\partial g_{2}\over \partial t}=0$ in $(\sharp)$, that ${\partial g_{3}\over \partial t}=0$ by symmetry. Rewriting the equations $(*******)$ of the previous lemma in terms of the stress tensor, using the above relations, we obtain;\\

${-\nu_{12}s\over \epsilon_{0}}(-2{\partial p_{12}\over \partial x}-7{\partial p_{22}\over \partial y}+2{\partial p_{23}\over \partial z})$\\

$-{\nu_{13}s\over \epsilon_{0}}(-2{\partial p_{13}\over \partial x}+{\partial p_{22}\over \partial y}+{\partial p_{22}\over \partial z}-{\partial p_{23}\over \partial y}+2{\partial p_{33}\over \partial z})$\\

$-{ic(-\theta_{3}'\lambda_{12}-\theta_{3}''\mu_{12})\over\epsilon_{0}}{\partial g_{2}\over \partial z}-{ic(-\theta_{2}'\lambda_{13}-\theta_{2}''\mu_{13})\over\epsilon_{0}}{-\partial g_{2}\over \partial z}$\\

$+{\theta_{2}\over \epsilon_{0}}{\partial p_{12}\over \partial x}+{\theta_{3}\over \epsilon_{0}}{\partial p_{13}\over \partial x}+{\theta_{2}\nu_{22}\over\epsilon_{0}}{\partial p_{22}\over \partial y}+{\theta_{2}\nu_{23}\over\epsilon_{0}}{\partial p_{22}\over \partial z}+{(\theta_{2}'\lambda_{23}+\theta_{2}''\mu_{23})\over\epsilon_{0}}{\partial p_{23}\over \partial y}$\\

$+{(\theta_{3}'\lambda_{23}+\theta_{3}''\mu_{23})\over\epsilon_{0}}{-\partial p_{12}\over \partial x}+{\theta_{3}\nu_{23}\over\epsilon_{0}}{\partial p_{33}\over \partial y}+{\theta_{3}\nu_{33}\over\epsilon_{0}}{\partial p_{33}\over \partial z}=0$ $(********)$\\

with;\\

$-{\partial p_{12}\over \partial x}+2{\partial p_{22}\over \partial y}-{\partial p_{23}\over \partial y}-{\partial p_{23}\over \partial z}=0$\\

$-2{\partial p_{12}\over \partial x}-7{\partial p_{22}\over \partial y}+2{\partial p_{23}\over \partial z}=-{\partial p_{22}\over \partial y}$ $(A)$\\

by $(!!!!!)$ and the fact that ${\partial \sigma\over \partial y}=-{\partial p_{22}\over \partial y}$ in $(\sharp)$, and the additional relations;\\

$-2{\partial p_{13}\over \partial x}+{\partial p_{22}\over \partial y}+{\partial p_{22}\over \partial z}-{\partial p_{23}\over \partial y}+2{\partial p_{33}\over \partial z}=-{\partial p_{33}\over \partial z}$\\

 $-{\partial p_{13}\over \partial x}+2{\partial p_{33}\over \partial z}-{\partial p_{23}\over \partial z}-{\partial p_{23}\over \partial y}=0$ $(B)$\\

which we obtain by symmetry, see Lemma \ref{symmetrylemma}, ${\partial \sigma\over \partial z}=-{\partial p_{33}\over \partial z}$, from ${\partial \sigma\over \partial y}=-{\partial p_{22}\over \partial y}$, and from $-{\partial p_{12}\over \partial x}+2{\partial p_{22}\over \partial y}-{\partial p_{23}\over \partial y}-{\partial p_{23}\over \partial z}=0$.\\

Eliminating $\{{\partial p_{12}\over \partial x},{\partial p_{13}\over \partial x},{\partial p_{22}\over \partial y},{\partial p_{23}\over \partial z}\}$, from $(A),(B)$, we obtain the equations;\\

${\partial p_{12}\over \partial x}={1\over 12}({-\partial p_{33}\over \partial z}+{\partial p_{22}\over \partial z}-8{\partial p_{23}\over \partial y})$\\

${\partial p_{13}\over \partial x}={1\over 12}(19{\partial p_{33}\over \partial z}+5{\partial p_{22}\over \partial z}-4{\partial p_{23}\over \partial y})$\\

${\partial p_{22}\over \partial y}={1\over 12}(2{\partial p_{33}\over \partial z}-2{\partial p_{22}\over \partial z}+4{\partial p_{23}\over \partial y})$\\

${\partial p_{23}\over \partial z}={1\over 12}(5{\partial p_{33}\over \partial z}-5{\partial p_{22}\over \partial z}+4{\partial p_{23}\over \partial y})$ $(C)$\\

Substituting into $(********)$, we obtain;\\

${ic(\theta_{3}'\lambda_{12}+\theta_{3}''\mu_{12}-\theta_{2}'\lambda_{13}-\theta_{2}''\mu_{13})\over \epsilon_{0}}{\partial g_{2}\over \partial z}$\\

$+{\theta_{3}\nu_{23}\over\epsilon_{0}}{\partial p_{33}\over \partial y}$\\

$+{({\nu_{12}s\over 6}+\nu_{13}s-{\theta_{2}\over 12}+{19\theta_{3}\over 12}+{\theta_{3}'\lambda_{23}\over 12}+{\theta_{3}''\mu_{23}\over 12}+{\theta_{2}\nu_{22}\over 6}+\theta_{3}\nu_{33})\over\epsilon_{0}}{\partial p_{33}\over \partial z}$\\

$+{(-{\nu_{12}s\over 6}+{\theta_{2}\over 12}+{5\theta_{3}\over 12}-{\theta_{3}'\lambda_{23}\over 12}-{\theta_{3}''\mu_{23}\over 12}-{\theta_{2}\nu_{22}\over 6}+\theta_{2}\nu_{23})\over\epsilon_{0}}{\partial p_{22}\over \partial z}$\\

$+{({-2\nu_{12}s\over 3}-{2\theta_{2}\over 3}-{\theta_{3}\over 3}+{2\theta_{3}'\lambda_{23}\over 3}+{2\theta_{3}''\mu_{23}\over 3}+{\theta_{2}\nu_{22}\over 3}+\theta_{2}'\lambda_{23}+\theta_{2}''\mu_{23})\over\epsilon_{0}}{\partial p_{23}\over \partial y}=0$\\

Taking $\kappa=0$, $v_{1}=0$, $s=0$, $\tau=(1-s^{2})^{1\over 2}=1$, $\alpha=\tau cos(\theta)$, $\beta=\tau sin(\theta)$, $\gamma=-sin(\theta)$, $\delta=cos(\theta)$, for $\theta\in \{0,{\pi\over 4},{3\pi\over 4},{\pi\over 2}\}$ respectively, noting that the first coefficient in ${\partial g_{2}\over \partial z}$ is zero, we obtain the coefficient matrix $(A)_{ij}$, for $1\leq i\leq j\leq 4$ for the remaining $4$ variables $\{{\partial p_{33}\over \partial y},{\partial p_{33}\over \partial z},{\partial p_{22}\over \partial z},{\partial p_{23}\over \partial y}\}$, where;\\

$a_{11}=0, a_{12}=0, a_{13}=0, a_{14}={2\over 3\epsilon_{0}}$\\

$a_{21}={-1\over 2\sqrt{2}\epsilon_{0}}, a_{22}={13\over 6\sqrt{2}\epsilon_{0}}, a_{23}={-7\over 12\sqrt{2}\epsilon_{0}}, a_{24}={1\over 6\sqrt{2}\epsilon_{0}}$\\

$a_{31}={1\over 2\sqrt{2}\epsilon_{0}}, a_{32}={2\over \sqrt{2}\epsilon_{0}}, a_{33}={5\over 12\sqrt{2}\epsilon_{0}}, a_{34}={-13\over 6\sqrt{2}\epsilon_{0}}$\\

$a_{41}=0, a_{42}={19\over 12\epsilon_{0}}, a_{43}=0, a_{44}={-1\over\epsilon_{0}}$\\

By a straightforward calculation, we have that $det(A)={19\over 72\epsilon_{0}^{4}}\neq 0$, so that;\\

${\partial p_{33}\over \partial y}={\partial p_{33}\over \partial z}={\partial p_{22}\over \partial z}={\partial p_{23}\over \partial y}=0$\\

By $(C)$, we have that;\\

${\partial p_{12}\over \partial x}={\partial p_{13}\over \partial x}={\partial p_{22}\over \partial y}={\partial p_{23}\over \partial z}=0$\\

By $(!!!!!)$, we obtain that;\\

${\partial g_{2}\over \partial t}={\partial \sigma\over \partial y}=0$\\

By $(!!!)$, we have;\\

${\partial g_{3}\over \partial t}={\partial \sigma\over \partial z}=0$\\

By $(E)$, we have;\\

${\partial g_{2}\over \partial y}={\partial g_{3}\over \partial z}={\partial g_{1}\over \partial x}={\partial p_{22}\over \partial t}=0$\\

Substituting these values into $(*******)$ and using the relation;\\

${\partial g_{2}\over \partial z}=-{\partial g_{3}\over \partial y}$ $(F)$\\

we obtain;\\

${ic(\theta_{3}'\lambda_{12}+\theta_{3}''\mu_{12}-\theta_{2}'\lambda_{13}-\theta_{2}''\mu_{13})\over\epsilon_{0}}{\partial g_{2}\over \partial z}=0$, $(G)$\\

Let $s={1\over 2}$, $\tau=(1-s^{2})^{1\over 2}={\sqrt{3}\over 2}$, $\nu_{1}=1$, $1+\kappa^{2}={5\over 3}$, $\rho=(1-{v_{1}^{2}\over 1+\kappa^{2}})={\sqrt{2}\over \sqrt{5}}$, $\lambda^{2}=3+3\kappa^{2}$, $\alpha=\tau cos({\pi\over 8})$, $\beta=\tau sin({\pi\over 8})$, $\gamma=\rho cos({\pi\over 8})$, $\delta=-\rho sin({\pi\over 8})$, then $\alpha\delta-\beta\gamma=-{\sqrt{3}\over 2\sqrt{5}}$ and ${s\nu_{1}\over \sqrt{1+\kappa^{2}}}={\sqrt{3}\over 2\sqrt{5}}$. It is easily verified that the conditions of $(\dag)$ in the Lemma are met. Calculating the coefficient in $(G)$, we obtain that;\\

${-3ic \over 80\epsilon_{0}}{\partial g_{2}\over \partial z}=0$\\

so that, using $(F)$, ${\partial g_{2}\over \partial z}={\partial g_{3}\over \partial y}=0$. This proves the Lemma.

\end{proof}
\begin{lemma}
\label{newequations3}

Using the notation of Lemma \ref{infinity1}, we have that;\\

${\partial\sigma\over \partial x}+3{\partial p_{11}\over \partial x}={\partial g_{1}\over \partial t}+{\partial p_{11}\over \partial x}={\partial p_{12}\over \partial t}-c^{2}{\partial g_{2}\over \partial x}={\partial p_{13}\over \partial t}-c^{2}{\partial g_{3}\over \partial x}=ic{\partial p_{23}\over \partial x}+{\partial p_{23}\over \partial t}$\\

$={\partial g_{1}\over \partial y}={\partial g_{1}\over \partial z}={\partial p_{12}\over \partial y}={\partial p_{12}\over \partial z}={\partial p_{13}\over \partial y}={\partial p_{13}\over \partial z}={\partial p_{22}\over \partial x}={\partial p_{33}\over \partial x}=0$\\

\end{lemma}

\begin{proof}

Using the result of Lemma \ref{newequations}, multiplying the equations $(******)$ in Lemma \ref{infinity1} by $u$,   taking the limit as $u\rightarrow\infty$, and again noting that the equations $(******)$ in Lemma \ref{infinity1} hold for all $u$ with $0<u<c$, and are algebraic, we obtain that;\\

${-ic s^{3}\over\epsilon_{0}}{\partial\sigma\over \partial x}+{2c^{2}\nu_{12}s\over \epsilon_{0}}{\partial g_{1}\over \partial y}+{2c^{2}\nu_{13}s\over \epsilon_{0}}{\partial g_{1}\over \partial z}+{sic(2s^{2}-1)\over\epsilon_{0}}{\partial g_{1}\over \partial t}-{c^{2}(\theta_{2}-\theta_{1}'\lambda_{12}-\theta_{1}''\mu_{12})\over \epsilon_{0}}{\partial g_{2}\over \partial x}$\\

$-{c^{2}(\theta_{3}-\theta_{1}'\lambda_{13}-\theta_{1}''\mu_{13})\over\epsilon_{0}}{\partial g_{3}\over \partial x}-{s(s^{2}+1)ci\over\epsilon_{0}}{\partial p_{11}\over \partial x}-{ic(\theta_{2}'\lambda_{12}+\theta_{2}''\mu_{12})\over\epsilon_{0}}{\partial p_{12}\over \partial y}-{ic(\theta_{3}'\lambda_{12}+\theta_{3}''\mu_{12})\over\epsilon_{0}}{\partial p_{12}\over \partial z}$\\

$-{(-\theta_{2}+\theta_{1}'\lambda_{12}+\theta_{1}''\mu_{12})\over\epsilon_{0}}{\partial p_{12}\over \partial t}-{ic(\theta_{2}'\lambda_{13}+\theta_{2}''\mu_{13})\over \epsilon_{0}}{\partial p_{13}\over \partial y}-{ic(\theta_{3}'\lambda_{13}+\theta_{3}''\mu_{13})\over \epsilon_{0}}{\partial p_{13}\over \partial z}-{(-\theta_{3}+\theta_{1}'\lambda_{13}+\theta_{1}''\mu_{13})\over\epsilon_{0}}{\partial p_{13}\over \partial t}$\\

$-{ic\theta_{2}\nu_{12}\over\epsilon_{0}}{\partial p_{22}\over \partial x}-{ic(\theta_{1}'\lambda_{23}+\theta_{1}''\mu_{23})\over\epsilon_{0}}{\partial p_{23}\over \partial x}-{(\theta_{1}'\lambda_{23}+\theta_{1}''\mu_{23})\over\epsilon_{0}}{\partial p_{23}\over \partial t}-{ic\theta_{3}\nu_{13}\over\epsilon_{0}}{\partial p_{33}\over \partial x}=0$\\

Rearranging, we can write this as;\\

${-ic s^{3}\over\epsilon_{0}}({\partial\sigma\over \partial x}+3{\partial p_{11}\over \partial x})+{2c^{2}\nu_{12}s\over \epsilon_{0}}{\partial g_{1}\over \partial y}+{2c^{2}\nu_{13}s\over \epsilon_{0}}{\partial g_{1}\over \partial z}+{sic(2s^{2}-1)\over\epsilon_{0}}({\partial g_{1}\over \partial t}+{\partial p_{11}\over \partial x})$\\

$-{ic(\theta_{2}'\lambda_{12}+\theta_{2}''\mu_{12})\over\epsilon_{0}}{\partial p_{12}\over \partial y}-{ic(\theta_{3}'\lambda_{12}+\theta_{3}''\mu_{12})\over\epsilon_{0}}{\partial p_{12}\over \partial z}-{(-\theta_{2}+\theta_{1}'\lambda_{12}+\theta_{1}''\mu_{12})\over\epsilon_{0}}({\partial p_{12}\over \partial t}-c^{2}{\partial g_{2}\over \partial x})$\\

$-{ic(\theta_{2}'\lambda_{13}+\theta_{2}''\mu_{13})\over \epsilon_{0}}{\partial p_{13}\over \partial y}-{ic(\theta_{3}'\lambda_{13}+\theta_{3}''\mu_{13})\over \epsilon_{0}}{\partial p_{13}\over \partial z}-{(-\theta_{3}+\theta_{1}'\lambda_{13}+\theta_{1}''\mu_{13})\over\epsilon_{0}}({\partial p_{13}\over \partial t}-c^{2}{\partial g_{3}\over \partial x})$\\

$-{ic\theta_{2}\nu_{12}\over\epsilon_{0}}{\partial p_{22}\over \partial x}-{(\theta_{1}'\lambda_{23}+\theta_{1}''\mu_{23})\over\epsilon_{0}}(ic{\partial p_{23}\over \partial x}+{\partial p_{23}\over \partial t})-{ic\theta_{3}\nu_{13}\over\epsilon_{0}}{\partial p_{33}\over \partial x}=0$ $(*)$\\

Now follow the method of the proof in Lemma \ref{newequations} to obtain the result.\\

\end{proof}

\begin{lemma}
\label{solution}
Let $S$ be surface non-radiating, for real charge and current $(\rho,\overline{J})$, then there exists a complex solution $(\overline{E},\overline{B})$ to Maxwell's equations with $\overline{E}\times\overline{B}=0$.
\end{lemma}

\begin{proof}
By Lemma \ref{newequations}, there exists a complex solution $(\overline{E},\overline{B})$ to Maxwell's equations, with, in components, $\overline{E}\times\overline{B}={1\over\epsilon_{0}}(g_{1},g_{2},g_{3})$, such that;\\

${\partial g_{1}\over \partial x}={\partial g_{2}\over \partial y}={\partial g_{2}\over \partial z}={\partial g_{2}\over \partial t}={\partial g_{3}\over \partial y}={\partial g_{3}\over \partial z}={\partial g_{3}\over \partial t}=0$\\

By Lemma \ref{newequations3}, we also have that;\\

${\partial g_{1}\over \partial y}={\partial g_{1}\over \partial z}={\partial p_{12}\over \partial t}-c^{2}{\partial g_{2}\over \partial x}={\partial p_{13}\over \partial t}-c^{2}{\partial g_{3}\over \partial x}={\partial g_{1}\over \partial t}+{\partial p_{11}\over \partial x}=0$\\

It follows that $g_{1}$ is independent of $\overline{x}\in \mathcal{R}^{3}$. By the boundary condition that $lim_{|\overline{x}|\rightarrow\infty}g_{1}(\overline{x},t)=0$, see Remarks \ref{strongercondition}, for $t\in\mathcal{R}_{\geq 0}$, we have that $g_{1}=0$. Similarly, $g_{2}$ is independent of the coordinates $(y,z,t)$, and we have that;\\

${\partial g_{2}\over \partial x}={1\over c^{2}}{\partial p_{12}\over \partial t}$\\

By Lemma \ref{newequations}, we have that ${\partial p_{12}\over \partial x}=0$. It follows that ${\partial p_{12}\over \partial t}$ is independent of $x$, and, therefore, $g_{2}$ is constant. Again, using the boundary condition that $lim_{|\overline{x}|\rightarrow\infty}g_{2}(\overline{x},t)=0$, for $t\in\mathcal{R}_{\geq 0}$, we have that $g_{2}=0$. Finally, $g_{3}$ is again independent of the coordinates $(y,z,t)$, and we have that;\\

${\partial g_{3}\over \partial x}={1\over c^{2}}{\partial p_{13}\over \partial t}$\\

Again, by Lemma \ref{newequations}, we have that ${\partial p_{13}\over \partial x}=0$. By the same argument, $g_{3}$ is constant, and, using the boundary condition, that $g_{3}=0$.

\end{proof}

\begin{lemma}
\label{complex}
Let $(\overline{E},\overline{B})$ be a complex solution to Maxwell's equations for a real pair $(\rho,\overline{J})$, with $\overline{E}\times\overline{B}=0$, then on the open set $U\subset \mathcal{R}^{4}$, for which $\overline{B}\neq 0$, we have that $\overline{E}=\lambda\overline{B}$, and $\rho|U=0$.\\

\end{lemma}

\begin{proof}
Writing $\overline{E}$ and $\overline{B}$ in components $(e_{1},e_{2},e_{3})$ and $(b_{1},b_{2},b_{3})$, the condition that $\overline{E}\times\overline{B}=0$ amounts to the equations;\\

$e_{2}b_{3}=e_{3}b_{2}$ $e_{3}b_{1}=e_{1}b_{3}$ $e_{1}b_{2}=e_{2}b_{1}$\\

Without loss of generality, we can assume that $\{b_{1},b_{2},b_{3}\}$ are non-vanishing on $U$, to obtain that;\\

${e_{2}\over b_{2}}={e_{3}\over b_{3}}=\lambda$\\

${e_{1}\over b_{1}}={e_{3}\over b_{3}}=\mu$\\

${e_{1}\over b_{1}}={e_{2}\over b_{2}}=\nu$\\

and, clearly then, $\lambda=\mu=\nu$, so that $\overline{E}=\lambda\overline{B}$, $(*)$. We have that the pairs $(Re(\overline{E}),Re(\overline{B}))$ and $(Im(\overline{E}),Im(\overline{B}))$ satisfy Maxwell's equations for the pair $(\rho,\overline{J})$, and in free space. In particularly, $div(Re(\overline{B}))=div(Im(\overline{B}))=0$, $(**)$. From $(*)$, we have that;\\

$\overline{E}=f Re(\overline{B})+g Im(\overline{B})$\\

and, from $(**)$;\\

${\rho\over \epsilon_{0}}=div(\overline{E})=grad(f)\centerdot Re(\overline{B})+grad(g)\centerdot Im(\overline{B})$ $(***)$\\

Let $\gamma$ be a level surface for $f$, with interior $U_{\gamma}$, then by the divergence theorem;\\

$\int_{U_{\gamma}}{\rho\over \epsilon_{0}}d\overline{x}$\\

$=\int_{U_{\gamma}}div(\overline{E})d\overline{x}$\\

$=\int_{\gamma}\overline{E}\centerdot d\overline{S}$\\

$=\int_{\gamma}(f Re(\overline{B})+g Im(\overline{B}))\centerdot d\overline{S}$\\

$=f\int_{\gamma}Re(\overline{B})\centerdot d\overline{S}+\int_{\gamma}g Im(\overline{B})\centerdot d\overline{S}$\\

$=\int_{\gamma}g Im(\overline{B})\centerdot d\overline{S}$\\

$=\int_{U_{\gamma}}div(g Im(\overline{B}))d\overline{x}$\\

$=\int_{U_{\gamma}}grad(g)\centerdot Im(\overline{B})d\overline{x}$\\

By continuity of $f$, we can cover $U$ with open sets of the form $U_{\gamma}$ contained within any ball $B(x_{0},\epsilon)$ of radius $\epsilon>0$, so that we conclude;\\

${\rho\over \epsilon_{0}}=grad(g)\centerdot Im(\overline{B})$ $(****)$\\

Similarly, we can conclude that;\\

${\rho\over \epsilon_{0}}=grad(f)\centerdot Re(\overline{B})$ $(*****)$\\

From $(***),(****),(*****)$, we have that $\rho=0$ on $U$.

\end{proof}

\begin{lemma}
\label{surface}
Let the frame $S$ be surface non-radiating in the sense of \cite{dep1}, for charge and current $(\rho,\overline{J})$, then in any inertial frame $S'$ connected to $S$ by a boost with velocity vector $\overline{v}$, with $|\overline{v}|<c$, we have that $S'$ is surface non-radiating, for the transformed current and charge $(\rho',\overline{J}')$. Moreover, if the frame $S$ is surface non-radiating for charge and current $(\rho,\overline{J})$, then for any $g\in O(3)$, if $(\rho^{g},\overline{J}^{g})$, are the transformed current and charge in the rotated or reflected frame $S'$, then $S'$ is surface non-radiating.

\end{lemma}

\begin{proof}
Let $S''$ be a frame connected to $S'$ by a velocity vector $\overline{w}$. By Lemma \ref{velocity}, we have that $B_{\overline{w}}B_{\overline{v}}=R_{g}B_{\overline{v}*\overline{w}}$, where $g\in SO(3)$. Let $S'''$ be connected to $S$ by the velocity vector $\overline{v}*\overline{w}$, then as $S$ is surface non-radiating, there exist $(\overline{E}''',\overline{B}''')$ satisfying Maxwell's equations in $S'''$ for the transformed charge and current $(\rho''',\overline{J}''')$, with $div'''(\overline{E}'''\times \overline{B}''')=0$. By Lemma \ref{maxwells}, we have that in $S''$, $(\overline{E}'',\overline{B}'')$ satisfy Maxwell's equations for the transformed current and charge $(\rho'',\overline{J}'')$ with $div(\overline{E}''\times \overline{B}'')=0$, where $\overline{E}''=\overline{E}'''^{g}$ and $\overline{B}''=\overline{B}'''^{g}$. As $(\rho,\overline{J})$ transforms as a $4$-vector between inertial frames, and using the result of Lemma \ref{potential}, we have verified the surface non-radiating condition for $S'$, with the transformed current $(\rho',\overline{J}')$. For the last part, let $S''$ be a frame connected to $S'$ by a velocity vector $\overline{v}$ and let $\overline{w}=g^{-1}(\overline{w})$. In the frame $S'''$ connected to $S$ by the velocity vector $\overline{w}$, by the definition of surface non radiating, there exist fields $(\overline{E}''',\overline{B}''')$, satisfying Maxwell's equations in $S'''$, with $\bigtriangledown'''(\overline{E}'''\times\overline{B}''')=0$. Using Lemmas \ref{maxwells}, \ref{potential} and \ref{twist1}, if $(\overline{E}'',\overline{B}'')$ are the fields in $S''$ corresponding to $(\overline{E}^{g},sign(g)\overline{B}^{g})$, where $(\overline{E},\overline{B})$ are the fields in $S$ corresponding to $(\overline{E}''',\overline{B}''')$ in $S'''$, then $\bigtriangledown''\centerdot(\overline{E}''\times \overline{B}'')=0$. Moreover, $(\rho'',\overline{J}'',\overline{E}'',\overline{B}'')$ satisfy Maxwell's equations in $S''$ for the current and charge $(\rho'',\overline{J}'')$ in $S''$ corresponding to $(\rho^{g},\overline{J}^{g})$.   .
\end{proof}

\begin{lemma}
\label{symmetrylemma}
Let $\tau$ be a permutation of $(1,2,3)$, and let $(\rho,\overline{J},\overline{E},\overline{B})$ satisfy Maxwell's equations in $S$. Let $(\overline{E}^{\tau},sign(\tau)\overline{B}^{\tau})$ be the corresponding fields in the
reflected frame $S'$. Let $\{h_{i},h_{i}'\}$, for $1\leq i\leq 3$, $\{p_{ij},p_{ij}'\}$, for $1\leq i\leq j\leq 3$, and $\{\sigma,\sigma'\}$ be the components of the Poynting vector, stress tensor and energies for $(\overline{E},\overline{B})$ and $(\overline{E}^{\tau},sign(\tau)\overline{B}^{\tau})$ respectively, then;\\

$\sigma'=\sigma^{\tau}$ $h_{i}'=h_{\tau(i)}^{\tau}$ $p_{ij}'=p_{\tau(i)\tau(j)}^{\tau}$\\

Let hypotheses be as in Lemma \ref{infinity1}, with the assumption that $S$ is surface non-radiating, for charge and current $(\rho,\overline{J})$, and $(\overline{E}_{\infty},\overline{B}_{\infty})$ are the fields constructed in the limit frame $S_{\infty}$, then the conclusion is satisfied by $(\overline{E}_{\infty}^{\tau_{23}},-\overline{B}_{\infty}^{\tau_{23}})$ in the reflected frame $S_{\infty}^{\tau_{23}}$ for charge and current $(\rho^{\tau_{23}},\overline{J}^{\tau_{23}})$. Moreover, for the permutation $\tau_{23}$ of $(1,2,3,4)$, if, in the context of Lemmas \ref{newequations} and \ref{newequations3}, we derive a relation of the form;\\

$\sum_{1\leq k\leq 4}\alpha_{k}{\partial\sigma\over \partial x_{k}}+\sum_{1\leq i\leq 3,1\leq k\leq 4}\beta_{ik}{\partial g_{i}\over \partial x_{k}}+\sum_{1\leq i\leq j\leq 3,1\leq k\leq 4}\gamma_{ijk}{\partial p_{ij}\over \partial x_{k}}=0$ $(*)$\\

where $\{\alpha_{k},\beta_{ik},\gamma_{ijk}\}\subset \mathcal{C}$, then we can derive the relation;\\

$\sum_{1\leq k\leq 4}\alpha_{k}{\partial\sigma\over \partial x_{\tau_{23}(k)}}+\sum_{1\leq i\leq 3,1\leq k\leq 4}\beta_{ik}{\partial g_{\tau_{23}(i)}\over \partial x_{\tau_{23}(k)}}$\\

$+\sum_{1\leq i\leq j\leq 3,1\leq k\leq 4}\gamma_{ijk}{\partial p_{\tau_{23}(i)\tau_{23}(j)}\over \partial x_{\tau_{23}(k)}}=0$ $(**)$\\

\end{lemma}

\begin{proof}
For the first claim, it is sufficient to prove the result for the elementary permutation $\tau_{23}$. We have that, in components, $(e_{1}',e_{2}',e_{3}')=(e_{1}^{\tau_{23}},e_{2}^{\tau_{23}},e_{3}^{\tau_{23}})$ and $(b_{1}',b_{2}',b_{3}')=(-b_{1}^{\tau_{23}},-b_{3}^{\tau_{23}},-b_{2}^{\tau_{23}})$. Then, by a straightforward calculation;\\

$e'^{2}=e_{1}'^{2}+e_{2}'^{2}+e_{3}'^{2}=(e_{1}^{\tau_{23}})^{2}+(e_{3}^{\tau_{23}})^{2}+(e_{2}^{\tau_{23}})^{2}
=(e^{2})^{\tau_{23}}$\\

$b'^{2}=b_{1}'^{2}+b_{2}'^{2}+b_{3}'^{2}=(-b_{1}^{\tau_{23}})^{2}+(-b_{3}^{\tau_{23}})^{2}+(-b_{2}^{\tau_{23}})^{2}
=(b^{2})^{\tau_{23}}$\\

$\sigma'={1\over\epsilon_{0}}(e'^{2}+c^{2}b'^{2})
={1\over\epsilon_{0}}((e^{2})^{\tau_{23}}+c^{2}(b^{2})^{\tau_{23}})=\sigma^{\tau_{23}}$\\

$h_{1}'=(e_{2}b_{3}-e_{3}b_{2})^{\tau_{23}}=h_{1}^{\tau_{23}}$\\

$h_{2}'=(e_{1}b_{2}-e_{2}b_{1})^{\tau_{23}}=h_{3}^{\tau_{23}}$\\

$h_{3}'=(e_{3}b_{1}-e_{1}b_{3})^{\tau_{23}}=h_{2}^{\tau_{23}}$\\

$p_{ij}'=-\epsilon_{0}(e_{i}'e_{j}'+c^{2}b_{i}'b_{j}')+\delta_{ij}\sigma'$\\

$=-\epsilon_{0}(e_{\tau_{23}(i)}^{\tau_{23}}e_{\tau_{23}(j)}^{\tau_{23}}+c^{2}(-b_{\tau_{23}(i)}^{\tau_{23}})(-b_{\tau_{23}(j)}^{\tau_{23}}))+\delta_{ij}\sigma^{\tau_{23}}$\\

$=-\epsilon_{0}(e_{\tau_{23}(i)}^{\tau_{23}}e_{\tau_{23}(j)}^{\tau_{23}}+c^{2}b_{\tau_{23}(i)}^{\tau_{23}})b_{\tau_{23}(j)}^{\tau_{23}}))+\delta_{\tau_{23}(i)\tau_{23}(j)}\sigma^{\tau_{23}}$\\

$=p_{\tau_{23}(i)\tau_{23}(j)}^{\tau_{23}}$\\

For the second claim, we have by both parts of Lemma \ref{surface}, that the reflected frame $S'^{\tau_{23}}$ corresponding to $S'$ is surface non-radiating for the transformed current and charge $(\rho'^{\tau_{23}},\overline{J}'^{\tau_{23}})$, where $(\rho',\overline{J}')$ correspond to $(\rho,\overline{J})$ in $S$. Let $\{\overline{E}_{r},\overline{B}_{r}\}$ be the fields constructed in Lemma \ref{polynomial}, with corresponding fields $\{\overline{E}_{r}^{\tau_{23}},-\overline{B}_{r}^{\tau_{23}}\}$ in the reflected frames $S_{r}''^{\tau_{23}}$ corresponding to $S_{r}''$, then $\bigtriangledown'''(\overline{E}_{r}^{\tau_{23}}\times -\overline{B}_{r}^{\tau_{23}})=0$ by Lemma \ref{maxwells} and, as $\tau_{23}$ fixes $\overline{e}_{1}$, $S_{r}''^{\tau_{23}}$ is connected to $S''$ by the velocity vectors $-r\overline{e}_{1}$. Moreover, in the notation of Lemma \ref{infinity1}, $\overline{u}$ is fixed, so that when we construct $(\overline{E}_{\infty}',\overline{B}_{\infty}')$ from the fields $\{\overline{E}_{r}^{\tau_{23}},-\overline{B}_{r}^{\tau_{23}}\}$, it is clear, using the fact that the transformations connecting the frames $\{S,S',S_{r}''\}$ with $S_{\infty}$ are the same as those between $\{S^{\tau_{23}},S'^{\tau_{23}},S_{r}''^{\tau_{23}}\}$ and $S_{\infty}^{\tau_{23}}$, that $\overline{E}_{\infty}'=\overline{E}_{\infty}^{\tau_{23}}$ and $\overline{B}_{\infty}'=-\overline{B}_{\infty}^{\tau_{23}}$. For the final claim, let $(\rho,\overline{J},\overline{E},\overline{B})$, be the tuple, satisfying Maxwell's equations in the base frame $S$, for which we derive the relation $(*)$, then $(\rho^{\tau_{23}},\overline{J}^{\tau_{23}},\overline{E}^{\tau_{23}},-\overline{B}^{\tau_{23}})$ satisfies Maxwell's equations in the reflected frame $S^{\tau_{23}}$, and, by Lemma \ref{surface}, $S^{\tau_{23}}$ is surface non-radiating for the reflected charge and current $(\rho^{\tau_{23}},\overline{J}^{\tau_{23}})$. Using the fact that $(\overline{E},\overline{B})$ corresponds to the fields $(\overline{E}_{\infty},\overline{B}_{\infty})$ in the limit frame $S_{\infty}$, by the proof of the second claim, we have that $(\overline{E}_{\infty}^{\tau_{23}},-\overline{B}_{\infty}^{\tau_{23}})$ corresponds to the fields $(\overline{E}_{\infty}^{\tau_{23}},-\overline{B}_{\infty}^{\tau_{23}})$ in the reflected frame $S_{\infty}^{\tau_{23}}$. We can then follow the proof of Lemma \ref{infinity}, to obtain the same relation $(*)$ for the quantities $\{\sigma', (g_{i}')_{1\leq i\leq 3},(p_{ij}')_{1\leq i\leq j\leq 3}\}$,  corresponding to $(\overline{E}^{\tau_{23}},-\overline{B}^{\tau_{23}})$. By the first part of the lemma, we obtain the relation;\\

$\sum_{1\leq k\leq 4}\alpha_{k}{\partial\sigma^{\tau_{23}}\over \partial x_{k}}+\sum_{1\leq i\leq 3,1\leq k\leq 4}\beta_{ik}{\partial g_{\tau_{23}(i)}^{\tau_{23}}\over \partial x_{k}}$\\

$+\sum_{1\leq i\leq j\leq 3,1\leq k\leq 4}\gamma_{ijk}{\partial p_{\tau_{23}(i)\tau_{23}(j)}^{\tau_{23}}\over \partial x_{k}}=0$\\

Using the chain rule, we then obtain that;\\

$\sum_{1\leq k\leq 4}\alpha_{k}{\partial\sigma\over \partial x_{\tau_{23}(k)}}+\sum_{1\leq i\leq 3,1\leq k\leq 4}\beta_{ik}{\partial g_{\tau_{23}(i)}\over \partial x_{\tau_{23}(k)}}$\\

$+\sum_{1\leq i\leq j\leq 3,1\leq k\leq 4}\gamma_{ijk}{\partial p_{\tau_{23}(i)\tau_{23}(j)}\over \partial x_{\tau_{23}(k)}}=0$\\

as required.

\end{proof}

\begin{lemma}
\label{polynomial}
We can construct limit fields $(\overline{E}_{\infty},\overline{B}_{\infty})$, in the limit frame $S_{\infty}$, with $div_{\infty}(\overline{E}_{\infty}\times \overline{E}_{\infty})=0$.
\end{lemma}

\begin{proof}
Let $\epsilon>0$,$\delta<0$, and let $S'$ travel with velocity vector $\overline{u}=(c-\epsilon)\overline{e}_{1}$, relative to $S$, and let $S''$ travel with velocity vector $\overline{w}=(-{c^{2}\over c-\epsilon}-\delta)\overline{e}_{1}$ relative to $S'$. By Lemma \ref{velocity}, there exists $g\in SO(3)$ with;\\

$B_{\overline{w}}B_{\overline{u}}=R_{g}B_{\overline{u}*\overline{w}}$\\

with $\overline{u}*\overline{w}={\overline{u}+\overline{w}\over 1+{\overline{u}\centerdot\overline{w}\over c^{2}}}+{\gamma_{u}\over c^{2}(\gamma_{u}+1)}{\overline{u}\times (\overline{u}\times\overline{w})\over 1+{\overline{u}\centerdot\overline{w}\over c^{2}}}$\\

$={\overline{u}+\overline{w}\over 1+{\overline{u}\centerdot\overline{w}\over c^{2}}}$\\

$={((c-\epsilon)-{c^{2}\over c-\epsilon}-\delta)\overline{e}_{1}\over{-\delta(c-\epsilon)\over c^{2}}}$\\

$={c^{2}((c-\epsilon)-{c^{2}\over c-\epsilon}-\delta)\over -\delta(c-\epsilon)}\overline{e}_{1}$ $(*)$\\

By inspection of $(*)$, for given $0<\epsilon<c$, we can see that, as $\delta\rightarrow 0$ from below, $\overline{u}*\overline{w}\rightarrow\infty$ along the direction $\overline{e}_{1}$. By Lemma \ref{surface}, we can assume that $S'$ is surface non radiating, and there exist a family of electric and magnetic fields $\{\overline{E}_{r},\overline{B}_{r}\}$, with $0\leq r<c$, such that $div(\overline{E}_{r}\times\overline{B}_{r})=0$ in the inertial frames $S''_{r}$, travelling at velocity $-r\overline{e}_{1}$ relative to $S'$. Assume that there exists a uniform polynomial approximation $\{\overline{E}_{r}^{\gamma},\overline{B}_{r}^{\gamma}\}$, with error term $\gamma>0$, to the fields, transferred back to the base frame $S$. By continuity, for sufficiently small $\delta$, we can find a polynomial family $\{\overline{E}_{r}^{\gamma'},\overline{B}_{r}^{\gamma'}\}$, for $0<r<|\overline{w}|$, with error term $\gamma'$, and $0<\gamma'<2\gamma$, such that $|div_{\infty}(\overline{E}_{r}\times\overline{B}_{r})|<\gamma'$. In the limit frame $S_{\infty}$, using Lemma \ref{divcurl}, Lemma \ref{velocity} and Definition \ref{extension}, we obtain that $|div(\overline{E}_{\infty}\times\overline{B}_{\infty})|<\gamma'$ as well.
\end{proof}

\begin{rmk}
\label{referbelow}
A more rigorous proof of the final claim in the previous lemma is given below.
\end{rmk}

\begin{lemma}
\label{bounded}
Let $S$ be a frame with bounded current $(\rho,\overline{J})$ and let $S'_{\epsilon}$ be connected to $S$ by the velocity vector $(c-\epsilon)\overline{e}_{1}$, for $0<\epsilon\leq c$, and $S''_{\epsilon,\delta}$ be connected to $S'_{\epsilon}$ by the velocity vector $(-c+\delta)\overline{e}_{1}$, where $\delta=(1+\tau)\epsilon$, for;\\

$|\tau|\leq {1\over 2}\leq {c\over c-1}\leq{{2\over c}-{\epsilon\over c^{2}}\over 1-{1\over c}+{\epsilon \over c^{2}}}$\\

then the transfers $(\rho_{\epsilon,\delta},\overline{J}_{\epsilon,\delta})$ to $S''_{\epsilon,\delta}$ are uniformly bounded in the frames $S''_{\epsilon,\delta}$. Moreover, we can assume there exists a constant $F$ independent of $\epsilon$, and, for any given $\epsilon$, a family of tuples $(\rho_{\epsilon,\delta},\overline{J}_{\epsilon,\delta},\overline{E}_{\epsilon,\delta},\overline{B}_{\epsilon,\delta})$, satisfying Maxwell's equations, with $div(\overline{E}_{\epsilon,\delta}\times\overline{B}_{\epsilon,\delta})=0$ in the frame $S''_{\epsilon,\delta}$, and $max(|\overline{E}_{\epsilon,\delta}|,|\overline{E}_{\epsilon,\delta}|)\leq F$.

\end{lemma}
\begin{proof}
We have that $S''_{\epsilon,\delta}$ is connected to $S$ by the boost matrix $B_{\overline{u}*\overline{v}}$ where $\overline{u}=(c-\epsilon)\overline{e}_{1}$, $\overline{v}=(-c+\delta)\overline{e}_{1}$ and;\\

$|\overline{u}*\overline{v}|=|{\overline{u}+\overline{v}\over 1+{\overline{u}\centerdot\overline{v}\over c^{2}}}+{\gamma_{u}\over c^{2}(\gamma_{u}+1)}{\overline{u}\times (\overline{u}\times\overline{v})\over 1+{\overline{u}\centerdot\overline{v}\over c^{2}}}|$\\

$=|{\overline{u}+\overline{v}\over 1+{\overline{u}\centerdot\overline{v}\over c^{2}}}|$\\

$=|{(\delta-\epsilon)\overline{e}_{1}\over {\delta+\epsilon\over c}-{\epsilon\delta\over c^{2}}}|$\\

$=|{\epsilon(\theta-1)\overline{e}_{1}\over \epsilon({\theta+1\over c})-{\theta\epsilon^{2}\over c^{2}}}$ $(\theta=1+\tau)|$\\

$=|{(\theta-1)\overline{e}_{1}\over ({\theta+1\over c})-{\theta\epsilon\over c^{2}}}|\leq 1$\\

A straightforward calculation using the transfer rules for $(\rho,\overline{J})$ to frames $S_{\overline{w}}$, connected to $S$ by a velocity vector $\overline{w}$, with $|\overline{w}|\leq 1$, shows that the transfers $(\rho_{\epsilon,\delta},\overline{J}_{\epsilon,\delta})$ are uniformly bounded. For the last part, we can, using the conjecture $(ii)$ in Remark \ref{strongercondition}, assume there exists a family $(\overline{E}_{s\overline{e}_{1}},\overline{B}_{s\overline{e}_{1}})$ on the frames $S_{s\overline{e}_{1}}$, connected to $S$ by the velocity vector $s\overline{e}_{1}$, for $|s|<1$, with $div_{s\overline{e}_{1}}(\overline{E}_{s\overline{e}_{1}}\times \overline{B}_{s\overline{e}_{1}})=0$, such that the transfers $(\overline{E}_{s\overline{e}_{1}}',\overline{B}_{s\overline{e}_{1}}')$ to $S$ form a smooth family on $B(\overline{0},r_{0})\times (0,t_{0})$. By continuity, the transfers are bounded by some constant $F$ as required.         \\

\end{proof}

\begin{lemma}{Polynomial Approximation}
\label{approxinfinity}

For any $\gamma>0$, with $F$ as in Lemma \ref{bounded}, there exists a sequence of pairs $(\overline{E}_{n,\infty},\overline{B}_{n,\infty})$ for $n\geq n_{\gamma}$, in the limit frame $S_{\infty}$, with $div_{S_{\infty}}|(\overline{E}_{n,\infty}\times \overline{B}_{n,\infty})|<\gamma$ and $max(|\overline{E}_{n,\infty}|,|\overline{B}_{n,\infty}|)<F+1$ on some $B(\overline{0},r_{\infty})\times (0,t_{\infty})$.

\end{lemma}
\begin{proof}
Let $S'_{\epsilon}$ be as in Lemma \ref{bounded}, and let $\{\overline{E}'_{\epsilon,\delta},\overline{B}'_{\epsilon,\delta}\}$ be the transfers of the fields $\{\overline{E}_{\epsilon,\delta},\overline{B}_{\epsilon,\delta}\}$, guaranteed by Lemma \ref{bounded}, to $S'_{\epsilon}$. By the proof of Lemma \ref{bounded}, restricted to some $B(\overline{0},r_{\epsilon})\times (0,t_{\epsilon})$, they form a smooth bounded family on $S'_{\epsilon}$,indexed by $-c+\delta\in (-c+{\epsilon\over 2},-c+{3\epsilon\over 2})$. By the Stone-Weierstrass approximation theorem, there exists a uniformly convergent sequence of polynomial approximations $\{\overline{E}'_{n,\epsilon,\delta},\overline{B}'_{n,\epsilon,\delta}\}$ to $\{\overline{E}'_{\epsilon,\delta},\overline{B}'_{\epsilon,\delta}\}$. By choosing the approximating polynomials in the frames $S_{s\overline{e}_{1}}$ from the previous Lemma, using continuity, the fact that the transfer of polynomial fields are polynomial, and formulating the fact that $div_{S''_{\epsilon,\delta}}(\overline{E}'_{\epsilon,\delta}\times\overline{B}'_{\epsilon,\delta})=0$ algebraically in the base frame $S$, see Lemma \ref{geometry}, we can assume that, for any $\gamma>0$, there exists $\{n_{\gamma},\epsilon_{\gamma}\}$, such that $|div_{S''_{\epsilon,\delta'}}(\overline{E}'_{n,\epsilon,\delta'}\times \overline{B}'_{n,\epsilon,\delta'})|<\gamma$, for $-c+\delta'\in (-c-{\epsilon\over 2},-c)$, and $n\geq n_{\gamma}$, restricted to some $B(\overline{0},r_{\epsilon,\delta})\times (0,t_{\epsilon,\delta})$, where $S''_{\epsilon,\delta'}$ is connected to $S_{\epsilon}$ by the velocity vector $(-c+\delta')\overline{e}_{1}$. Taking the limit as $\epsilon\rightarrow 0$, using Lemma \ref{polynomial}, the constant $F$ from Lemma \ref{bounded}, we obtain a sequence of pairs $(\overline{E}_{n,\infty},\overline{B}_{n,\infty})$ in the limit frame $S_{\infty}$, with $div_{S_{\infty}}|(\overline{E}_{n,\infty}\times \overline{B}_{n,\infty})|<\gamma$ and $max(|\overline{E}_{n,\infty}|,|\overline{B}_{n,\infty}|)<F+1$ on some $B(\overline{0},r_{\infty})\times (0,t_{\infty})$.

\end{proof}

\begin{lemma}
\label{triangles}
Choose a sequence of error terms $\gamma_{m}>0$ with $lim_{m\rightarrow\infty}\gamma_{m}=0$, and pairs $(\overline{E}_{n_{m},\infty},\overline{B}_{n_{m},\infty})$, satisfying the conclusion of Lemma \ref{approxinfinity}. Then, as $m\rightarrow\infty$ the sequence $(\overline{E}_{n_{m},\infty},\overline{B}_{n_{m},\infty})$ converges to a pair $(\overline{E}_{\infty},\overline{B}_{\infty})$ satisfying Maxwell's equations in $S_{\infty}$ for the transferred charge and current $(\rho_{\infty},\overline{J}_{\infty})$ with $div_{\infty}(\overline{E}_{\infty}\times\overline{B}_{\infty})=0$. Moreover, we obtain the conclusion of Lemma \ref{solution} for the transfer $(\overline{E},\overline{B})$ back to the base frame $S$.
\end{lemma}

\begin{proof}
By the construction of Lemma \ref{approxinfinity}, we have that the sequence $(\overline{E}_{n_{m},\infty},\overline{B}_{n_{m},\infty})$ is Cauchy and uniformly bounded,
and converges to a bounded limit on $S_{\infty}$. If $(\rho',\overline{J}',\overline{E}',\overline{B})'$ is a tuple, satisfying Maxwell's equations in a base frame $S'$, then, by the proof in \cite{L}, $(\rho'',\overline{J}'',\overline{E}'',\overline{B}'')$ satisfies Maxwell's equations at corresponding points of $S''$, connected to $S$ by a real velocity vector $\overline{v}$, with $|\overline{v}|<c$, $(*)$. By the generalisation of the rules for transforming derivatives, the algebraic formulation of the connecting relations at corresponding points, complex linearity of the transformed derivative, and the generic formulation of $(*)$,  $(\rho''',\overline{J}''',\overline{E}''',\overline{B}''')$ satisfies Maxwell's equations at corresponding points of $S'''$, connected to $S$ by a complex velocity vector $\overline{v}$, with $\overline{v}^{2}\neq c^{2}$. Taking limits, this also holds for a transformation to a limit frame $S_{\infty}$, if the limit exists. It follows that the transformation of the fields $(\overline{E}_{\epsilon,\delta},\overline{B}_{\epsilon,\delta})$ to fields $(\overline{E}_{\epsilon,\delta,\infty},\overline{B}_{\epsilon,\delta,\infty})$ in the limit frame $S_{\infty}$ satisfy Maxwell's equations. and so the transformations $(\overline{E}_{n, \epsilon,\delta',\infty},\overline{B}_{n,\epsilon,\delta',\infty})$ to the limit frame $S_{\infty}$ satisfy Maxwell's equations up to a constant $\epsilon(n)$, which converges to $0$ as $n\rightarrow\infty$, so that $(\overline{E}_{n_{m},\infty},\overline{B}_{n_{m},\infty})$ satisfy Maxwell's equations up to a constant
$\epsilon'(n_{m})$, which again converges to zero as $m\rightarrow\infty$, both in $S_{\infty}$ and the frame $S_{\epsilon,\delta'}$, for sufficiently small $\{\epsilon,\delta'\}$. In particularly, $(\overline{E}_{\infty},\overline{B}_{\infty})$ satisfying Maxwell's equations in $S_{\infty}$ for the transferred charge and current $(\rho_{\infty},\overline{J}_{\infty})$, and, so does the transfer $(\overline{E},\overline{B})$ of $(\overline{E}_{\infty},\overline{B}_{\infty})$ back to the base frame $S$, for the original current and charge $(\rho,\overline{J})$. The claim that $div_{\infty}(\overline{E}_{\infty}\times\overline{B}_{\infty})=0$ follows from the transformation rules for derivatives, back to the frames $S_{\epsilon,\delta'}$, the fact that the polynomial approximations $(\overline{E}_{n, \epsilon,\delta'},\overline{B}_{n,\epsilon,\delta'})$ converge to  smooth fields $((\overline{E}_{\epsilon,\delta'},\overline{B}_{\epsilon,\delta'}))$ in the frame $S_{\epsilon,\delta'}$, interchanging limits with derivatives in $S_{\epsilon,\delta'}$ and the construction that $|div_{\infty}(\overline{E}_{n_{m},\infty}\times \overline{B}_{n_{m},\infty})|<\gamma_{m}$, with $\gamma_{m}\rightarrow 0$, as $m\rightarrow\infty$. For the final claim, we can approximate the fields $(\overline{E}_{\infty},\overline{B}_{\infty})$ by the polynomial fields  $(\overline{E}_{n_{m},\infty},\overline{B}_{n_{m},\infty})$, in $S_{\infty}$, and follow through the argument of Lemma \ref{infinity1}, to obtain the conclusion of Lemma \ref{solution} for the fields $(\overline{E}_{n_{m}},\overline{B}_{n_{m}})$ transferred back to the base frame $S$, up to a constant $\epsilon''(n_{m})$, which converges to $0$ as $m\rightarrow\infty$. As the fields $(\overline{E}_{n_{m}},\overline{B}_{n_{m}})$ and their derivatives converge to $(\overline{E},\overline{B)}$ in the base frame $S$, we obtain the conclusion of Lemma \ref{solution} for $(\overline{E},\overline{B})$.
\end{proof}

\begin{rmk}
\label{strongercondition}
In the definition of surface non-radiating, for the frame $S$, with charge and current $(\rho,\overline{J})$, we can impose the additional requirement, that, for any given velocity $\overline{v}$, with $|\overline{v}|=1$, there exist pairs $\{\overline{E}_{s\overline{v}},\overline{B}_{s\overline{v}}\}$, with $0\leq s<c$, such that the condition of surface non-radiating is fulfilled, and the series is smooth and decaying at infinity, that is the fields $\{\overline{E}_{s\overline{v}},\overline{B}_{s\overline{v}}\}$ are smooth and ;\\

$(i)$. $lim_{|\overline{x}_{s\overline{v}}|\rightarrow\infty}max(|\overline{E}_{s\overline{v}}|,|\overline{B}_{s\overline{v}}|)=0$\\

in the coordinates $(\overline{x}_{s\overline{v}},t_{s\overline{v}})$ of the frame $S_{s\overline{v}}$. With these extra assumptions, we conjecture, using polynomial approximations, that it possible to choose $\{\overline{E}_{s\overline{v}},\overline{B}_{s\overline{v}}\}$ such that the above conditions hold, and also;\\

$(ii)$ The families defined by;\\

$E(\overline{x},t,s)=\overline{E}_{s\overline{v}}'(\overline{x},t)$, $B(\overline{x},t,s)=\overline{B}_{s\overline{v}}'(\overline{x},t)$\\

are smooth on $\mathcal{R}^{3}\times\mathcal{R}_{\geq 0}\times (0,c)$. In particularly, for finite  $\{t_{0},r_{0},c_{0}\}\subset\mathcal{R}$, with $0\leq c_{0}<c$, $max_{0\leq s\leq c_{0}}(|\overline{E}_{s\overline{v}}'|_{B(\overline{0},r_{0})\times (0,t_{0})}|,|\overline{B}_{s\overline{v}}'|_{B(\overline{0},r_{0})\times (0,t_{0})}|)\leq G_{r,t_{0},c_{0}}$ for some constant $G_{r_{0},t_{0},c_{0}}\in\mathcal{R}_{\geq 0}$, where $\{\overline{E}_{s\overline{v}},\overline{B}_{s\overline{v}}\}$ are the fields transferred back to the base frame $S$.

\end{rmk}

\begin{lemma}
\label{analytic}
If the frame $S$ is decaying surface non-radiating, in the sense of Definition \ref{strongly}, then if $(\rho,\overline{J})$ is real analytic, either $\rho=0$ and $\overline{J}=\overline{0}$, or $S$ is non-radiating, in the sense of \cite{dep1}.
\end{lemma}

\begin{proof}
The transfers of $(\rho,\overline{J})$ to any frame $S_{\overline{v}}$, connected to $S$ by a velocity vector $\overline{v}$ with $|\overline{v}|<c$ is also real analytic. By the proof of Lemma \ref{surface}, the frames $S_{\overline{v}}$ are also decaying surface non-radiating. By Lemma \ref{complex}, and using continuity, in each frame $S_{\overline{v}}$, either the transfers $\rho_{\overline{v}}$ are identically zero or there exists a real solution $(\overline{E}_{\overline{v}},\overline{B}_{\overline{v}})$ to Maxwell's equations with $\overline{B}_{\overline{v}}=0$. If $S$ is not non-radiating, then, without loss of generality, we can assume that $\rho=0$ in the base frame $S$. By the transformation rules for $(\rho,\overline{J})$, we have that;\\

$\rho_{\overline{v}}=-{\gamma_{v}<\overline{v},\overline{J}>\over c^{2}}$\\

$\overline{J}_{\overline{v}}=\gamma_{v}\overline{J}_{||,\overline{v}}+\overline{J}_{\perp,\overline{v}}$\\

As $\overline{J}$ is analytic, $\{\overline{v}:<\overline{v},\overline{J}>=0\}$ includes $\overline{0}$ and, if infinite, is both open and closed inside the ball $B_{|\overline{v}|<c}$, so that $\overline{J}_{||,\overline{v}}=0$, for every $\overline{v}\in B_{|\overline{v}|<c}$ with $\overline{v}\neq 0$, in particular, $\overline{J}=0$. We can, therefore assume that $\rho_{\overline{v}}\neq 0$ in all but finitely many frames $S_{\overline{v}}$, and there exist real solutions $(\overline{E}_{\overline{v}},\overline{B}_{\overline{v}})$ to Maxwell's equations with $\overline{B}_{\overline{v}}=0$. Now, we can use the proof of Lemma 2.7 in \cite{dep1}, to derive the equation;\\

$\overline{v}\times (\bigtriangledown(\rho)+{1\over c^{2}}{\partial \overline{J}\over \partial t})=\overline{0}$ $(*)$\\

valid for all but finitely many $\overline{v}\in B_{|\overline{v}|<c}$. Using continuity, we can conclude that $(*)$ holds for all $\overline{v}\in B_{|\overline{v}|<c}$ and that;\\

$(\bigtriangledown(\rho)+{1\over c^{2}}{\partial \overline{J}\over \partial t})=\overline{0}$\\

Then follow through the rest of the proof of Lemma 2.7 in \cite{dep1} to conclude that $\square^{2}(p)=0$ and $\square^{2}(\overline{E})=\overline{0}$. Now use the proof of Lemma 2.4 in \cite{dep1} to get $\square^{2}(\overline{J})=\overline{0}$, and Lemma 2.5 in \cite{dep1} to conclude that $S$ is non-radiating.

\end{proof}

\end{section}

\begin{section}{Some Thermodynamic Arguments}

\begin{defn}
\label{reversal}
Given $(\rho,\overline{J},\overline{E},\overline{B})$ satisfying Maxwell's equations, and $t_{0}\in\mathcal{R}_{>0}$, we define the reversed process $(\rho',\overline{J}',\overline{E}',\overline{B}')$ on $\mathcal{R}^{3}\times (0,t_{0})$ by;\\

$\rho'(\overline{x},t)=\rho(\overline{x},t_{0}-t)$\\

$\overline{J}'(\overline{x},t)=-\overline{J}(\overline{x},t_{0}-t)$\\

$\overline{E}'(\overline{x},t)=\overline{E}(\overline{x},t_{0}-t)$\\

$\overline{B}'(\overline{x},t)=-\overline{B}(\overline{x},t_{0}-t)$\\

\end{defn}

\begin{lemma}
\label{maxwellsagain}
For the reversed process, $(\rho',\overline{J}',\overline{E}',\overline{B}')$, we have that $(\rho',\overline{J}')$ satisfies the continuity equation and $(\rho',\overline{J}',\overline{E}',\overline{B}')$ satisfies Maxwell's equations on $\mathcal{R}^{3}\times (0,t_{0})$. Moreover $div(\overline{E}'\times\overline{B}')=-div(\overline{E}\times\overline{B})$

\end{lemma}
\begin{proof}
For the first part, we have, using the chain rule, the definitions and the continuity equation for $(\rho,\overline{J})$, that;\\

${\partial \rho'\over \partial t}|_(\overline{x},t)=-{\partial \rho\over \partial t}|_(\overline{x},t_{0}-t)$\\

$=-div(\overline{J})|_{(\overline{x},t_{0}-t)}$\\

$=div(\overline{J}')_{(\overline{x},t)}$\\

For the second part, we have, using the chain rule again, the definitions, and Maxwell's equations for $(\rho,\overline{J},\overline{E},\overline{B})$, that;\\

$(i)$. $div(\overline{E}')|_{(\overline{x},t)}=div(\overline{E})|_{(\overline{x},t_{0}-t)}={\rho\over\epsilon_{0}}|_{(\overline{x},t_{0}-t)}={\rho\over\epsilon_{0}}|_{(\overline{x},t)}$\\

$(ii)$. $(\bigtriangledown\times \overline{E}')|_{(\overline{x},t)}=(\bigtriangledown\times \overline{E})|_{(\overline{x},t_{0}-t)}=-{\partial \overline{B}\over \partial t}|_{(\overline{x},t_{0}-t)}=-{\partial \overline{B}'\over \partial t}|_{(\overline{x},t)}$\\

$(iii)$. $div(\overline{B}')|_{(\overline{x},t)}=-div(\overline{B})_{(\overline{x},t_{0}-t)}=0$\\

$(iv)$. $(\bigtriangledown\times \overline{B}')_{(\overline{x},t)}=(\bigtriangledown\times -\overline{B})_{(\overline{x},t_{0}-t)}=-(\epsilon_{0}\overline{J})|_{(\overline{x},t_{0}-t)}-(\mu_{0}\epsilon_{0}{\partial \overline{E}\over \partial t})|_{(\overline{x},t_{0}-t)}$\\

$=(\epsilon_{0}\overline{J}')|_{(\overline{x},t)}+(\mu_{0}\epsilon_{0}{\partial \overline{B}'\over \partial t})|_{(\overline{x},t)}$\\

as required. The last claim follows easily from the definitions of $\{\overline{E}',\overline{B}'\}$\\

\end{proof}

\begin{defn}
Given a solution $(\rho,\overline{J},\overline{E},\overline{B})$ to Maxwell's equation, we say that $(\overline{E},\overline{B})$ is classically non-radiating if, uniformly in $t\in\mathcal{R}_{>0}$, we have that;\\

$lim_{r\rightarrow\infty}\int_{B(0,r)}div(\overline{E}_{t}\times\overline{B}_{t})d\overline{x}=0$\\

\end{defn}

\begin{lemma}
\label{surfaceradiating}
Given a smooth solution $(\rho,\overline{J},\overline{E},\overline{B})$ to Maxwell's equations, and $t_{0}\in\mathcal{R}_{>0}$ with $div(\overline{E}\times\overline{B})|_{t_{0}}\neq 0$ and $(\overline{E},\overline{J})|_{t_{0}}\neq 0$, there exists a smooth volume $S\subset\mathcal{R}^{3}$ and $\epsilon>0$, with;\\

$\int_{S}div(\overline{E}_{t}\times\overline{B}_{t})|d\overline{x}\neq 0$\\

$\int_{S}(\overline{E}_{t},\overline{J}_{t})d\overline{x}\neq 0$\\

for $t\in (t_{0}-\epsilon,t_{0}+\epsilon)$.\\

 \end{lemma}

\begin{proof}
Choose $\{\overline{x}_{0},\overline{x}_{1}\}\subset\mathcal{R}^{3}$, with $div(\overline{E}\times\overline{B})(\overline{x}_{0},t_{0})\neq 0$ and $(\overline{E},\overline{J})(\overline{x}_{1},t_{0})\neq 0$. As $div(\overline{E}\times\overline{B})$ and $(\overline{E},\overline{J})$ are smooth, there exist disjoint balls $B(\overline{x}_{0},r_{0})$ and $B(\overline{x}_{0},r_{1})$ with;\\

$\int_{B(\overline{x}_{0},r_{0})}div(\overline{E}_{t_{0}}\times\overline{B}_{t_{0}}))d\overline{x}\neq 0$\\

$\int_{B(\overline{x}_{1},r_{1})}(\overline{E}_{t_{0}},\overline{J}_{t_{0}})d\overline{x}\neq 0$\\

Shrinking the ball $B(\overline{x}_{0},r_{0})$ if necessary to avoid cancellations, we can assume that;\\

$\int_{B(\overline{x}_{0},r_{0})\cup B(\overline{x}_{1},r_{1})}div(\overline{E}_{t_{0}}\times\overline{B}_{t_{0}}))d\overline{x}\neq 0$\\

$\int_{B(\overline{x}_{0},r_{0})\cup B(\overline{x}_{1},r_{1})}(\overline{E}_{t_{0}},\overline{J}_{t_{0}})d\overline{x}\neq 0$\\

As $div(\overline{E}_{t_{0}}\times\overline{B}_{t_{0}})$ and $(\overline{E}_{t_{0}},\overline{J}_{t_{0}})$ are smooth, they are bounded on a ball $B(0,r)$ with $B(0,r)\supset B(\overline{x}_{0},r_{0})$ and $B(0,r)\supset B(\overline{x}_{1},r_{1})$. Choosing a sufficiently small strip $S'$ connecting the balls $B(\overline{x}_{0},r_{0})$ and $B(\overline{x}_{1},r_{1})$, and letting $S=B(\overline{x}_{0},r_{0})\cup B(\overline{x}_{1},r_{1})\cup S'$ be a smooth volume, we can assume that;\\

$\int_{S}div(\overline{E}_{t_{0}}\times\overline{B}_{t_{0}}))d\overline{x}\neq 0$\\

$\int_{S}(\overline{E}_{t_{0}},\overline{J}_{t_{0}})d\overline{x}\neq 0$\\

Using smoothness again, we can assume that there exists $\epsilon>0$ such that;\\

$\int_{S}div(\overline{E}_{t}\times\overline{B}_{t}d\overline{x}\neq 0$\\

$\int_{S}(\overline{E}_{t},\overline{J}_{t})d\overline{x}\neq 0$\\

for $t\in (t_{0}-\epsilon,t_{0}+\epsilon)$, as required.\\

\end{proof}

\begin{lemma}
\label{equiflux}
Given $(\rho,\overline{J},\overline{E},\overline{B})$, with $(\overline{E},\overline{B})$ classically non-radiating, such that the hypotheses of Lemma \ref{surfaceradiating} are satisfied. Let $T$ be the surface boundary of $S$, then for any $\kappa>0$, there exists volumes $\{S,S_{\kappa}\}$ with $S\cap S_{\kappa}=\emptyset$, $T\subset \overline{S_{\kappa}}$, $\{\overline{T},\overline{T}',\overline{T}''\}$ outward normals to the volumes $\{S,S_{\kappa},B(\overline{0},r_{\kappa})\}$, such that;\\

$\int_{T}(\overline{E}_{t}\times\overline{B}_{t})\centerdot d\overline{T}=
-\int_{T}(\overline{E}_{t}\times\overline{B}_{t})\centerdot d\overline{T'}$\\

$\int_{S}div(\overline{E}_{t}\times\overline{B}_{t})d\overline{x}\neq 0$\\

$\int_{S}(\overline{E}_{t},\overline{J}_{t})d\overline{x}\neq 0$\\

$|\int_{\delta B(\overline{0},r_{\delta})}(\overline{E}_{t}\times\overline{B}_{t})\centerdot d\overline{T}''|<\kappa$\\

for $t\in (t_{0}-\epsilon,t_{0}+\epsilon)$\\

\end{lemma}

\begin{proof}
By the definition of classically non-radiating, for any $\delta>0$, there exists $r_{d}>0$, with $S\subset B(\overline{0},r_{d})$ such that;\\

$|\int_{B(\overline{0},r_{d})}div(\overline{E}_{t}\times\overline{B}_{t})d\overline{x}|<\delta$ $(*)$\\

for $t\in (t_{0}-\epsilon,t_{0}+\epsilon)$. Let $S_{\delta}={B^{0}(\overline{0},r_{d})\setminus \overline{S}}$, then, as $\overline{T}'$ reverses the direction of $\overline{T}$ ;\\

$\int_{T}(\overline{E}_{t}\times\overline{B}_{t})\centerdot d\overline{T}
+\int_{T}(\overline{E}_{t}\times\overline{B}_{t})\centerdot d\overline{T'}=0$\\

and by $(*)$ and the divergence theorem;\\

$|\int_{\delta B(\overline{0},r_{d})}(\overline{E}_{t}\times\overline{B}_{t})\centerdot d\overline{T}''|=|\int_{B(\overline{0},r_{d})}div(\overline{E}_{t}\times\overline{B}_{t})d\overline{x}|<\kappa$\\

as required.\\
\end{proof}

\begin{lemma}
\label{equilibrium}
Let notation be as above, then, assuming thermal equilibrium for electrons, we cannot have, in a classically non radiating system, that;\\

$\int_{S}div(\overline{E}_{t}\times\overline{B}_{t})d\overline{x}>0$\\

$\int_{S}(\overline{E}_{t},\overline{J}_{t})d\overline{x}<0$ $(\dag)$\\

for $t\in (t_{0}-\epsilon,t_{0}+\epsilon)$.\\

\end{lemma}

\begin{proof}

Suppose that;\\

$\int_{S}div(\overline{E}_{t}\times\overline{B}_{t})d\overline{x}>0$\\

$\int_{S}(\overline{E}_{t},\overline{J}_{t})d\overline{x}<0$ $(\dag)$\\

for $t\in (t_{0}-\epsilon,t_{0}+\epsilon)$.\\

Using Lemmas \ref{surfaceradiating} and \ref{equiflux}, we claim that;\\

$\int_{S_{\kappa}}(\overline{E}_{t},\overline{J}_{t})d\overline{x}\geq min(-\int_{S}(\overline{E}_{t},\overline{J}_{t})d\overline{x},
\int_{S}div(\overline{E}_{t}\times\overline{B}_{t})d\overline{x})$ $(\dag\dag)$\\

Suppose not, then;\\

$\int_{S_{\kappa}}(\overline{E}_{t},\overline{J}_{t})d\overline{x}
<-\int_{S}(\overline{E}_{t},\overline{J}_{t})d\overline{x}$\\

$\int_{S_{\kappa}}(\overline{E}_{t},\overline{J}_{t})d\overline{x}
<\int_{S}div(\overline{E}_{t}\times\overline{B}_{t})d\overline{x}$\\

Let $\{E_{S_{\kappa}},E_{S},E_{S,el},E_{S_{\kappa},el},E_{S,field},E_{S_{\kappa},field}\}$ denote the total energies in $\{S,S_{\kappa}\}$, the energies stored in the electrons contained in $\{S,S_{\kappa}\}$ and the electromagnetic energies restricted to $\{S,S_{\kappa}\}$, see \cite{G}. By Poynting's theorem and the divergence theorem;\\

$(i)$ ${d E_{S_{\kappa},el}\over dt'}|_{t}<-{d E_{S,el}\over dt'}|_{t}$\\

$(ii)$ ${d E_{S_{\kappa},el}\over dt'}|_{t}<(\int_{S}(\overline{E}_{t}\times\overline{B}_{t})\centerdot d\overline{T})_{t}$\\

By $(\dag)$, there is an energy flux from the electrons in $S$ to the total energies in $S_{\kappa}$. By $(i)$ and Poynting's Theorem, there is some energy transferred into the electromagnetic energy. By $(ii)$, not all the energy transferred from the electrons in $S$ are transferred to the energy of the electrons in $S_{\kappa}$. It follows that some of the energy flux from the electrons in $S$ is transferred into the electromagnetic energy of $S_{\kappa}$. By the last claim in Lemma \ref{maxwellsagain}, the process is reversible, which contradicts Kelvin's formulation of the second law of thermodynamics, see \cite{P}, that it is impossible to devise an engine which, working in a cycle, shall produce no effect other than the extraction of heat from a reservoir and the performance of an equal amount of mechanical work. Given that $(\dag\dag)$ holds, we have;\\

$\int_{S_{\kappa}}(\overline{E}_{t},\overline{J}_{t})d\overline{x}
\geq -\int_{S}(\overline{E}_{t},\overline{J}_{t})d\overline{x}$\\

$\int_{S_{\kappa}}(\overline{E}_{t},\overline{J}_{t})d\overline{x}
\geq \int_{S}div(\overline{E}_{t}\times\overline{B}_{t})d\overline{x}$\\

so that;\\

$(iii)$ ${d E_{S_{\kappa},el}\over dt'}|_{t}\geq -{d E_{S,el}\over dt'}|_{t}$\\

$(iv)$ ${d E_{S_{\kappa},el}\over dt'}|_{t}\geq (\int_{S}(\overline{E}_{t}\times\overline{B}_{t})\centerdot d\overline{T})_{t}$\\

By $(iii)$, and $(\dag)$, we have that $E_{S_{\kappa},el}$ is increasing, by $(iv)$, the energy in the field from $S_{\kappa}$ is constant or decreasing and transferring to the electrons in $S_{\kappa}$. As there is a net energy flux from $S$ to $S_{\kappa}$, and using $(\dag)$ again, there is an energy transfer from electrons in $S$ to electrons in $S_{\kappa}$. Assuming thermal equilibrium and raising the temperature of the electrons in $S_{\kappa}$ by a small amount, see \cite{P}, noting again that the process is reversible, this contradicts Clausius's formulation of the second law of thermodynamics; that it is impossible to devise an engine which, working in a cycle, shall produce no effect other than the transfer of heat from a colder to a hotter body.

\end{proof}

\begin{rmk}
\label{systems}
We can assume that an atomic system is classically non-radiating in all inertial frames, as, otherwise, by Rutherford's observation, the system would lose energy and collapse. Thermal equilibrium for electrons in all frames also seems a reasonable criterion for such systems, though there are difficulties in finding the correct definition of temperature in electromagnetism. We leave as a conjecture whether the above lemma holds with just the assumption that;\\

$\int_{S}div(\overline{E}_{t}\times\overline{B}_{t})d\overline{x}\neq 0$\\

$\int_{S}(\overline{E}_{t},\overline{J}_{t})d\overline{x}\neq 0$\\

for $t\in (t_{0}-\epsilon,t_{0}+\epsilon)$.\\ 

Given this, we can probably obtain the conclusion, from Lemma \ref{surfaceradiating}, that either;\\

 $div(\overline{E}_{t}\times\overline{B}_{t})=0$\\
 
in all frames, or;\\
 
 $(\overline{E}_{t},\overline{J}_{t})=0$\\
 
in all frames. In which case, we can either use the main result of the paper to conclude that the system is non-radiating and, by \cite{dep1}, that the charge and current ${\rho,\overline{J}}$ obey certain wave equations, or classify the case where $(\overline{E},\overline{J})=0$ in all frames. We leave this point of view to another paper.

\end{rmk}

\end{section}


\begin{thebibliography}{99}
\bibitem{L} Electromagnetic Fields and Waves, Dale Corson, Francoise Lorrain, Paul Lorrain, Freeman and Company, (1988).\\
\bibitem{F} The Meaning of Rotation in the Special Theory of Relativity, Philip Franklin, Proceedings of the National Academy of Sciences, Volume 8, Number 9, (1922).\\
\bibitem{G} Introduction to Electrodynamics, David Griffiths, Pearson, (2008).\\
\bibitem{J} Fun with Fields, William Johnson, PhD thesis, University of California, Berkeley, (2016).\\
\bibitem{P} The Elements of Classical Thermodynamics, A.B.Pippard, CUP, (1957).\\
\bibitem{dep1} Some Arguments for the Wave Equation in Quantum Theory, Tristram de Piro, Open Research Journal of Mathematical Sciences, (2021).\\
\bibitem{R} Introduction to Special Relativity, Wolfgang Rindler, Oxford Science Publications, (1991).\\
\bibitem{U1} The Relativistic Velocity Composition Paradox and the Thomas Rotation, Abraham Ungar, Foundations of Physics, Vol 19, No.11, (1989).\\
\bibitem{U2} Thomas Rotation and the Parametrization of the Lorentz Transformation Group, Abraham Ungar, Foundations of Physics Letters, (1988).\\
\bibitem{Z} Zariski Geometries, Geometry from the Logician's Point of View, Boris Zilber, London Mathematical Society, (2010).\\
\end{thebibliography}
\end{document}